\documentclass[11pt,letterpaper]{amsart}
\usepackage{hyperref}
\usepackage{graphicx, color}
\usepackage{amssymb,mathtools}
\usepackage{enumerate}
\usepackage[margin=1in]{geometry}
\usepackage[normalem]{ulem}

\usepackage{xcolor}
\hypersetup{
    colorlinks,
    linkcolor={red!50!black},
    citecolor={blue!50!black},
    urlcolor={blue!80!black}
}

\newtheorem{theorem}{Theorem}[section]
\newtheorem{proposition}[theorem]{Proposition}
\newtheorem{lemma}[theorem]{Lemma}
\newtheorem{corollary}[theorem]{Corollary}
\theoremstyle{definition}

\theoremstyle{remark}

\newtheorem{conjecture}[theorem]{Conjecture}

\DeclareFontFamily{U} {MnSymbolC}{}
\DeclareFontShape{U}{MnSymbolC}{m}{n}{
  <-6> MnSymbolC5
  <6-7> MnSymbolC6
  <7-8> MnSymbolC7
  <8-9> MnSymbolC8
  <9-10> MnSymbolC9
  <10-12> MnSymbolC10
  <12-> MnSymbolC12}{}
\DeclareFontShape{U}{MnSymbolC}{b}{n}{
  <-6> MnSymbolC-Bold5
  <6-7> MnSymbolC-Bold6
  <7-8> MnSymbolC-Bold7
  <8-9> MnSymbolC-Bold8
  <9-10> MnSymbolC-Bold9
  <10-12> MnSymbolC-Bold10
  <12-> MnSymbolC-Bold12}{}

\DeclareSymbolFont{MnSyC} {U} {MnSymbolC}{m}{n}

\DeclareMathSymbol{\minus}{\mathrel}{MnSyC}{16}
\DeclareMathSymbol{\plus}{\mathrel}{MnSyC}{20}

\newcommand{\p}{\plus}
\newcommand{\m}{\minus}
\newcommand{\pp}{{\plus\plus}}

\DeclareFontFamily{U} {MnSymbolD}{}
\DeclareFontShape{U}{MnSymbolD}{m}{n}{
  <-6> MnSymbolD5
  <6-7> MnSymbolD6
  <7-8> MnSymbolD7
  <8-9> MnSymbolD8
  <9-10> MnSymbolD9
  <10-12> MnSymbolD10
  <12-> MnSymbolD12}{}
\DeclareFontShape{U}{MnSymbolD}{b}{n}{
  <-6> MnSymbolD-Bold5
  <6-7> MnSymbolD-Bold6
  <7-8> MnSymbolD-Bold7
  <8-9> MnSymbolD-Bold8
  <9-10> MnSymbolD-Bold9
  <10-12> MnSymbolD-Bold10
  <12-> MnSymbolD-Bold12}{}

\DeclareSymbolFont{MnSyD} {U} {MnSymbolD}{m}{n}

\DeclareMathSymbol{\eq}{\mathrel}{MnSyD}{0}
\DeclareMathSymbol{\leq}{\mathrel}{MnSyD}{66}

\newcommand{\mathbbm}[1]{\text{\usefont{U}{bbm}{m}{n}#1}}

\newcommand{\bbGamma}{{\mathpalette\makebbGamma\relax}}
\newcommand{\makebbGamma}[2]{%
  \raisebox{\depth}{\scalebox{1}[-1]{$\mathsurround=0pt#1\mathbb{L}$}}%
}

\let\Re\undefine

\DeclareMathOperator{\Re}{Re}

\DeclareMathOperator{\sign}{sgn}
\DeclareMathOperator{\tr}{tr}
\DeclareMathOperator{\rank}{rank}
\DeclareMathOperator{\diag}{diag}
\DeclareMathOperator{\spn}{span}
\DeclareMathOperator{\cone}{cone}
\DeclareMathOperator{\conv}{conv}
\DeclareMathOperator{\extr}{extr}

\DeclareMathOperator{\cut}{cut}
\DeclareMathOperator{\spr}{spread}
\DeclareMathOperator{\str}{stretch}
\DeclareMathOperator{\im}{im}
\DeclareMathOperator{\Beta}{B}

\newcommand{\Sph}{S}
\newcommand{\tp}{{\scriptscriptstyle\mathsf{T}}}
\newcommand{\G}{\mathsf{G}}
\newcommand{\GW}{\mathsf{GW}}
\newcommand{\DD}{\mathsf{dd}}

\begin{document}
\title{Symmetric Grothendieck inequality}
\author{Shmuel~Friedland}
\address{Department of Mathematics, Statistics and Computer Science,  University of Illinois, Chicago, IL,  60607-7045.}
\email{friedlan@uic.edu}
\author{Lek-Heng~Lim}
\address{Computational and Applied Mathematics Initiative, Department of Statistics,
University of Chicago, Chicago, IL 60637-1514.}
\email{lekheng@galton.uchicago.edu}
\dedicatory{In memory of Joram Lindenstrauss, who introduced the first author to the Grothendieck inequality.}

\begin{abstract}
We establish an analogue of the Grothendieck inequality where the rectangular matrix is replaced by a  symmetric (or Hermitian) matrix and the bilinear form by a quadratic form. We call this the symmetric Grothendieck inequality; despite its name, it is a generalization --- the original Grothendieck inequality is a special case. While there are other proposals for such an inequality, ours differs in two important ways: (i) we have no  additional requirement like positive semidefiniteness for the symmetric matrix; (ii) our symmetric Grothendieck constant is universal, i.e., independent of the matrix and its dimensions, like the original Grothendieck constant. A consequence of our symmetric Grothendieck inequality is a ``conic Grothendieck inequality'' for any family of cones of symmetric matrices: The original  Grothendieck inequality is a special case; as is the Nesterov $\pi/2$-Theorem, which corresponds to the cones of positive semidefinite matrices; as well as the  Goemans--Williamson inequality, which corresponds to the cones of weighted Laplacians. For yet other cones, e.g., of diagonally dominant matrices, we obtain new Grothendieck-like inequalities. With a slight extension, we obtain a unified framework that treats any Grothendieck-like inequality as an inequality between two norms within a family of ``Grothendieck norms'' restricted to an appropriate family of cones. This approach allows us to place on an equal footing the Goemans--Williamson inequality, Nesterov $\pi/2$-Theorem, Ben-Tal--Nemirovski--Roos $4/\pi$-Theorem, generalized Grothendieck inequality, order-$p$ Grothendieck inequality, rank-constrained positive semidefinite Grothendieck inequality, etc, and in turn allows us to simplify proofs, extend results from real to complex, obtain new bounds or establish sharpness of existing ones. Unsurprisingly, the symmetric Grothendieck inequality may also be applied  to obtain  uniform polynomial-time approximation bounds for various  NP-hard combinatorial, integer, and nonconvex optimization problems.
\end{abstract}


\maketitle

\section{Introduction}\label{sec:intro}

The Grothendieck inequality \cite{Grothendieck} states that for $\Bbbk = \mathbb{R}$ or $\mathbb{C}$, there is a finite  constant $K_\G > 0$ such that for every $d,m,n\in\mathbb{N}$ with $d \ge m+n$ and every matrix $B=(b_{ij})\in\Bbbk^{m\times n}$,
\begin{equation}\label{eq:GI}
\max_{\lVert x_i\rVert \le 1,\; \lVert y_j \rVert \le 1}\biggl\lvert \sum_{i=1}^m\sum_{j=1}^n b_{ij} \langle  x_i,y_j\rangle\biggr\rvert\leq K_\G \max_{\lvert \varepsilon_i \rvert \le 1,\; \lvert \delta_j \rvert \le 1}\biggl\lvert\sum_{i=1}^m\sum_{j=1}^n b_{ij}  \bar{\varepsilon}_i\delta_j\biggr\rvert,
\end{equation}
where the maximum on the left is taken over all $x_i,y_j$ in the  unit ball of $\Bbbk^d$, and the maximum on the right is taken over all  $\delta_i, \varepsilon_j$ in the unit disk of  $\Bbbk$, $i =1,\dots,m$, $j=1,\dots,n$.  The original version in \cite{Grothendieck}, while stated differently, is equivalent to the version in \eqref{eq:GI}, which is due to \cite{Lindenstrauss}. 

The maxima on the left- and right-hand sides of \eqref{eq:GI} define matrix norms
\begin{equation}\label{eq:norms}
\lVert B\rVert_\G\coloneqq \max_{\lVert x_i\rVert \le 1,\; \lVert y_j \rVert \le 1}\biggl\lvert \sum_{i=1}^m\sum_{j=1}^n b_{ij} \langle  x_i,y_j\rangle\biggr\rvert \quad \text{and}\quad
\lVert B\rVert_{\infty,1}\coloneqq \max_{\lvert \varepsilon_i \rvert \le 1, \; \lvert \delta_j \rvert \le 1}\biggl\lvert\sum_{i=1}^m\sum_{j=1}^n b_{ij}  \bar{\varepsilon}_i\delta_j\biggr\rvert
\end{equation}
respectively. So we have
\begin{equation}\label{eq:GI2}
\lVert B \rVert_{\infty,1} \le \lVert B \rVert_\G \le K_\G \lVert B \rVert_{\infty,1},
\end{equation}
where the first inequality is easy to see and the second inequality is the Grothendieck inequality \eqref{eq:GI}. By the equivalence of norms in finite-dimensional spaces, an inequality of the form \eqref{eq:GI2} must exist but the surprise here is that the constant $K_\G$ is independent of the dimensions $m,n$! The existence of such a universal constant $K_\G$ is the crux of the Grothendieck inequality and we expect any extension to also have such a property.

The two norms in \eqref{eq:norms} as well as the value of the smallest possible $K_\G$ in \eqref{eq:GI2} depend on the choice of the field $\Bbbk$ and we will denote them by $\lVert\, \cdot\, \rVert_\G^\Bbbk$,  $\lVert\, \cdot\, \rVert_{\infty,1}^\Bbbk$, $K_\G^\Bbbk$ when we need to emphasize this dependence. The sharpest constants $K_\G^\mathbb{R}$ and  $K_\G^\mathbb{C}$  are known respectively as the real and complex \emph{Grothendieck constants}. Their exact values are unknown but there are excellent bounds due to Davie \cite{Davie, Davie2}, Haagerup \cite{Haagerup}, and Krivine \cite{Krivine2},
\begin{equation}\label{eq:KH}
1.67696 \le K_\G^\mathbb{R} \le 1.78221, \qquad  1.33807 \le K_\G^\mathbb{C} \le 1.40491.
\end{equation}
That the Grothendieck constants have moderately small values has important implications. The most direct one is that by virtue of \eqref{eq:GI2}, the $\G$-norm, which can be computed in polynomial-time via semidefinite programming, gives us a good approximation of the $(\infty,1)$-norm, which is known to be NP-hard.

\subsection{Symmetric Grothendieck inequality}

The main goal of our article is to establish a generalization. Consider the following norms on $\mathbb{S}^n$, the vector space of $n \times n$ symmetric/Hermitian matrices over $\Bbbk$:
\begin{equation}\label{eq:GammaTheta}
\lVert A \rVert_\Gamma \coloneqq \max_{\lVert x_i\rVert \le 1}\; \biggl\lvert\sum_{i=1}^n \sum_{j=1}^n a_{ij}\langle x_i, x_j\rangle \biggr\rvert \qquad\text{and}\qquad
\lVert A \rVert_\Theta \coloneqq \max_{ \lvert \delta_i \rvert \le 1}\; \biggl\lvert\sum_{i=1}^n\sum_{j=1}^n a_{ij}  \bar{\delta}_i\delta_j\biggr\rvert.
\end{equation}
As in the Grothendieck inequality,  the maximum in $\lVert\, \cdot\, \rVert_\Gamma$ is over all $x_1,\dots, x_n$ in the unit ball of $\Bbbk^d$ while  the maximum in $\lVert\, \cdot\, \rVert_\Theta$ is over all $\delta_1,\dots, \delta_n$ in the unit disk of $\Bbbk$, and where we have assumed $d \ge n$. These norms are distinct from those in \eqref{eq:norms}, with different values even on $2\times 2$ diagonal matrices.

In Section~\ref{sec:bounds}, we will establish an analogue of \eqref{eq:GI2},
\begin{equation}\label{eq:GGI2}
\lVert A \rVert_\Theta \le  \lVert A \rVert_\Gamma \le K_\Gamma \lVert A \rVert_\Theta.
\end{equation}
Again the first inequality is obvious, and the effort is in showing that there exists a universal constant $K_\Gamma$, independent of $d$ and $n$, such that the second inequality
\begin{equation}\label{eq:GGI}
\max_{\lVert x_i\rVert \le 1}\; \biggl\lvert\sum_{i=1}^n \sum_{j=1}^n a_{ij}\langle x_i, x_j\rangle \biggr\rvert \le K_\Gamma \max_{ \lvert \delta_i \rvert \le 1}\; \biggl\lvert\sum_{i=1}^n\sum_{j=1}^n a_{ij}  \bar{\delta}_i\delta_j\biggr\rvert
\end{equation}
holds for all $A \in \mathbb{S}^n$. We will call \eqref{eq:GGI} the \emph{symmetric Grothendieck inequality} and $K_\Gamma$, the smallest  constant such that \eqref{eq:GGI2} holds, the \emph{symmetric Grothendieck constant}. While there are other inequalities in the literature that go by the same name, \eqref{eq:GGI} is different, as we will explain next.

Firstly, the symmetric Grothendieck inequality \eqref{eq:GGI} generalizes  the Grothendieck inequality: for any $B \in \Bbbk^{m \times n}$, setting
\begin{equation}\label{eq:symm}
A = \begin{bmatrix}  0 & B \\ B^* & 0 \end{bmatrix} \in \mathbb{S}^{m+n}
\end{equation}
recovers \eqref{eq:GI} from \eqref{eq:GGI}. Secondly, we emphasize that the symmetric Grothendieck inequality is  not a case of the Grothendieck inequality restricted to symmetric matrices --- we will see in Corollary~\ref{cor:ne} that $\lVert\,\cdot\,\rVert_\Theta \ne \lVert\,\cdot\,\rVert_{\infty, 1}$ and $\lVert\,\cdot\,\rVert_\Gamma \ne \lVert\,\cdot\,\rVert_\G$ on $\mathbb{S}^n$. Readers who think that they might have seen \eqref{eq:GGI} before (say, in \cite{Charikar,Khot,Rietz}) should note that a key difference here is that the matrix $A \in \mathbb{S}^n$ is not assumed to be positive semidefinite; in particular, the absolute values on both sides of \eqref{eq:GGI} cannot be dropped. If $A \in \mathbb{S}^n$ is in addition positive semidefinite,  then we will see in Proposition~\ref{posdefcase} that
\begin{equation}\label{eq:equalnorms}
\lVert A \rVert_\Theta = \lVert A \rVert_{\infty, 1} \qquad \text{and}\qquad \lVert A \rVert_\Gamma = \lVert A \rVert_\G,
\end{equation}
and both the symmetric and original Grothendieck inequalities become weaker versions of the Nesterov $\pi/2$-Theorem \cite{Nes, Rietz}, sometimes also called the positive semidefinite Grothendieck inequality. Another key difference from other similarly-named inequalities is that $K_\Gamma$ is truly a universal constant independent of $n$, just like the original Grothendieck constant.

As in the case of Grothendieck inequality, the values of the norms  $\lVert\, \cdot\, \rVert_\Theta$, $\lVert\, \cdot\, \rVert_\Gamma$ and the constant $K_\Gamma$  will depend on whether $\Bbbk$ is $\mathbb{R}$ or $\mathbb{C}$. We will write   $\lVert\, \cdot\, \rVert_\Theta^\Bbbk$, $\lVert\, \cdot\, \rVert_\Gamma^\Bbbk$, $K_\Gamma^\Bbbk$ whenever it is necessary to emphasize the dependence.  We will deduce analogues of the upper bounds in \eqref{eq:KH} for the symmetric Grothendieck constants:
\[
K_\Gamma^\mathbb{R}\le \sinh \frac{\pi}{2}\approx 2.30130 \qquad \text{and}\qquad
K_\Gamma^\mathbb{C}\le \frac{8}{\pi}-1\approx 1.54648,
\]
and show that they are lower bounded by the original Grothendieck constants. We will also establish a closely related variant of \eqref{eq:GGI} where the maxima are taken over $\lVert x_i\rVert = 1$ and $\lvert \delta_i \rvert = 1$ respectively.

\subsection{Conic Grothendieck inequalities}\label{sec:CGI}

An immediate consequence of the symmetric Grothendieck inequality is that it gives a ``Grothendieck inequality for cones'' by restricting \eqref{eq:GGI2} to any family of cones $\mathsf{C} = \{C_n \subseteq \mathbb{S}^n : n \in \mathbb{N}\}$. As we will discuss in Section~\ref{sec:cones}, we obtain an inequality
\begin{equation}\label{eq:CGI}
\lVert A \rVert_\Theta \le  \lVert A \rVert_\Gamma \le K_\mathsf{C} \lVert A \rVert_\Theta
\end{equation}
that holds for all $A \in C_n$ with a universal constant $K_\mathsf{C} \le K_\Gamma$ independent of $n$. The original Grothendieck inequality corresponds to the special case where the cones are subspaces of matrices of the form \eqref{eq:symm}; but  other well-known inequalities are also special cases of \eqref{eq:CGI} --- the Goemans--Williamson inequality \cite{GW95} is the case when $C_n$ is the cone of weighted Laplacians on $n$-vertex graphs; the Nesterov $\pi/2$-Theorem \cite{Nes,Rietz} is the case when $C_n$ is the cone of real symmetric positive semidefinite $n \times n$ matrices. For other family of cones, say, taking $C_n$ to be the cone of diagonally dominant $n \times n$ matrices, we get yet other ``Grothendieck-like inequalities'' that are new as far as we know. The values of these conic  Grothendieck constants (assuming $\mathbb{R}$ for simplicity),
\begin{equation}\label{eq:CGconst}
K_\mathsf{C} \coloneqq
\sup_{n \in \mathbb{N}} \biggl[
\max_{A \in C_n \subseteq \mathbb{S}^n} \frac{\lVert A \rVert_\Gamma}{\lVert A \rVert_\Theta} \biggr]=
\begin{cases}
1/\alpha_\GW & \text{if }  C_n = \{ A \in \mathbb{S}^n :  A \mathbbm{1} = 0,\; a_{ij} \le 0 \; \text{for all}\; i \ne j \}, \\[1ex]
\pi/2 & \text{if }  C_n = \mathbb{S}^n_\p, \\[1ex]
K_\G & \text{if }  C_n = \bigl\{\begin{bsmallmatrix} 0 & B \\ B^* & 0 \end{bsmallmatrix} \in \mathbb{S}^n : B \in \mathbb{R}^{m \times (n - m)} \bigr\}, \\[1ex]
K_\Gamma & \text{if }  C_n = \mathbb{S}^n,
\end{cases}
\end{equation}
can sometimes be determined exactly (the Nesterov $\pi/2$-Theorem and  Goemans--Williamson inequality) and other times be bounded (the original and symmetric Grothendieck inequalities); here $\alpha_\GW \approx 0.878567$ is the celebrated  Goemans--Williamson constant. We also have complex analogues of \eqref{eq:CGconst} with different values for the respective constants. For example, over $\mathbb{C}$, the constant $\pi/2$ in the Nesterov $\pi/2$-Theorem becomes the $4/\pi$ in the Ben-Tal--Nemirovski--Roos $4/\pi$-Theorem \cite{Ben-Tal}.

There has been speculation \cite{KO} that the Goemans--Williamson inequality is somehow related to the original Grothendieck inequality \eqref{eq:GI} although we have not been able to find an unambiguous statement of their exact relation.  Since a weighted Laplacian $L$ is always positive semidefinite, we may invoke \eqref{eq:equalnorms} and write
\[
\lVert L \rVert_{\infty,1} \le  \lVert L \rVert_\G \le K_\G \lVert L \rVert_{\infty,1},
\]
and since the set of all weighted Laplacians $\mathbb{L}^n$ is a strictly smaller subset of $\mathbb{R}^{n \times n}$, the Grothendieck constant $K_\G$ may be reduced to the (reciprocal of) Goemans--Williamson constant $1/\alpha_\GW$. This in our view spells the relationship between the two inequalities --- Goemans--Williamson inequality is Grothendieck inequality restricted to $\mathbb{L}^n$. The crux of  Goemans--Williamson inequality, namely, the existence of a universal constant independent of $n$, the size of the graph $G$, is characteristic of the Grothendieck inequality.

\subsection{Grothendieck norms}\label{sec:d}

One small novelty introduced in this article is that we approach the Grothendieck inequality and its extensions by  viewing them as relations within a family  of norms and seminorms. For any $d,m,n \in \mathbb{N}$, $A \in \mathbb{S}^n$, and $B\in\Bbbk^{m \times n}$, we define
\begin{equation}\label{eq:d}
\begin{aligned}
\lVert A\rVert_{\gamma,d} &\coloneqq \max \biggl\{ \biggl\lvert \sum_{i=1}^n \sum_{j=1}^n a_{ij}\langle x_i,x_j\rangle \biggr\rvert :  x_i\in \Bbbk^d,\; \lVert x_i \rVert  = 1\biggr\},\\ 
\lVert A\rVert_{\Gamma,d} &\coloneqq \max \biggl\{ \biggl\lvert \sum_{i=1}^n \sum_{j=1}^n a_{ij}\langle x_i,x_j\rangle \biggr\rvert :  x_i\in \Bbbk^d,\; \lVert x_i \rVert  \le 1\biggr\},\\ 
\lVert B\rVert_{\G,d} &\coloneqq \max \biggl\{ \biggl\lvert \sum_{i=1}^m \sum_{j=1}^{n} b_{ij}\langle x_i,y_j\rangle \biggr\rvert :  x_i, y_j\in \Bbbk^d,\; \lVert x_i \rVert = 1,\; \lVert y_j \rVert  = 1\biggr\}\\
&= \max \biggl\{ \biggl\lvert \sum_{i=1}^m \sum_{j=1}^{n} b_{ij}\langle x_i,y_j\rangle \biggr\rvert :  x_i, y_j\in \Bbbk^d,\; \lVert x_i \rVert \le 1,\; \lVert y_j \rVert \le 1\biggr\}.
\end{aligned}
\end{equation}
We will show that $\lVert\, \cdot\, \rVert_{\G,d}$ and $\lVert\, \cdot\, \rVert_{\Gamma,d}$ define norms and $\lVert\, \cdot\, \rVert_{\gamma,d}$ a seminorm. We will refer to them collectively as \emph{Grothendieck norms}; or the Grothendieck $d$-norms if we need to specify their \emph{rank} $d \in \mathbb{N}$;
or more specifically as the $(\gamma,d)$-seminorm, $(\Gamma,d)$-norm, and  $(\G,d)$-norm respectively if we need to specify both their \emph{types} and rank. The norms in \eqref{eq:norms} and \eqref{eq:GammaTheta} are special cases:
\[
\lVert\, \cdot\, \rVert_\G = \lVert\, \cdot\, \rVert_{\G,m+n}, \qquad \lVert\, \cdot\, \rVert_{\infty,1} = \lVert\, \cdot\, \rVert_{\G,1}, \qquad \lVert\, \cdot\, \rVert_\Gamma = \lVert\, \cdot\, \rVert_{\Gamma,n}, \qquad \lVert\, \cdot\, \rVert_\Theta = \lVert\, \cdot\, \rVert_{\Gamma,1}.
\]
In fact, we will see that $\lVert\, \cdot\, \rVert_{\G,d} = \lVert\, \cdot\, \rVert_\G$ as soon as $d(d+1)/2 > m+n$ over $\mathbb{R}$ and $d^2 > m+n$ over $\mathbb{C}$;  similarly for 
$\lVert\, \cdot\, \rVert_{\gamma,d}$ and $\lVert\, \cdot\, \rVert_{\Gamma,d}$.
We will establish various characterizations of these norms; closed-form expressions for special matrices; and their relations across different types, ranks, and base fields ($\mathbb{R}$ or $\mathbb{C}$).

More generally, we prove that for each of these three types of norms/seminorms, there is a Grothendieck-like inequality relating any pair of $d$-norm and $p$-norm with $d \le p$:
\[
\lVert A \rVert_{\gamma,p}\le K_{\gamma,d,p} \lVert A \rVert_{\gamma,d},\qquad
\lVert A \rVert_{\Gamma,p}\le K_{\Gamma,d,p} \lVert A \rVert_{\Gamma,d}, \qquad
\lVert B \rVert_{\G,p}\le K_{\G,d,p} \lVert B \rVert_{\G,d}
\]
for all $A \in \mathbb{S}^n$ and all $B \in \Bbbk^{m \times n}$. The original and symmetric Grothendieck inequalities \eqref{eq:GI} and \eqref{eq:GGI} then correspond to the limiting case $d = 1$, $p\to \infty$. Other special cases include the order-$p$ Grothendieck inequality \cite{Acin, Bene, Hirsch, Krivine2}, when the pair of norms are $\lVert\, \cdot\, \rVert_{\G,1}$ and $\lVert\, \cdot\, \rVert_{\G,p}$; the generalized Grothendieck inequality \cite{Bri3}, when the pair of norms are $\lVert\, \cdot\, \rVert_{\G,d}$ and $\lVert\, \cdot\, \rVert_{\G,p}$, $d\le p$; and the rank-constrained positive semidefinite Grothendieck inequality \cite{Bri1, Bri2}, when the pair of norms are $\lVert\, \cdot\, \rVert_{\gamma,d}$ and $\lVert\, \cdot\, \rVert_\gamma$, restricted to $\mathbb{S}^n_\p$. In fact, when combined with our discussion in Section~\ref{sec:CGI}, we obtain an all-encompassing inequality
\begin{equation}\label{eq:moagi}
\lVert A \rVert_{\Gamma,p}^\Bbbk \le K_{\mathsf{C},d,p}^\Bbbk  \lVert A \rVert_{\Gamma,d}^\Bbbk .
\end{equation}
Every inequality that we have mentioned up to this point may be obtained as a special case of \eqref{eq:moagi} by choosing an appropriate family of cones $\mathsf{C}$, appropriate values of  $d$ and $p$, and an appropriate field $\Bbbk =\mathbb{R}$ or $\mathbb{C}$.

\subsection{Polynomial-time approximations}

Like the original Grothendieck inequality, the symmetric Grothendieck inequality has implications on the polynomial-time approximations of NP-hard quantities. We will study several such approximations in Section~\ref{sec:apps}, and will describe two here.

For a bipartite graph with vertices in an independent set of size $m$ and another of size $n$, its adjacency matrix $A \in \mathbb{S}^{m +n}$ takes the form in \eqref{eq:symm} with a weight matrix $B \in \mathbb{R}^{m \times n}$. The \emph{cut norm} of $B$ as defined in \cite{Alon3} has important combinatorial properties but is NP-hard to compute. A consequence of  the symmetric Grothendieck inequality is that
\[
\frac{1}{8K_\Gamma} \lVert L_A \rVert_\Gamma \le  \lVert B \rVert_{\cut} \le \frac{3}{8} \lVert L_A \rVert_\Gamma,
\]
where $L_A \coloneqq \diag(A\mathbbm{1}) - A$ is the Laplacian corresponding to $A$, 
thus affording an approximation via the polynomial-time computable $\Gamma$-norm.

Another consequence of the symmetric Grothendieck inequality is that for any $A \in \mathbb{S}^n$, the gap between two NP-hard combinatorial problems,
\begin{align*}
\str(A)  &= \Bigl( \max\nolimits_{x \in \{-1,+1\}^n} x^\tp\! Ax \Bigr) - \Bigl( \min\nolimits_{x \in \{-1,+1\}^n} x^\tp\! Ax \Bigr),
\intertext{which was studied in \cite{Charikar}, can be uniformly approximated by a polynomial-time computable}
\spr(A) &= \Bigl( \max\nolimits_{X  \in \mathbb{G}^n} \tr( AX)\Bigr) - \Bigl( \min\nolimits_{X  \in \mathbb{G}^n} \tr( AX)\Bigr),
\end{align*}
where $\mathbb{G}^n$ is the convex set of $n\times n$ correlation matrices. More precisely, the symmetric Grothendieck inequality will imply that
\[
\str(A) \le \spr(A) \le K_\gamma \str(A),
\]
where $K_\gamma$ is the symmetric Grothendieck constant corresponding to the $(\gamma,d)$-seminorms in \eqref{eq:d}.

\subsection*{Notations}

We introduce some notations and conventions that will make an appearance later. To avoid clutter, any statement about $\mathbb{S}^n$ would mean that it applies to both the  space of real symmetric matrices and the space of complex Hermitian matrices alike. If a statement applies only to the former or the latter, we will write $\mathbb{S}^n(\mathbb{R})$ or $\mathbb{S}^n(\mathbb{C})$ respectively. This same convention will apply to all other notations below. We remind the reader that $\mathbb{S}^n(\mathbb{C})$ is a vector space over $\mathbb{R}$ but \emph{not} over $\mathbb{C}$; and that $\mathbb{S}^n \times \mathbb{S}^n \to \mathbb{R}$, $(A,B) \mapsto \tr(AB)$ defines a \emph{real} inner product over both  $\mathbb{S}^n(\mathbb{R})$ and $\mathbb{S}^n(\mathbb{C})$ alike.

The following two subspaces of $\mathbb{S}^n$ will play important roles in our article:
\begin{align*}
\mathbb{S}^n_\eq &\coloneqq \{A\in \mathbb{S}^n : a_{11}=\dots=a_{nn}\},\\
\mathbb{S}^n_\circ &\coloneqq \{A\in \mathbb{S}^n : a_{11}=\dots=a_{nn}=0\},
\end{align*}
i.e., the subspace of matrices with equal diagonal entries, and its subspace comprising those whose diagonal entries are all zeros. As usual,  we write $\mathbb{S}^n_\p$ and $\mathbb{S}^n_\pp$ for the cones of positive semidefinite and positive definite matrices respectively. Throughout this article a \emph{cone} will always mean a convex cone.

For $x,y \in \Bbbk^d$, we write
\[
\langle x, y \rangle_\Bbbk = \begin{cases}
x^\tp y & \text{if }\Bbbk = \mathbb{R},\\
x^* y & \text{if }\Bbbk = \mathbb{C},
\end{cases}
\]
for the Euclidean and Hermitian inner products.
Again, we will usually write $\langle x, y \rangle$ and only indicate field dependence when it is necessary.

For $x_1,\dots,x_n\in \Bbbk^d$, we write
\[
G(x_1,\dots,x_n)\coloneqq \begin{bmatrix}
 \langle x_1,x_1\rangle & \langle x_1,x_2\rangle & \dots & \langle x_1,x_n\rangle \\
 \langle x_2,x_1\rangle & \langle x_2,x_2\rangle & \dots & \langle x_2,x_n\rangle \\
 \vdots & \vdots & \ddots & \vdots \\
 \langle x_n,x_1\rangle & \langle x_n,x_2\rangle & \dots & \langle x_n,x_n\rangle
\end{bmatrix} \in \mathbb{S}^n
\]
for the corresponding Gram matrix. Let
\begin{equation}\label{eq:Gnd}
\begin{aligned}
\mathbb{G}^n_d &\coloneqq \{G(x_1,\dots,x_n) \in \mathbb{S}^n :  \lVert x_1 \rVert =1,\dots,\lVert x_n\rVert=1, \; x_1,\dots,x_n\in\Bbbk^d\}\\
&= \{ G \in \mathbb{S}^n_\p : g_{11} =\dots =g_{nn} =1,\; \rank(G) \le d\}
\end{aligned}
\end{equation}
be the set of  rank-$d$ \emph{correlation matrices}; and
\begin{equation}\label{eq:Gammand}
\begin{aligned}
\bbGamma^n_d &\coloneqq \{G(x_1,\dots,x_n) \in \mathbb{S}^n :  \lVert x_1 \rVert \le 1,\dots,\lVert x_n\rVert\le 1, \; x_1,\dots,x_n\in\Bbbk^d\}\\
&= \{ G \in \mathbb{S}^n_\p : g_{11}, \dots, g_{nn} \in [0,1],\; \rank(G) \le d\}
\end{aligned}
\end{equation}
be the set of  rank-$d$ \emph{subcorrelation matrices}.
Clearly,
\[
\mathbb{G}^n_1 \subseteq \mathbb{G}^n_2 \subseteq \dots \subseteq \mathbb{G}^n_n \subseteq \mathbb{S}_\eq^n \cap \mathbb{S}^n_\p,\qquad
\bbGamma^n_1 \subseteq \bbGamma^n_2 \subseteq \dots \subseteq \bbGamma^n_n \subseteq  \mathbb{S}^n_\p.
\]
For any $d \ge n$, we have $\mathbb{G}^n_d = \mathbb{G}^n_n$ and $\bbGamma^n_d = \bbGamma^n_n$, which we denote by $\mathbb{G}^n$ and $\bbGamma^n$ respectively, i.e.,
\begin{align*}
\mathbb{G}^n &= \{ G \in \mathbb{S}^n_\p : g_{11} =\dots =g_{nn} =1\},\\
\bbGamma^n  &= \{ G \in \mathbb{S}^n_\p : g_{11}, \dots, g_{nn} \in [0,1]\},
\end{align*}
are the \emph{convex} sets of correlation matrices and subcorrelation matrices respectively.
For the special case $d = 1$, the rank-one correlation and subcorrelation matrices may also  be expressed as
\begin{align*}
\mathbb{G}^n_1 &= \{ xx^* \in \mathbb{S}^n :  \lvert x_1\rvert =1,\dots,\lvert x_n\rvert =1\},\\
\bbGamma^n_1 &= \{ xx^* \in \mathbb{S}^n :  \lvert x_1\rvert \le 1,\dots,\lvert x_n\rvert \le 1\}.
\end{align*}
Note also that
\[
\spn(\mathbb{G}^n) = \mathbb{S}_\eq^n, \qquad \cone(\mathbb{G}^n) = \mathbb{S}_\eq^n \cap \mathbb{S}^n_\p, \qquad
\spn(\bbGamma^n) = \mathbb{S}^n, \qquad \cone(\bbGamma^n) = \mathbb{S}^n_\p.
\]

Reiterating the point we made in the first paragraph, we write $\mathbb{S}^n(\Bbbk)$,  $\mathbb{S}^n_\circ(\Bbbk)$,  $\mathbb{S}^n_\eq(\Bbbk)$, $\mathbb{S}^n_\p(\Bbbk)$, $\mathbb{S}^n_\pp(\Bbbk)$, $\mathbb{G}^n(\Bbbk)$,  $\bbGamma^n(\Bbbk)$ for $\Bbbk =\mathbb{R}$ or $\mathbb{C}$ only when it is necessary to emphasize the field dependence.

We write $\mathbb{R}_\p =[0,\infty)$ for the nonnegative reals. The closed unit disk, denoted $\mathbb{D}$, can either be real, $\mathbb{D} = [-1,1]$, or complex, $\mathbb{D} = \{z \in \mathbb{C} : \lvert z \rvert \le 1\}$, and will be clearly specified. Likewise the unit circle, denoted $\mathbb{T}$, can either be real,  $\mathbb{T} =\{-1,1\}$, or complex, $\mathbb{T} =  \{z \in \mathbb{C} : \lvert z \rvert = 1\}$, and will be clearly specified. Derivatives of these notations like
$\mathbb{R}_\p^{m \times n}$, $\mathbb{S}^n(\mathbb{R}_\p\!)$, $\mathbb{S}^n_\circ(\mathbb{R}_\p\!) $, $\mathbb{S}^n_\eq(\mathbb{D})$, etc, are self-explanatory.

All vectors $x \in \Bbbk^n$ would be regarded as column vectors. Row vectors will always be denoted with transpose or  adjoint, $x^\tp$ or $x^*$. When enclosed in parentheses, a vector denoted $(x_1,\dots,x_n)\in \Bbbk^n$ would mean a \emph{column} vector with entries $x_1,\dots,x_n \in \Bbbk$. Matrices are denoted in uppercase letters, e.g., $A$, with its $(i,j)$th entry denoted in corresponding lowercase letters, e.g., $a_{ij}$.

We define our sign function over $\Bbbk = \mathbb{R}$ and $\mathbb{C}$ as
\begin{equation}\label{eq:sign}
\sign z= \begin{cases}
z/\lvert z \rvert & z\ne 0,\\
1 & z = 0.
\end{cases}
\end{equation}
The somewhat nonstandard convention $\sign 0 \coloneqq 1$, ensuring that $\sign$ always takes values in $\mathbb{T}$, will simplify some proofs (e.g., Theorem~\ref{FRNineq}) but makes no difference elsewhere (e.g., Lemma~\ref{2-thm1}).

\section{Grothendieck norms}\label{sec:Gnorms}

The perspective adopted in this article is that the Grothendieck inequality is a statement about the $\G$-norm and the $(\infty,1)$-norm, and that the symmetric Grothendieck inequality is an analogous statement about the $\Gamma$-norm and the $\Theta$-norm. More generally, all inequalities discussed in this article, including various generalizations \cite{Acin,  Bri3, Bri1, Bri2,  Bene, Hirsch} and specializations \cite{Ben-Tal, GW95, Nes, Rietz} of the  Grothendieck inequality, will be viewed as relations between the Grothendieck $d$-norms for different values of $d$. As such, we will provide a fairly extensive discussion of these norms in this section.

\subsection{The $\gamma$- and $\theta$-seminorms}

An observation \cite[Lemma~2.2]{FLZ18} about  the $\G$-norm and the $(\infty,1)$-norm in \eqref{eq:norms} is that we may replace ``$\le 1$'' with ``$=1$'' in the maxima and leave the norms unchanged, i.e.,
\begin{equation}\label{eq:norms1}
\lVert B\rVert_\G = \max_{\lVert x_i\rVert  =\lVert y_j \rVert = 1}\biggl\lvert \sum_{i=1}^m\sum_{j=1}^n b_{ij} \langle  x_i,y_j\rangle\biggr\rvert \quad \text{and}\quad
\lVert B\rVert_{\infty,1} = \max_{\lvert \varepsilon_i \rvert = \lvert \delta_j \rvert = 1}\biggl\lvert\sum_{i=1}^m\sum_{j=1}^n b_{ij}  \bar{\varepsilon}_i\delta_j\biggr\rvert.
\end{equation}
In other words, for the $\G$-norm and $(\infty,1)$-norm, it does not matter whether we take maximum over the unit ball or the unit sphere.\footnote{This applies more generally to $\lVert\,\cdot\,\rVert_{\G,d}$, which is why it has two expressions in \eqref{eq:d}.} This is not the case for their symmetric counterparts in \eqref{eq:GammaTheta}, replacing the ``$\le 1$'' with ``$=1$'' in the maxima in \eqref{eq:GammaTheta} gives us two related \emph{seminorms} on $\mathbb{S}^n$ that we will denote as
\begin{equation}\label{eq:seminorms2}
\lVert A \rVert_\gamma \coloneqq \max_{\lVert x_i\rVert  = 1}\; \biggl\lvert\sum_{i=1}^n \sum_{j=1}^n a_{ij}\langle x_i, x_j\rangle \biggr\rvert
\quad\text{and}\quad
\lVert A \rVert_\theta \coloneqq \max_{ \lvert \delta_i \rvert = 1}\; \biggl\lvert\sum_{i=1}^n\sum_{j=1}^n a_{ij}  \bar{\delta}_i\delta_j\biggr\rvert,
\end{equation}
where $x_1,\dots, x_n \in \Bbbk^d$ with $d \ge n$, and $\delta_1,\dots, \delta_n \in \Bbbk$.
Clearly,
\begin{equation}\label{eq:semilenorm}
\lVert A \rVert_\theta\le \lVert A \rVert_\Theta, \qquad \lVert A \rVert_\gamma\le \lVert A \rVert_\Gamma,
\end{equation}
and the inequalities are strict when $A \in \mathbb{S}^n$ is nonzero, diagonal, and traceless. It is also clear that
\begin{equation}\label{eq:tleg}
 \lVert A \rVert_\theta\le  \lVert A \rVert_\gamma. 
\end{equation}
We will later see that there is also a symmetric Grothendieck inequality for these seminorms: 
\begin{equation}\label{eq:GGILE}
\lVert A \rVert_\gamma \le K_\gamma \lVert A \rVert_\theta
\end{equation}
for some universal constant $K_\gamma$ independent of $n \in \mathbb{N}$ and $A \in \mathbb{S}^n$.

\subsection{Basic properties}

While there is an apparent dependence on $d$ in the definitions of $\lVert\,\cdot\,\rVert_\G$, $\lVert\,\cdot\,\rVert_\Gamma$, and $\lVert\,\cdot\,\rVert_\gamma$, our assumption that $d\ge m+n$ in the first case and $d \ge n$ in the latter two cases remove this dependence. However, it will aid our discussions here and in the rest of our article to study a family of norms/seminorms that interpolate between $\lVert\,\cdot\,\rVert_{\infty,1}$ and $\lVert\,\cdot\,\rVert_\G$, $\lVert\,\cdot\,\rVert_\Theta$ and $\lVert\,\cdot\,\rVert_\Gamma$,  $\lVert\,\cdot\,\rVert_\theta$ and $\lVert\,\cdot\,\rVert_\gamma$ --- these are the Grothendieck $d$-norms that we introduced in Section~\ref{sec:d}. We will establish that they are indeed norms/seminorms in Lemma~\ref{gamnorms}.  We will see that results for Grothendieck $d$-norms are often no harder to establish for general $d$ than for specific cases like $d=1$, $n$, $m+n$.
\begin{lemma}\label{lem:d}
Let $d,m,n \in \mathbb{N}$ and the Grothendieck $d$-norms be as defined in \eqref{eq:d}. For any $A \in \mathbb{S}^n$ and $B\in\Bbbk^{m \times n}$, we have
\begin{equation}\label{eq:increase}
\begin{aligned}
\lVert A\rVert_\theta &= \lVert A\rVert_{\gamma,1} \le \lVert A\rVert_{\gamma,2} \le \dots \le \lVert A\rVert_{\gamma,n} =   \lVert A\rVert_\gamma,\\
\lVert A\rVert_\Theta &= \lVert A\rVert_{\Gamma,1} \le \lVert A\rVert_{\Gamma,2} \le \dots \le \lVert A\rVert_{\Gamma,n} =   \lVert A\rVert_\Gamma,\\
\lVert B\rVert_{\infty,1} &= \lVert B\rVert_{\G,1} \le \lVert B\rVert_{\G,2} \le \dots \le \lVert B\rVert_{\G,m+n} =   \lVert B\rVert_\G,
\end{aligned}
\end{equation}
and for all $d \ge n$ (first two) or $d \ge m+n$ (third),
\begin{equation}\label{eq:stabilize}
\lVert A\rVert_{\gamma,d} =   \lVert A\rVert_\gamma, \qquad \lVert A\rVert_{\Gamma,d} =   \lVert A\rVert_\Gamma, \qquad \lVert B\rVert_{\G,d} =   \lVert B\rVert_\G.
\end{equation}
\end{lemma}
\begin{proof}
The equalities and inequalities in \eqref{eq:increase} follow from the definitions in \eqref{eq:d}. The stabilization \eqref{eq:stabilize} is well-known for the $(\G,d)$-norms and the same reason applies to the $(\Gamma,d)$-norms and $(\gamma,d)$-seminorms: When $d \ge n$, any $n$ vectors $x_1,\dots,x_n\in \Bbbk^d$ lie in an $n$-dimensional subspace of $\Bbbk^d$ and maximizing over it gives the same value as maximizing over $\Bbbk^n$. 
\end{proof}

The inequalities \eqref{eq:increase} refine
\[
\lVert A\rVert_\theta \le \lVert A\rVert_\gamma, \qquad 
\lVert A\rVert_\Theta \le \lVert A\rVert_\Gamma, \qquad
\lVert B\rVert_{\infty, 1} \le \lVert B\rVert_\G
\]
that we saw in \eqref{eq:tleg}, \eqref{eq:GGI2}, and \eqref{eq:GI2}. 

By virtue of \eqref{eq:stabilize}, we may assume that $d =n$ when we discuss $\lVert\, \cdot\, \rVert_\gamma$ or $\lVert\, \cdot\, \rVert_\Gamma$, and $d = m+n$ when we discuss $\lVert\, \cdot\, \rVert_\G$. In fact, the stabilizations in \eqref{eq:stabilize} hold for much smaller values of $d$: We will see in Section~\ref{sec:stable} that over $\mathbb{R}$,  $\lVert\,\cdot\,\rVert_{\gamma,d} =   \lVert \,\cdot\,\rVert_\gamma$ and $\lVert\,\cdot\,\rVert_{\Gamma,d} =   \lVert\,\cdot\,\rVert_\Gamma$ as soon as $d(d+1)/2 >n$ and $\lVert\,\cdot\,\rVert_{\G,d} =   \lVert\,\cdot\,\rVert_\G$ as soon as $d(d+1)/2 > m+n$;   over $\mathbb{C}$, these happen even sooner --- $d(d+1)/2$ may be replaced by $d^2$.  

The $(\G,d)$-norms have additional  alternative expressions that will be useful later. The following hold over over $\mathbb{R}$ and $\mathbb{C}$ alike but we may of course drop the ``$\Re$''  over $\mathbb{R}$. Note however that there is no analogue of \eqref{eq:G=le2} for the  $(\gamma,d)$-seminorms and $(\Gamma,d)$-norms, i.e., the absolute value cannot be replaced by the real part in their definitions in \eqref{eq:d}.
\begin{proposition}\label{prop2.2}
Let $d \in \mathbb{N}$. If $B\in\Bbbk^{m \times n}$, then
\begin{align}
\lVert B\rVert_{\G,d} &= \max \biggl\{ \Re \sum_{i=1}^m \sum_{j=1}^{n} b_{ij}\langle x_i,y_j\rangle :  x_i,y_j\in \Bbbk^d,\; \lVert x_i \rVert \le 1,\;\lVert y_j \rVert  \le 1\biggr\}. \label{eq:G=le2}
\intertext{If $A \in \mathbb{S}^n$, then}
\lVert A\rVert_{\G,d} &= \frac{1}{2} \max \biggl\{ \sum_{i=1}^n \sum_{j=1}^{n} a_{ij}(\langle x_i,y_j\rangle+\langle y_i,x_j\rangle) :  x_i,y_j\in \Bbbk^d,\; \lVert x_i \rVert \le 1,\; \lVert y_j \rVert \le 1\biggr\}. \label{eq:G=le}
\end{align}
In both \eqref{eq:G=le2} and \eqref{eq:G=le}, the \emph{``$\le 1$''} constraints may be replaced by \emph{``$= 1$''}.
\end{proposition}
\begin{proof}
The equalities \eqref{eq:G=le2} are a simple consequence of the fact that  any $z\in\mathbb{C}$  has $\lvert z \rvert=\Re t z$ for some $t \in \mathbb{T}$, and $\lVert tx_i\rVert=\lvert t \rvert\lVert x_i \rVert = \lVert x_i \rVert$; ditto for $\mathbb{R}$. If $A =A^*$, then
\[
\overline{\tr A Z}=\tr \overline{AZ}=\tr\overline{(AZ)^\tp}=\tr Z^*\! A^*=\tr Z^*\! A= \tr A Z^*,
\]
and we get
\[
\Re \tr (A Z) =\tr\biggl[A\biggl(\frac{Z+Z^*}{2}\biggr)\biggr],
\]
from which \eqref{eq:G=le} follows.
\end{proof}

Define the maps $\Delta, \Xi : \mathbb{S}^n \to \mathbb{S}^n$ by
\[
\Delta(A) = \begin{bmatrix} a_{11} &  & & \\ & a_{22} & & \\ & & \ddots & \\ & & & a_{nn} \end{bmatrix},
\qquad 
\Xi(A) = \begin{bmatrix}
\frac{1}{n} \tr(A) & a_{12} & \dots & a_{1n} \\
a_{21} & \frac{1}{n} \tr(A) &  & a_{2n} \\
\vdots &  & \ddots & \vdots \\
a_{n1} & a_{n2} & \dots & \frac{1}{n} \tr(A)  \\
\end{bmatrix}
\]
for any $A\in \mathbb{S}^n$. The maps $\Delta$ and $\Xi$ are respectively the orthogonal projections onto the subspace of diagonal matrices  and the subspace $\mathbb{S}^n_\eq$.
Some immediate observations are that $\tr\bigl(\Xi(A)\bigr) = \tr(A)$, that
$x^*\! Ax =x^*\Xi(A)x$ if $xx^*\in \mathbb{G}^n_1$, and that  $\Delta(x x^*)=I$ iff $xx^* \in \mathbb{G}^n_1$.

\begin{lemma}\label{gamnorms}
Let $A\in \mathbb{S}^n$ and $d \in \mathbb{N}$.  Then
\begin{enumerate}[\upshape (i)]
\item\label{it:zero} $\lVert A \rVert_{\gamma,d}=0$ if and only if $\Xi(A)=0$;

\item\label{it:restrict}  $\lVert\,\cdot\,\rVert_{\gamma,d}$ defines a seminorm on  $\mathbb{S}^n$ and a norm on $\mathbb{S}_\eq^n$;

\item\label{it:norm} $\lVert\,\cdot\,\rVert_{\Gamma,d}$ defines a norm on $\mathbb{S}^n$.
\end{enumerate}
\end{lemma}
\begin{proof}
Nonnegativity, scaling, and triangle inequality are either obvious or straightforward to establish in all cases. This proof focuses on the definiteness.
For $n=1,2$, the proofs of \eqref{it:zero} and \eqref{it:restrict} are easy. We will assume that $n\ge 3$.

Let $\mathbb{T}$ denote the unit circle in $ \Bbbk = \mathbb{R}$ or $\mathbb{C}$ and  $\mu$ be the corresponding Haar measure. Thus $\mathbb{T} =\{-1,1\}$ and $\mu$ the uniform counting measure if $ \Bbbk = \mathbb{R}$; and $\mathbb{T} = \{e^{\mathrm{i}\vartheta} \in\mathbb{C} : \vartheta\in [0, 2\pi)\}$ and  $d\mu =(2\pi)^{-1} d\vartheta$ if  $\Bbbk = \mathbb{C}$. Our arguments below will apply to both fields.

Consider the $n$-torus
\[
\mathbb{T}^n \coloneqq  \mathbb{T} \times \dots \times \mathbb{T} = \{x\in\Bbbk^n : \Delta (x x^*)=I\} = \{x\in\Bbbk^n : x x^* \in \mathbb{G}^n_1\}.
\]
Let $\mu_n$ be the product measure on $\mathbb{T}^n$. For any $A\in \Bbbk^{n\times n}$, straightforward calculations give
\[
\lvert\tr(A)\rvert=\biggl\lvert \int_{x\in\mathbb{T}^n}\!\!\! x^*\! A x\, d\mu_n\biggr\rvert \le \int_{x\in\mathbb{T}^n}\!\!\! \lvert x^*\! A x\rvert \, d\mu_n\le \lVert A\rVert_\theta.
\]
If $\lVert A\rVert_\theta =0$, then  $\tr(A)=0$ and so $\Delta\bigl(\Xi(A)\bigr)=0$. We claim that $a_{ij}=0$ for all $i\ne j$.  Without loss of generality, we may assume $i=n-1$ and $j=n$. We will show that $a_{n-1,n} = 0$.  Let  $x(s)= (s,t,1) \in \mathbb{T}^n$ where $t  \in \mathbb{T}$ is fixed and  $s\in\mathbb{T}^{n-2}$.  Again it is  straightforward to see that
\[
\Re(ta_{n-1,n})=\frac{1}{2} \int_{s\in \mathbb{T}^{n-2}}\!\!\! x(s)^*\! Ax(s) \; d\mu_{n-2}.
\]
Hence $\lvert \Re(ta_{n-1,n})\rvert \le \lVert A \rVert_\theta=0$.  Since $t\in\mathbb{T}$ is arbitrary, we have $a_{n-1,n} =0$.  Hence $\Xi(A)=0$. The converse is obvious: If  $\Xi(A)=0$, then $\lVert A \rVert_\theta=0$.  

In other words, if $\lVert A \rVert_{\gamma,1}=\lVert A \rVert_\theta=0$, then $A$ is a diagonal matrix with $\tr(A)=0$, and so $\lVert A \rVert_{\gamma,d}=0$ for all $d \in \mathbb{N}$.  On the other hand, if  $\lVert A \rVert_{\gamma,d}=0$ for any $d\in \mathbb{N}$, then $\lVert A \rVert_\theta=0$ by \eqref{eq:increase}. We have established \eqref{it:zero}.

Since $\lVert\,\cdot\,\rVert_{\gamma,d}$ vanishes exactly on $\ker(\Xi)$, the subspace of traceless diagonal matrices, it defines a norm on the orthogonal complement $\im(\Xi)=\mathbb{S}^n_\eq$. This establishes \eqref{it:restrict}.

For \eqref{it:norm}, it is obvious by their definitions that for any $A \in \mathbb{S}^n$,
\[
\lVert A \rVert_{\gamma,d} \le \lVert A \rVert_{\Gamma,d}.
\]
So if $\lVert A \rVert_{\Gamma,d}=0$, then $\lVert A \rVert_{\gamma,d}=0$ and thus $\Xi(A)=0$. So $A$ is a traceless diagonal matrix.  But it also follows from \eqref{eq:increase} that if $\lVert A \rVert_{\Gamma,d}=0$, then $\lVert A \rVert_\Theta=0$, and so  $a_{ii} = e_i^\tp\! Ae_i = \tr(Ae_ie_i^\tp)=0$, $i=1,\dots,n$.  This shows that $A=0$. Hence $\lVert\, \cdot\, \rVert_{\Gamma,d}$ is a norm on $\mathbb{S}^n$.
\end{proof}
The proof of Lemma~\ref{gamnorms}  in fact shows that $\lVert\,\cdot\,\rVert_{\gamma,d}$ is a norm on a subspace $W\subseteq  \mathbb{S}^n$ if and only if $W$ does not contain a nonzero traceless diagonal matrix.  Furthermore, any subspace $W$ on which $\lVert\,\cdot\,\rVert_{\gamma,d}$ restricts to a norm can have dimension at most $\dim \mathbb{S}_\eq^n$.
We record this slightly stronger conclusion  below.
\begin{corollary}
The seminorm $\lVert\,\cdot\,\rVert_{\gamma,d}$ on $\mathbb{S}^n$ restricts to  a norm on any subspace $W\subseteq \mathbb{S}^n$ that does not contain a nonzero diagonal matrix with zero trace. If  $\lVert\,\cdot\,\rVert_{\gamma,d}$ defines a norm on $W\subseteq \mathbb{S}^n$, then
\[
\dim W \le
\begin{cases}
n(n - 1)/2 +1 &\Bbbk = \mathbb{R},\\
n^2 - n +1 &\Bbbk = \mathbb{C}.
\end{cases}
\]
\end{corollary}

\subsection{Characterizations of and relations between Grothendieck norms}

We will look at some alternative characterizations of the Grothendieck norms and relations between them that will be useful later. 

The most straightforward characterization is to simply rewrite the $(\Gamma,d)$-norms and $(\gamma,d)$-seminorms in Lemma~\ref{lem:d}  in terms of the correlation and subcorrelation matrices in  \eqref{eq:Gnd} and \eqref{eq:Gammand}.
\begin{proposition}\label{prop:correlation}
Let $A \in \mathbb{S}^n$ and $d \in \mathbb{N}$. Then
\begin{equation}\label{eq:normcorrd}
\lVert A \rVert_{\gamma,d} =\max\{\lvert \tr (AG) \rvert : G \in \mathbb{G}^n_d\}, \quad
\lVert A \rVert_{\Gamma,d} = \max\{\lvert \tr (AG) \rvert : G \in \bbGamma^n_d\},
\end{equation}
and
\begin{equation}\label{eq:altGamma}
\begin{aligned}
\lVert A \rVert_{\Gamma,d}&=\max\{\lVert DAD \rVert_{\gamma,d} :  D = \diag(\delta_1,\dots, \delta_n),\; \delta_i \in[0,1] \}.
\end{aligned}
\end{equation}
\end{proposition}
\begin{proof}
Let $x_i \in\Bbbk^d$ with $\lVert x_i \rVert\le 1$.  Write $x_i=\delta_i t_i$ with $t_i\in\Bbbk^d$, $\lVert t_i \rVert=1$, and $\delta_i\in [0,1]$. Then $\sum_{i,j=1}^{m,n} a_{ij} \langle  x_i,x_j\rangle= \sum_{i,j=1}^{m,n} \delta_ia_{ij}\delta_j \langle  t_i,t_j\rangle$.
\end{proof}
\begin{corollary}
Let  $d\in\mathbb{N}$.
If $A\in \mathbb{S}^n$, then
\begin{equation}\label{ineqvarnrm}
\lVert A \rVert_{\Gamma,d}\le \lVert A \rVert_{\G,d}\le 2 \lVert A \rVert_{\Gamma,d}.
\end{equation}
If $B\in\Bbbk^{m\times n}$, then
\begin{equation}\label{eqvarnrm}
\lVert B\rVert_{\G,d} = \frac12  \biggl\lVert \begin{bmatrix}0&B\\B^*&0\end{bmatrix} \biggr\rVert_{\gamma,d} =\frac12  \biggl\lVert \begin{bmatrix}0&B\\B^*&0\end{bmatrix} \biggr\rVert_{\Gamma,d}.
\end{equation}
\end{corollary}
\begin{proof}
For $x_i,y_i\in\Bbbk^d$, $\lVert x_i\rVert \le 1$, $\lVert y_i\rVert\le 1$, $i =1,\dots,n$, set
\[
u_i=\frac{1}{2}(x_i+y_i),\quad v_i=\frac{1}{2}(x_i-y_i),
\]
and let $Z = \bigl(\langle x_i, y_j\rangle\bigr)_{i,j=1}^n \in \Bbbk^{n\times n}$.
Note that $\lVert u_i\rVert \le 1$, $\lVert v_i\rVert \le 1$, and
\[
\frac{Z+Z^*}{2} = G(u_1,\dots,u_n) -G(v_1,\dots,v_n).
\]
Now use \eqref{eq:G=le} to deduce $ \lVert A \rVert_{\G,d}\le 2\lVert A \rVert_{\Gamma,d}$.

Let $B \in \Bbbk^{m \times n}$ be given and set
\[
A = \begin{bmatrix}0&B\\B^*&0\end{bmatrix} \in \mathbb{S}^{m+n}_\eq.
\]
Let $x_1,\dots,x_m,y_1,\dots,y_n \in \Bbbk^d$ be  unit vectors and $G=G(x_1,\dots,x_m,y_1,\dots,y_n) \in \mathbb{G}^{m+n}_d$. Observe that
\[
\tr(AG) = 2\Re\biggl(\sum_{i=1}^m \sum_{j=1}^n b_{ij} \langle x_i,y_j\rangle \biggr).
\]
Using the first characterization of $(\G,d)$-norms in \eqref{eq:d}, we see that $\lVert A \rVert_{\gamma,d} \le 2 \lVert B \rVert_{\G,d}$.  
Replacing $y_j$ with $ty_j$ for all $j=1,\dots,n$, where $t\in \mathbb{T}$ is suitably chosen so that 
\[
\lVert B \rVert_{\G,d}=\biggl\lvert \sum_{i=1}^m\sum_{j=1}^n b_{ij} \langle  x_i,y_j\rangle\biggr\rvert =\sum_{i=1}^m\sum_{j=1}^n b_{ij} \langle  x_i,ty_j\rangle,
\]
we see that $\lVert A \rVert_{\gamma,d} = 2\lVert B \rVert_{\G,d}$. Repeating the same arguments with the second characterization of $(\G,d)$-norms in \eqref{eq:d} gives $\lVert A \rVert_{\Gamma,d}= 2\lVert B \rVert_{\G,d}$.
\end{proof}
The relations between the different types of Grothendieck norms are best seen from their restrictions to special subspaces and subcones of $\mathbb{S}^n$, notably, $\mathbb{S}^n_\circ$, $\mathbb{S}^n_\p$, and $\mathbb{S}^n(\mathbb{R}_\p\!)$, which we will discuss in the following.

\begin{lemma}\label{lem:Szero}
Let  $A \in \mathbb{S}^n_\circ$  and $d \in \mathbb{N}$. Then
\[
\lVert A \rVert_{\gamma,d} = \lVert A \rVert_{\Gamma,d}.
\]
\end{lemma}
\begin{proof}
Clearly $\lVert A \rVert_{\gamma,d} \le \lVert A \rVert_{\Gamma,d}$.
For $A \in \mathbb{S}^n_\circ$, $a_{11} =\dots = a_{nn}=0$ and so
\[
\tr (AG) = \sum\nolimits_{i\ne j} a_{ij} g_{ji}
\]
is independent of the diagonal entries of $G$. By the characterizations in \eqref{eq:normcorrd}, let $G \in \bbGamma^n_d$ be such that
$\lVert A \rVert_{\Gamma,d} =\lvert \tr (AG) \rvert$. We claim that $G \in \mathbb{G}^n_d$, i.e., its diagonal entries are all $1$'s. Suppose not; let $G = G(x_1,\dots,x_n)$ have a maximal number of $x_i$'s with $\lVert x_i\rVert=1$ but that $\lVert x_k\rVert <1$ for some $k$. Since $\lVert A\rVert_{\Gamma,d}=\tr\bigl(tAG(x_1,\dots,x_n)\bigr)$ for $t=\pm 1$, it is a real linear function in $x_k$ when all other variables are fixed, and hence attains its maximum when $\lVert x_k\rVert = 1$, a contradiction.
\end{proof}

We now show that over  $\mathbb{S}^n_\p$, the three different types of Grothendieck $d$-norms agree.  As usual, for $A, B \in \mathbb{S}^n$, we write $A\succeq B$ iff $A -B \in\mathbb{S}^n_\p$.
\begin{proposition}\label{posdefcase}
Let  $d \in \mathbb{N}$.
If $A\in \mathbb{S}^n$, then
\begin{equation}\label{symmcase} 
\lVert A \rVert_{\gamma,d} \le \lVert A \rVert_{\Gamma,d} \le \lVert A\rVert_{\G,d}.
\end{equation}
If $A\in \mathbb{S}^n_\p$, then
\begin{equation}\label{posdefcase1} 
\lVert A \rVert_{\gamma,d} =\lVert A \rVert_{\Gamma,d} = \lVert A\rVert_{\G,d};
\end{equation}
furthermore, for any $B\in \mathbb{S}^n$ with $-A\preceq B\preceq A$,
\[
\lVert B \rVert_{\gamma,d} \le \lVert A \rVert_{\gamma,d}, \qquad
\lVert B \rVert_{\Gamma,d} \le \lVert A \rVert_{\Gamma,d}, \qquad
\lVert B \rVert_{\G,d}\le \lVert A \rVert_{\G,d}.
\]
\end{proposition}
\begin{proof} 
The inequalities \eqref{symmcase} are obvious from definitions.  It remains to show that equality must hold throughout if $A \in \mathbb{S}^n_\p$. 

Any $A \in \mathbb{S}^n_\p$ has a square root $R\in \mathbb{S}^n_\p$, so $A=R^2$ and $R^*=R$. 
Let $x_1,\dots,x_n, y_1,\dots,y_n\in\Bbbk^d$ be of unit norms.  Define the matrices $X=[x_1,\dots, x_n]$, $Y=[y_1,\dots, y_n]\in \Bbbk^{d \times n}$.  Observe that  $\bigl(\langle x_i, y_j\rangle\bigr)_{i,j=1}^n=X^*Y$.  Hence
\[
\sum_{i=1}^n\sum_{j=1}^n a_{ij}\langle x_i,y_j\rangle = \tr(\overline{A}X^*Y) = \tr(Y\overline{A}X^*) = \tr\bigl[(\overline{R}Y^*)^*(\overline{R}X^*)\bigr].
\]
Cauchy--Schwarz applied to the inner product $\langle U, V\rangle=\tr(U^*V)$ on  $\Bbbk^{d\times n}$ gives
\[
\bigl\lvert\tr\bigl[(\overline{R}Y^*)^*(\overline{R}X^*)\bigr]\bigr\rvert^2 \le \tr(Y\overline{A} Y^*)
\tr(X\overline{A} X^*) \le \max\bigl\{\tr(X\overline{A}X^*)^2,\tr(Y\overline{A} Y^*)^2\bigr\}.
\]
Since $\lVert x_i\rVert =\lVert y_j\rVert =1$ for all $i,j$, the last inequality gives us $\lVert \overline{A} \rVert_{\G,d} = \lVert \overline{A} \rVert_{\gamma,d}$ and thus $\lVert A \rVert_{\G,d} = \lVert A \rVert_{\gamma,d}$.

Now suppose $A \in \mathbb{S}^n_\p$ and $B \in \mathbb{S}^n$ with $-A\preceq B\preceq A$.  By taking limits, it suffices to consider the case when $A \in \mathbb{S}^n_\pp$; in which case $R \in \mathbb{S}^n_\pp$ as well.  Let $C=R^{-1}BR^{-1}$.
Then $-I\preceq C\preceq I$, i.e., all eigenvalues of $C$ lie in the interval $[-1,1]$.  So the last statement is equivalent to $C^2\preceq I$.

By Cauchy--Schwarz on $\Bbbk^{d \times n}$,
\begin{align*}
\lvert \tr(BX^*Y)\rvert^2 = \lvert \tr(YBX^*)\rvert^2 &= \lvert\tr(YRCRX^*)\rvert^2\\
&=\bigl\lvert\tr\bigl[(RY^*)^*(CRX^*)\bigr]\bigr\rvert^2
\le \tr(YR^2Y^*) \tr(XRC^2RX^*).
\end{align*}
As $C^2\preceq I$, it follows that $\tr(Z^* C^2Z) \le \tr(Z^* Z)$ for any $Z\in \Bbbk^{d\times n}$.
Hence
\[
\lvert\tr(BX^*Y)\rvert^2\le \max\bigl\{ \tr(XAX^*)^2,
\tr(YAY^*)^2\bigr\}.
\]
This shows that $\lVert B \rVert_{\G,d}\le \lVert A \rVert_{\G,d}$. As $\lVert B \rVert_{\gamma,d} \le \lVert B \rVert_{\Gamma,d}\le \lVert B \rVert_{\G,d}$ and $\lVert A \rVert_{\G,d} = \lVert A \rVert_{\Gamma,d} = \lVert A \rVert_{\gamma,d}$, the corresponding inequalities for $\lVert\,\cdot\,\rVert_{\gamma,d}$ and $\lVert\,\cdot\,\rVert_{\Gamma,d}$ follow.
\end{proof}

A consequence of \eqref{eq:altGamma} and Proposition~\ref{posdefcase}  is the following characterization of the Grothendieck norms on $\mathbb{S}^n_\p$. Note that a departure from \eqref{eq:normcorrd} is that we do not need to take absolute values of the trace.
\begin{corollary}\label{cor:altGinfty1}
Let $A \in \mathbb{S}^n_\p$ and $d \in \mathbb{N}$. Then
\[
\lVert A \rVert_{\gamma,d} =\lVert A \rVert_{\Gamma,d} = \lVert A \rVert_{\G,d} =\max\{\tr (AG) : G \in \mathbb{G}^n_d\}.
\]
\end{corollary}
For the special case $d = n$,   $\mathbb{G}^n_d=\mathbb{G}^n$ is a closed convex set and Corollary~\ref{cor:altGinfty1} says that the Grothendieck norms reduce to the \emph{support function} \cite[pp.~63 and 120]{Boyd} of  $\mathbb{G}^n$.

The Grothendieck $d$-norms are difficult to compute in general. For one they are not unitarily invariant; we suspect that they are NP-hard for a wide range of $d$'s --- see Conjecture~\ref{conj:NPhard}. To get a handle on these norms, it will often be helpful to examine the few simple cases where they have closed-form expressions. In cases \eqref{it:diag} and \eqref{it:tridiag} below, note that the diagonal of a matrix in $\mathbb{S}^n(\mathbb{C})$ must have real values; in particular, a diagonal matrix in $\mathbb{C}^{n \times n}$ does not necessarily belong to $\mathbb{S}^n(\mathbb{C})$.
\begin{lemma}\label{lem:nonneg}
Let  $d \in \mathbb{N}$ and $\Bbbk =\mathbb{R}$ or $\mathbb{C}$.
\begin{enumerate}[\upshape (i)]
\item\label{it:diag} If $A=\diag(a_{11},\dots,a_{nn})\in \mathbb{S}^n$, then
\[
\lVert A\rVert_{\gamma,d}= \biggl\lvert \sum_{i=1}^n  a_{ii} \biggr\rvert, \qquad
\lVert A\rVert_{\Gamma,d}=\max \biggl( \sum_{i=1}^n \max (a_{ii},0) , -\sum_{i=1}^n \min(a_{ii}, 0) \biggr); \]
and if $B=\diag(b_{11},\dots,b_{nn})\in \Bbbk^{n \times n}$, then
\[
\lVert B\rVert_{\G,d} = \sum_{i=1}^n \lvert b_{ii} \rvert.
\]
\item\label{it:tridiag} If $A\in \mathbb{S}^n$ is a symmetric tridiagonal matrix, then
\[
\lVert A\rVert_{\gamma,d} =  \biggl\lvert \sum_{i=1}^n  a_{ii} \biggr\rvert +2\sum_{i=1}^{n-1} \lvert a_{i,i+1}\vert .
\]

\item\label{it:nonneg} If $A \in \mathbb{S}^n(\mathbb{R}_\p\!)$, then
\[
\lVert A\rVert_{\gamma,d} = \lVert A\rVert_{\Gamma,d}
= \lVert A\rVert_{\G,d} = \sum_{i=1}^n\sum_{j=1}^n a_{ij};
\]
and if $B \in \mathbb{R}^{m \times n}_\p$, then
\[
\lVert B\rVert_{\G,d} = \sum_{i=1}^m\sum_{j=1}^n b_{ij}.
\]
\item\label{it:genbilap}  If $A = \begin{bsmallmatrix}A_1&-B\\-B^\tp&A_2\end{bsmallmatrix} \in \mathbb{S}^{m+n}(\mathbb{R})$ has $A_1\in \mathbb{S}^{m}(\mathbb{R}_\p\!)$, $A_2\in \mathbb{S}^{n}(\mathbb{R}_\p\!)$, and  $B \in \mathbb{R}^{m \times n}_\p$,  then
\[
\lVert A\rVert_{\gamma,d} = \lVert A\rVert_{\Gamma,d} = \lVert A\rVert_{\G,d} 
= \sum_{i=1}^{m+n}\sum_{j=1}^{m+n} |a_{ij}|.
\]
\end{enumerate}
\end{lemma}
\begin{proof}
For a diagonal $A\in\mathbb{S}^n$, $\tr( AG)=\tr( A)$ for any $G\in \mathbb{G}^n_d$ and so $\lVert A \rVert_{\gamma,d}=\lvert \tr(A) \rvert$. For $G\in\bbGamma_d^n$, $\tr( A G)=\sum_{i=1}^n a_{ii}g_{ii}$ for $g_{ii}\in[0,1]$, giving us the required expression for $\lVert A \rVert_{\Gamma,d}$.
 
For a tridiagonal $A\in\mathbb{S}^n$, $\tr(AG )=\tr(A)+2\sum_{i=1}^{n-1}\Re(a_{i,i+1}g_{i+1,i})$ for any $G\in\mathbb{G}_d^n$.  Hence $\lvert \tr( AG )\rvert \le \lvert \tr( A )\rvert +2\sum_{i=1}^{n-1} \lvert a_{i,i+1}\rvert $ for all $G\in\mathbb{G}_d^n$ and so $\lVert A \rVert_{\gamma,d}\le  \lvert \tr( A )\rvert +2\sum_{i=1}^{n-1} \lvert a_{i,i+1}\rvert $.   Replacing $A$ by $-A$ if $\tr( A)<0$, we may assume without loss of generality that $\tr( A)\ge 0$.  Let $\delta_1,\dots,\delta_n\in \mathbb{T}$ and $D = \diag(\delta_1,\dots,\delta_n)$. As $D^*\! AD =(\bar{\delta}_i a_{ij}\delta_j)$, we have $\lVert D^*\! AD\rVert_{\gamma,d} = \lVert A \rVert_{\gamma,d}$.  Replacing $A$ by $D^*\! AD$ if necessary, where
$D=\diag(1,\delta_2,\dots,\delta_n)$ is such that $\bar{\delta}_ia_{i,i+1} \delta_{i+1}\ge 0$, $i=1,\dots,n-1$,  we may assume that the off-diagonal entries of $A$ are nonnegative.  Therefore $\mathbbm{1}^\tp\! A\mathbbm{1} = \tr( A)+2\sum_{i=1}^{n-1}\lvert a_{i,i+1}\rvert $.  Thus $\lVert A \rVert_{\gamma,d}\ge \lvert \tr( A) \rvert+2\sum_{i=1}^{n-1} \lvert a_{i,i+1}\rvert $.

For \eqref{it:genbilap}, by their definitions in \eqref{eq:d},  $\lVert A\rVert_{\gamma,d} \le\lVert A\rVert_{\Gamma,d} \le \lVert A\rVert _{\G,d} \le\sum_{i=1}^{m+n}\sum_{j=1}^{m+n} \lvert a_{ij} \rvert$.  Now observe that if we set $x=(\mathbbm{1}_m,-\mathbbm{1}_n) \in \mathbb{R}^{m +n}$, then $x^\tp\! Ax$ attains this upper bound. The equalities in \eqref{it:nonneg} are special cases of \eqref{it:genbilap}, obtained either by setting $m = 0$ or setting $A_1 = A_2 = 0$.
\end{proof}
For the tridiagonal case \eqref{it:tridiag}, we are unable to obtain similar expressions for $\lVert A\rVert_{\Gamma,d}$ and $\lVert A\rVert_{\G,d}$. The last case  \eqref{it:genbilap} is of interest since the weighted Laplacian of a bipartite graph takes this form with $A_1,A_2$ diagonal matrices. In all cases  above, the norms happen to be independent of $d$, we know of course that this is not true in general. The diagonal case \eqref{it:diag}, while easy to prove, has an important implication --- it shows that all three types of Grothendieck norms are distinct on $\mathbb{S}^n$. In particular, the symmetric Grothendieck inequality \eqref{eq:GGI2} cannot be obtained simply by restricting the original Grothendieck inequality \eqref{eq:GI2} to $\mathbb{S}^n$. 
\begin{corollary}\label{cor:ne}
There exists $A \in \mathbb{S}^n$ such that
\[
\lVert A \rVert_{\gamma,d} < \lVert A \rVert_{\Gamma,d} < \lVert A \rVert_{\G,d}.
\]
\end{corollary}
\begin{proof}
Take $A = \begin{bsmallmatrix} 1 & 0 \\ 0 & -1 \end{bsmallmatrix}$. Then
$\lVert A \rVert_{\gamma,d} =0$, $\lVert A \rVert_{\Gamma,d} =1$, $\lVert A \rVert_{\G,d} = 2$ by Lemma~\ref{lem:nonneg}\eqref{it:diag}.
\end{proof}

\subsection{Real and complex Grothendieck norms}\label{sec:RCnorms}

The value of a matrix norm in general depends on the choice of the base field. This has been discussed for operator norms in \cite{HK05}: Take the $(\infty,1)$-norm and $B = \begin{bsmallmatrix} 1 & -1\\ 1 & 1 \end{bsmallmatrix} \in \mathbb{R}^{2 \times 2}$ for example,
\begin{align*}
\lVert B \rVert_{\infty,1}^\mathbb{R} &= 
\max \Bigl\{\Bigl\lvert\sum\nolimits_{i,j=1}^2 b_{ij}  \bar{\varepsilon}_i\delta_j\Bigr\rvert :  \varepsilon_i, \delta_j \in \{ -1, +1\} \Bigr\} = 2,\\
\lVert B \rVert_{\infty,1}^\mathbb{C} &= 
\max \Bigl\{\Bigl\lvert\sum\nolimits_{i,j=1}^2 b_{ij}  \bar{\varepsilon}_i\delta_j\Bigr\rvert:
\varepsilon_i, \delta_j \in \{ e^{\mathrm{i} \vartheta} : \vartheta \in [0,2\pi)\} \Bigr\} = 2 \sqrt{2}.
\end{align*}
We will show that this is generally the case for the Grothendieck $d$-norms for smaller values of $d$ but not when $d$ is large enough. In particular, the values of the $\G$-norm, $\Gamma$-norm, and $\gamma$-seminorm for a real matrix stay the same regardless of whether we take them over $\mathbb{R}$ or $\mathbb{C}$.

We begin by establishing a precise relation between the real and complex versions of  the Grothendieck $d$-norms in \eqref{eq:d}.
\begin{proposition}\label{prop:RCnorms}
Let $d \in \mathbb{N}$ and $A \in \mathbb{S}^n(\mathbb{R})$. Then
\begin{equation}\label{gammaC}
\lVert A \rVert_{\gamma,d}^\mathbb{C}=\lVert A \rVert_{\gamma,2d}^\mathbb{R}, \qquad
\lVert A \rVert_{\Gamma,d}^\mathbb{C}=\lVert A \rVert_{\Gamma,2d}^\mathbb{R}.
\end{equation}
Let $d \in \mathbb{N}$ and $B\in \mathbb{R}^{m\times n}$.  Then
\begin{equation}\label{eq:GCnorm}
\lVert B \rVert_{\G,d}^\mathbb{C}=\lVert B \rVert_{\G,2d}^\mathbb{R}.
\end{equation}
\end{proposition}
\begin{proof}
Let $z_i\in \mathbb{C}^d$, $\lVert z_i\rVert = 1$, $i=1,\dots,n$. Write $z_i=u_i+\mathrm{i}v_i$ with $u_i,v_i\in\mathbb{R}^d$, $\lVert u_i\rVert^2+\lVert v_i\rVert^2=1$, and set $x_i=(u_i,v_i) \in \mathbb{R}^{2d}$ with $\lVert x_i\rVert = 1$. Let $A \in \mathbb{S}^n(\mathbb{R})$. Then
\[
\sum_{i=1}^n \sum_{j=1}^n a_{ij}\langle z_i, z_j\rangle_\mathbb{C} =
\sum_{i=1}^n \sum_{j=1}^n a_{ij}\Re\langle z_i, z_j\rangle_\mathbb{C} =
\sum_{i=1}^n \sum_{j=1}^n  a_{ij}\langle x_i, x_j\rangle_\mathbb{R}.
\]
Since $x_1,\dots,x_d$ lie in a subspace of $\mathbb{R}^{2d}$ of dimension at most $d$, it follows from the definitions  in \eqref{eq:GammaTheta} and \eqref{eq:seminorms2} that $\lVert A \rVert_{\gamma,d}^\mathbb{C}\le\lVert A \rVert_{\gamma,2d}^\mathbb{R}$.  Conversely, let $x_i \in \mathbb{R}^{2d}$, $\lVert x_i\rVert = 1$, $i =1,\dots,n$, be such that $\lVert A \rVert_{\gamma,2d}^\mathbb{R}= \bigl\lvert \sum_{i=1}^n \sum_{j=1}^n  a_{ij}\langle x_i, x_j\rangle_\mathbb{R} \bigr\rvert$. Write $x_i=(u_i,v_i)\in \mathbb{R}^{2d}$ and set  $z_i=u_i+\mathrm{i}v_i \in \mathbb{C}^d$.  Then 
\[
\lVert A \rVert_{\gamma,2d}^\mathbb{R}=\biggl\lvert \sum_{i=1}^n \sum_{j=1}^n a_{ij}\langle z_i, z_j\rangle_\mathbb{C}\biggr\rvert \le \lVert A \rVert_{\gamma,d}^\mathbb{C}.
\]
Hence the first equality in \eqref{gammaC} holds; and the second equality can be similarly established.

For \eqref{eq:GCnorm}, we start with the characterization of the $\G$-norm in \eqref{eq:norms1}.
Observe that
\begin{equation}\label{eq:cnorm1}
\lVert B\rVert_\G^\mathbb{C} =  \max_{\lVert w_i\rVert = \lVert z_j \rVert = 1}\Re \biggl( \sum_{i=1}^m\sum_{j=1}^n b_{ij} \langle  w_i,z_j\rangle\biggr),
\end{equation}
where $w_1,\dots,w_m,z_1,\dots,z_n \in \mathbb{C}^d$.
Write $w_i=s_i+\mathrm{i}t_i$, $z_j=u_j+\mathrm{i}v_j$ with $s_i,t_i,u_j,v_j\in\mathbb{R}^d$, and $x_i=(s_i,t_i)$, $y_j=(u_j,t_j)\in\mathbb{R}^{2d}$, $i=1,\dots,m$, $j =1,\dots, n$. Note that
\[
\Re\langle w_i,z_j\rangle_\mathbb{C}= s_i^\tp u_j +t_i^\tp v_j = \langle x_i,y_j\rangle_\mathbb{R}. 
\]
Also if $\lVert w_i \rVert=\lVert z_j \rVert=1$, then $\lVert x_i \rVert=\lVert y_j \rVert=1$.
Since $B\in\mathbb{R}^{m\times n}$,
\[
\Re \biggl( \sum_{i=1}^m \sum_{j=1}^n b_{ij}\langle w_i,z_j\rangle_\mathbb{C} \biggr) =
\sum_{i=1}^m \sum_{j=1}^n b_{ij}\Re\langle w_i,z_j\rangle_\mathbb{C} =
\sum_{i=1}^m \sum_{j=1}^n b_{ij} \langle x_i,y_j\rangle_\mathbb{R}.
\]
Taking maximum as in \eqref{eq:cnorm1}, we get \eqref{eq:GCnorm}.
\end{proof}

If we apply \eqref{gammaC} with $d\ge n$ and \eqref{eq:GCnorm} with $d \ge m+n$, then we have the following corollary by \eqref{eq:stabilize} in Lemma~\ref{lem:d}.
\begin{corollary}\label{cor:RCnorms}
Let $A \in \mathbb{S}^n(\mathbb{R})$. Then
\begin{equation}\label{eq:RCnorms1}
\lVert A \rVert_\gamma^\mathbb{R}=\lVert A \rVert_\gamma^\mathbb{C}, \qquad \lVert A \rVert_\Gamma^\mathbb{R}=\lVert A \rVert_\Gamma^\mathbb{C}.
\end{equation}
Let $B\in \mathbb{R}^{m\times n}$.  Then
\begin{equation}\label{GrothAreal}
\lVert B \rVert_\G^\mathbb{R}=\lVert B \rVert_\G^\mathbb{C}.
\end{equation}
\end{corollary}

The analogue of \eqref{GrothAreal} is false for the $(\infty,1)$-norm by the example $B = \begin{bsmallmatrix} 1 & -1\\ 1 & 1 \end{bsmallmatrix} \in \mathbb{R}^{2 \times 2}$ that we saw earlier. By setting $A =  \begin{bsmallmatrix} 0 & B\\ B^* & 0 \end{bsmallmatrix}$ and invoking  \eqref{eqvarnrm}, we get that $\lVert A \rVert_\theta^\mathbb{R} \ne \lVert A \rVert_\theta^\mathbb{C}$ and $\lVert A \rVert_\Theta^\mathbb{R} \ne \lVert A \rVert_\Theta^\mathbb{C}$. So the analogue of  \eqref{eq:RCnorms1} is likewise false for the $\theta$-seminorm and the $\Theta$-norm. However, the special case $d = 1$ in Proposition~\ref{prop:RCnorms} gives us the following characterizations.
\begin{corollary}\label{corARC2}
Let $A \in \mathbb{S}^n(\mathbb{R})$. Then
\begin{equation}\label{thetaC}
\lVert A \rVert_\theta^\mathbb{C} =\lVert A \rVert_{\gamma,2}^\mathbb{R},\qquad
\lVert A \rVert_\Theta^\mathbb{C}=\lVert A \rVert_{\Gamma,2}^\mathbb{R}.
\end{equation}
Let $B\in \mathbb{R}^{m\times n}$.  Then
\begin{equation}\label{infty1C}
\lVert B \rVert_{\infty,1}^\mathbb{C}=\lVert B \rVert_{\G,2}^\mathbb{R}.
\end{equation}
\end{corollary}

On the cone of positive semidefinite matrices $\mathbb{S}^n_\p(\mathbb{R})$ and the cone of nonnegative matrices $\mathbb{S}^n(\mathbb{R}_\p\!)$, Corollary~\ref{cor:altGinfty1} and Lemma~\ref{lem:nonneg}\eqref{it:nonneg} respectively hold true more generally across different fields.
\begin{corollary}\label{cor:nonnegRC}
Let $d \in \mathbb{N}$. If $A \in \mathbb{S}_\p^n(\mathbb{R})$, then
\begin{equation}\label{eq:allequal}
\lVert A \rVert_{\gamma,d}^\mathbb{R} =\lVert A \rVert_{\Gamma,d}^\mathbb{R} = \lVert A \rVert_{\G,d}^\mathbb{R} = \lVert A \rVert_{\gamma,d}^\mathbb{C} =\lVert A \rVert_{\Gamma,d}^\mathbb{C} = \lVert A \rVert_{\G,d}^\mathbb{C}.
\end{equation}
If $B \in \mathbb{R}^{m \times n}_\p$, then
\[
\lVert B\rVert_{\G,d}^\mathbb{R} = \lVert B\rVert_{\G,d}^\mathbb{C}.
\]
If $A \in \mathbb{S}^n(\mathbb{R}_\p\!)$, then not only does \eqref{eq:allequal} hold, its value is the same for all $d \in \mathbb{N}$.
\end{corollary}

\subsection{Stabilization of Grothendieck $d$-norms}\label{sec:stable}

We show here that \eqref{eq:stabilize} can be improved --- the stabilizations $\lVert\,\cdot\,\rVert_{\gamma,d} =   \lVert \,\cdot\,\rVert_\gamma$, $\lVert\,\cdot\,\rVert_{\Gamma,d} =   \lVert\,\cdot\,\rVert_\Gamma$, $\lVert\,\cdot\,\rVert_{\G,d} =   \lVert\,\cdot\,\rVert_\G$ occur for far smaller values of $d$ than what is given in  Lemma~\ref{lem:d}. This will require a more careful argument, beginning with an observation about  extreme points.
\begin{lemma}\label{extpoint}
For any $d \in \mathbb{N}$, the extreme points of the convex hulls $\conv(\mathbb{G}^n_d)$ and $\conv(\bbGamma^n_d)$, viewed as subsets of $\mathbb{S}^n$,
are contained in $\mathbb{G}^n_q$ and $\bbGamma^n_q$ respectively, where
\[
q = 
\begin{cases}
\biggl\lfloor \dfrac{-1+\sqrt{1+8n}}{2} \biggr\rfloor + 1 & \Bbbk = \mathbb{R},\\[2.5ex]
\lfloor \sqrt{n}\rfloor +1 & \Bbbk = \mathbb{C}.
\end{cases}
\]
\end{lemma}
\begin{proof}
We begin with $\Bbbk=\mathbb{R}$ and $\bbGamma^n_d$.  Since  $\bbGamma^n_d$ is compact, $\conv(\bbGamma^n_d)$ is compact, and so an extreme point $G$ of $\conv(\bbGamma^n_d)$ is contained in  $\bbGamma^n_d$.   If $d\le q $, then $G\in \bbGamma^n_q$ and the lemma trivially holds. So suppose $d> q$. Let $r\coloneqq\rank (G)$. If $r \le q < d$, then $\im(G)$ is isometric to an $r$-dimensional subspace $W$ of $\mathbb{R}^q$.  So if $G=G(x_1,\dots,x_n)$, $x_1,\dots,x_n\in\im(G)$, then there exist $y_1,\dots,y_n\in W$ such that $G=G(y_1,\dots,y_n)\in \bbGamma_q^n$.

Now suppose $r>q=\lfloor (-1+\sqrt{1+8n})/2\rfloor + 1$. Since the subspace of matrices $\{ A\in \mathbb{S}^n : \im(A) \subseteq \im(G)\}$ has dimension $r(r+1)/2>n$ and $\dim \mathbb{S}^n_\circ = n(n-1)/2$, their intersection must contain a nonzero $A$, i.e., with $\im(A)\subseteq \im(G)$ and $a_{11}=\dots =a_{nn}=0$.
Let  $G = V_r \Lambda_r V_r^\tp$ be a reduced eigenvalue decomposition with $\lambda_1 \ge \dots \ge \lambda_r > 0$. Since $\im(A) \subseteq \im(G)$,  we must have  $A= V_r B_r V_r^\tp$ for some $B_r\in\mathbb{S}^r$. As
 $\Lambda_r \succ 0$ it follows that there exists $\varepsilon > 0$ such that $\Lambda_r + t B_r\succ 0$  whenever $\lvert t\rvert <\varepsilon$.  Hence $G+ tA\succeq 0$ whenever $\lvert t\rvert <\varepsilon$.  Since $A$ has an all-zero diagonal, $G+ tA\in \bbGamma_d^n$.  As $A\ne 0$, it follows that $G_\p \coloneqq G+tA$ and $G_\m \coloneqq G-tA$ are both not equal to $G$ but yet $G=\frac{1}{2}(G_\p+G_\m)$.  So $G$ is not an extreme point of $\conv(\bbGamma_d^n)$.
The exact same proof applies with $\mathbb{G}^n_d$ and $\mathbb{G}^n_q$ in place of $\bbGamma^n_d$ and $\bbGamma^n_q$.

The arguments for $\Bbbk=\mathbb{C}$ are similar with the following observation: If $\rank (G)=r >q = \lfloor \sqrt{n}\rfloor +1$, then $\{ A\in \mathbb{S}^n(\mathbb{C}) : \im(A) \subseteq \im(G)\}$ has  real dimension $r^2 > n$; the real dimension of $\mathbb{S}_\circ^n(\mathbb{C})$ is $n(n-1)$, and that of $\mathbb{S}^n(\mathbb{C})$ is $n^2$.
\end{proof}

Since the maximum of a convex function over a compact convex set $S$ is attained on its set of extreme points $\extr(S)$, we obtained the following stability result for Grothendieck $d$-norms.
\begin{proposition}\label{eqdnrm}
Let $d,m,n\in \mathbb{N}$.
\begin{enumerate}[\upshape (i)]
\item If $d(d+1)/2>n$, then for all $A \in \mathbb{S}^n(\mathbb{R})$,
\[
\lVert A\rVert_{\gamma,d} = \lVert A\rVert_{\gamma,n} =   \lVert A\rVert_\gamma,\qquad
 \lVert A\rVert_{\Gamma,d}= \lVert A\rVert_{\Gamma,n} =   \lVert A\rVert_\Gamma.
\]
\item If $d(d+1)/2>m+n$, then for all $B \in\mathbb{R}^{m \times n}$,
\[
\lVert B\rVert_{\G,d}= \lVert B\rVert_{\G,n} = \lVert B\rVert_\G.
\]
\item If $d^2>n$, then for all  $A \in \mathbb{S}^n(\mathbb{C})$,
\[
\lVert A\rVert_{\gamma,d} = \lVert A\rVert_{\gamma,n} =   \lVert A\rVert_\gamma,\qquad
 \lVert A\rVert_{\Gamma,d}= \lVert A\rVert_{\Gamma,n} =   \lVert A\rVert_\Gamma.
\]
 \item If $d^2>m+n$, then for all $B \in \mathbb{C}^{m \times n}$,
\[
\lVert B\rVert_{\G,d}= \lVert B\rVert_{\G,n} =   \lVert B\rVert_\G.
\]
\end{enumerate}
\end{proposition}
\begin{proof}
We always have $\lVert A \rVert_{\gamma,d} \le \lVert A \rVert_\gamma$. By \eqref{eq:normcorrd},
\begin{align*}
\lVert A \rVert_\gamma &=  \max\{\lvert \tr (AG) \rvert : G \in \mathbb{G}^n\}
=\max\{\lvert \tr (AG) \rvert : G \in \extr(\mathbb{G}^n)\}\\
&\le \max\{\lvert \tr (AG) \rvert : G \in \mathbb{G}^n_d\}  =\lVert A \rVert_{\gamma,d},
\end{align*}
where the inequality follows from Lemma~\ref{extpoint}. The same argument applies to the $(\Gamma,d)$-norm. The equalities for the $(\G,d)$-norm follows from \eqref{eqvarnrm}.
\end{proof}

The Grothendieck $d$-norms $\lVert\, \cdot\, \rVert_{\gamma,d}$, $\lVert\, \cdot\, \rVert_{\Gamma,d}$, $\lVert\, \cdot\, \rVert_{\G,d}$ are NP-hard over $\mathbb{R}$ when $d =1$; see \eqref{eq:maxcut} for example. We suspect the following:
\begin{conjecture}\label{conj:NPhard}
The following norms are all NP-hard:
\begin{enumerate}[\upshape (i)]
\item $\lVert\, \cdot\, \rVert_{\gamma,d}^\mathbb{R}$ and $\lVert\, \cdot\, \rVert_{\Gamma,d}^\mathbb{R}$ on $\mathbb{S}^n(\mathbb{R})$ when $d(d+1)/2\le n$;
\item $\lVert\, \cdot\, \rVert_{\G,d}^\mathbb{R}$ on $\mathbb{R}^{m \times n}$ when $d(d+1)/2\le m+n$;
\item $\lVert\, \cdot\, \rVert_{\gamma,d}^\mathbb{C}$ and $\lVert\, \cdot\, \rVert_{\Gamma,d}^\mathbb{C}$ on $\mathbb{S}^n(\mathbb{C})$ when $d^2\le n$;
\item $\lVert\, \cdot\, \rVert_{\G,d}^\mathbb{C}$ on $\mathbb{C}^{m \times n}$ when $d^2\le m+n$.
\end{enumerate}
\end{conjecture}
It is perhaps worthwhile to point out that $\bbGamma^n_d$ and $\mathbb{G}^n_d$ are nonconvex sets when $d < n$, thus computing $\lVert\, \cdot\, \rVert_{\gamma,d}$ and $\lVert\, \cdot\, \rVert_{\Gamma,d}$ involve the nonconvex optimization problems in \eqref{eq:normcorrd}. 
On the other hand, when $d \ge n$,  $\bbGamma^n_d = \bbGamma^n$ and $\mathbb{G}^n_d = \mathbb{G}^n$ are convex sets and the optimization problems in \eqref{eq:normcorrd} are just semidefinite programs. 
An implication of Proposition~\ref{eqdnrm} is that when $d$ is in the range $\lfloor (-1+\sqrt{1+8n})/2\rfloor + 1 \le d < n$ (for $\mathbb{R}$) or $\lfloor \sqrt{n}\rfloor + 1 \le d < n$ (for $\mathbb{C}$), the nonconvex optimization problems  in \eqref{eq:normcorrd} are solvable in polynomial-time.  For the case of $\lVert\, \cdot\,\rVert_{\gamma,d}^\mathbb{R}$, $d(d+1)/2>n$, this follows from \cite[Theorem~2.1]{lemon}, a result that originally appeared in \cite{barvinok, pataki}.

\section{Gaussian integrals of sign functions}\label{sec:gauss}

We will introduce a function  $\varphi_d^\Bbbk : \mathbb{D} \to \Bbbk$, where $\mathbb{D}$ is the closed unit disc in $\Bbbk = \mathbb{R}$ or $\mathbb{C}$ and $d \in \mathbb{N}$, that will play a critical role in the proofs of all Grothendieck-like inequalities in this article.
The function  first appeared  in \cite{Krivine2} for $d =1$ and $\Bbbk =\mathbb{R}$, in \cite{Haagerup} for $d =1$ and $\Bbbk =\mathbb{C}$, was generalized to arbitrary $d \ge 1$ and $\Bbbk = \mathbb{R}$ in  \cite[Equation~2]{Bri1}, and will be further extended to arbitrary $d \ge 1$ and $\Bbbk = \mathbb{C}$ in this section. For our later purposes, we will also need to establish a few facts about $\varphi_d^\mathbb{R}$ not found in \cite{Bri1}.

For any $d,n\in\mathbb{N}$, the Gaussian function on $\Bbbk^{d \times n}$ is defined by
\[
G_{d,n}^\Bbbk(Z) \coloneqq \begin{cases}
(2\pi)^{-dn/2}\exp\bigl(- \lVert Z \rVert^2/2\bigr) &\text{if}\; \Bbbk=\mathbb{R},\\
\pi^{-dn} \exp(-\lVert Z \rVert^2) &\text{if}\; \Bbbk=\mathbb{C},
\end{cases}
\]
where  $Z\in \Bbbk^{d\times n}$ and $\lVert \,\cdot\, \rVert$ is the Frobenius or Hilbert--Schmidt norm. The special case $d=1$ will be denoted
\[
G_n^\Bbbk(z)=\begin{cases}
(2\pi)^{-n/2}\exp\bigl(- \lVert z \rVert^2/2\bigr) &\text{if}\; \Bbbk=\mathbb{R},\\
\pi^{-n} \exp(-\lVert z \rVert^2) &\text{if}\; \Bbbk=\mathbb{C},
\end{cases}
\]
where $z \in \Bbbk^n$.
We will extend the definition of sign function in \eqref{eq:sign} to a vector variable $z \in \Bbbk^n$ by
\[
\sign z \coloneqq \begin{cases}
z/\lVert z \rVert & z \ne 0,\\
e_1 & z = 0,
\end{cases}
\]
where $e_1 = (1,0,\dots,0) \in \Bbbk^n$. For each $d \in \mathbb{N}$, we define the function $\varphi_d^\Bbbk : \mathbb{D} \to \Bbbk$ by
\begin{equation}\label{defvarphic}
\varphi_d^\Bbbk(\langle u,v \rangle)  \coloneqq \int_{\Bbbk^{d \times n}} \!\! \bigl\langle \sign(Zu), \sign(Zv)\bigr\rangle G_{d,n}^\Bbbk(Z)\,dZ,
\end{equation}
defined for $u,v\in \Bbbk^n$, $\lVert u\rVert =\lVert v\rVert=1$.  The integral depends on $u,v$ only through their inner product $\langle u, v\rangle$ and does not depend on $n$ as long as $n \ge 2$, which explains the notation on the left of \eqref{defvarphic}. To see this, observe that the integral is invariant under a change-of-variable $Z \mapsto ZQ$  by any orthogonal/unitary $Q \in \Bbbk^{n \times n}$; choosing $Q$ so that $Qu=(1,0,0,\dots,0)$ and $Qv=(x,y,0,\dots,0)$ where $x=\langle u, v\rangle$, $y=\sqrt{1-\lvert\langle u, v\rangle\rvert^2}$, the integral becomes
\begin{equation}\label{eq:dblint}
\int_{\Bbbk^n} \!\int_{\Bbbk^n} \!
\frac{\bigl\langle z_1, \langle u, v\rangle z_1 +\sqrt{1-|\langle u, v\rangle|^2}z_2\bigr\rangle}{\lVert z_1 \rVert\|\langle u, v\rangle z_1 +\sqrt{1-|\langle u, v\rangle|^2}z_2\|} G_n^\Bbbk(z_1)G_n^\Bbbk(z_2)\, dz_1 dz_2,
\end{equation}
which evidently depends on $u,v$ only through $\langle u, v\rangle$.

In the following, we will drop the superscript $\Bbbk$ when a statement holds for both $\mathbb{R}$ and $\mathbb{C}$. It is straightforward to see from \eqref{eq:dblint} that  if $\lVert v\rVert=1$, then $\varphi_d(1) = \varphi_d (\langle v,v \rangle)=1$; and if  $\langle u,v\rangle=0$, then $\varphi_d(0) = \varphi_d(\langle u,v \rangle)=0$, as $\langle z_1,-z_2\rangle=-\langle z_1,z_2\rangle$. By \eqref{defvarphic}, $\varphi_d(-\langle u,v \rangle) =  \varphi_d(\langle u,-v \rangle) = - \varphi_d(\langle u,v \rangle)$; in particular, $\varphi_d^\mathbb{R}$ is an odd function over $\mathbb{R}$ and so all its even degree Taylor coefficients  are zero. In fact, we can say a lot more about these coefficients.
\begin{lemma}\label{varphicCform}
Let $d\in \mathbb{N}$ and $\varphi_d^\Bbbk$ be as defined in \eqref{defvarphic} for $\Bbbk = \mathbb{R}$ or $\mathbb{C}$.
\begin{enumerate}[\upshape (i)]
\item\label{varphicCform1} Let $b_{k,d}$ be the $k$th Taylor coefficient of $\varphi_d^\mathbb{R}$. Then
\[
b_{k,d} \ge 0, \qquad b_{2k,d} = 0, \qquad \sum_{k=0}^\infty b_{k,d} =1,
\]
for all $k =0,1,2,\dots.$

\item\label{varphicCform2} For all   $x \in \mathbb{C}$, $\lvert x \rvert \le 1$,
\[
\varphi_d^\mathbb{R}(x) = \sum_{k=0}^\infty b_{2k+1,d} \, x^{2k+1},\qquad
\varphi_d^\mathbb{C}(x) = \sum_{k=0}^\infty b_{2k+1, 2d} \, x \lvert x\rvert^{2k}.
\]
In particular, for all $x \in \mathbb{R}$, $\lvert x \rvert \le 1$,
\[
\varphi_d^\mathbb{C}(x)=\varphi_{2d}^\mathbb{R}(x).
\]

\item\label{varphiexplform} For all $x \in \mathbb{R}$, $\lvert x \rvert \le 1$,
\[
\varphi_d^\mathbb{R}(x) =\frac{2(1-x^2)^{d/2}}{\sqrt{\pi}}\sum_{k=0}^\infty 
\dfrac{2^{2k}\Gamma\bigl((d+2k+1)/2\bigr)^2\Gamma\bigl((2k+3)/2\bigr)}{(2k+1)!\Gamma\bigl(d/2\bigr)\Gamma\bigl((d+2k+2)/2\bigr)}
x^{2k+1}.
\]

\item\label{varphiexplformC} For all $x \in \mathbb{C}$, $\lvert x \rvert \le 1$,
\[
\varphi_d^\mathbb{C}(x) =\frac{2(1- \lvert x \rvert^2)^d}{\sqrt{\pi}}\sum_{k=0}^\infty 
\dfrac{2^{2k}\Gamma\bigl((2d+2k+1)/2\bigr)^2\Gamma\bigl((2k+3)/2\bigr)}{(2k+1)!\Gamma(d)\Gamma\bigl((2d+2k+2)/2\bigr)}
x \lvert x\rvert^{2k}.
\]
\end{enumerate}
\end{lemma}
\begin{proof}  
Let $d \in \mathbb{N}$ and $u_1,\dots,u_d\in \Bbbk^n$ with $\lVert u_1 \rVert = \dots = \lVert u_d \rVert =1$. Then the matrix
$\bigl(\varphi_d^\Bbbk(\langle u_i,u_j\rangle)\bigr) \in \mathbb{S}^d$ is positive semidefinite as
\[
\sum_{i=1}^d\sum_{j=1}^d \bar{a}_i \varphi_d^\Bbbk(\langle u_i,u_j\rangle) a_j=
\int_{\Bbbk^{d \times n}} \!\! \biggl\langle \sum_{i=1}^d a_i\sign(Zu_i), \sum_{j=1}^d a_j\sign(Zu_j)\biggr\rangle G_{d,n}^\Bbbk(Z)\,dZ\ge 0
\]
for any $a \in \Bbbk^d$.

Let $d\in \mathbb{N}$ be fixed. By the real and complex Schoenberg theorem \cite{kernel}, we must have power series of the forms:
\[
\varphi_d^\mathbb{R}(x)=\sum_{k=0}^\infty b_k x^k,\quad b_k \ge 0, \qquad \text{and}\qquad
\varphi_d^\mathbb{C}(x)=\sum_{j=0}^\infty\sum_{k=0}^\infty c_{jk} x^j  \bar{x}^k,\quad c_{jk} \ge 0,
\]
for all $j, k\in\mathbb{N} \cup \{0\}$.
As $\varphi_d^\Bbbk(0)=0$, $\varphi_d^\Bbbk(1)=1$, and $\varphi_d^\mathbb{R}$ is odd, we obtain
\[
b_0 =0, \qquad b_{2k}=0, \qquad c_{00}=0, \qquad
\sum_{k=0}^\infty b_k =1, \qquad  \sum_{j=0}^\infty\sum_{k=0}^\infty c_{jk} =1.
\]
This gives \eqref{varphicCform1} and thus the expansion for $\varphi_d^\mathbb{R}$ in \eqref{varphicCform2}. We will show that the double series for $\varphi_d^\mathbb{C}$ can be written as a series of the form in \eqref{varphicCform2}. For any $\zeta\in \mathbb{C}$, $\lvert \zeta \rvert = 1$,  we have $\varphi_d^\mathbb{C}(\zeta x) = \zeta\varphi_d^\mathbb{C}(x)$ and thus
\[
\varphi_d^\mathbb{C}(x)=\bar\zeta\varphi_d^\mathbb{C}(\zeta x)=\bar\zeta  \sum_{j=0}^\infty\sum_{k=0}^\infty c_{jk}x^j\bar x^k \zeta^j\bar\zeta^k.
\]
As usual, write the complex variable $x=s+\mathrm{i}t$ with real variables $s,t$, and set
\[
\frac{\partial}{\partial x} \coloneqq \frac{1}{2}\Bigl(\frac{\partial}{\partial s} -\mathrm{i}\frac{\partial}{\partial t}\Bigr), \qquad \frac{\partial}{\partial \bar x} \coloneqq \frac{1}{2}\Bigl(\frac{\partial}{\partial s} +\mathrm{i}\frac{\partial}{\partial t}\Bigr).
\]
Since $\sum_{j=0}^\infty\sum_{k=0}^\infty c_{jk} =1$, it follows that $\varphi_d^\mathbb{C}(s+\mathrm{i}t)$ is an analytic function of $s,t$ in the unit disc $\lvert s\rvert^2+\lvert t \rvert^2\le 1$.
As $\partial x/\partial x= \partial \bar x/\partial \bar x=1$ and $\partial x/\partial \bar x= \partial \bar x/\partial x=0$,
\[
\Bigl(\frac{\partial}{\partial x}\Bigr)^j \Bigl(\frac{\partial}{\partial \bar x}\Bigr)^k \varphi_d^\mathbb{C}(0)=j!k! c_{jk}=j!k! c_{jk}\bar\zeta \zeta^j\bar\zeta^k
\]
for any $j,k \in \mathbb{N} \cup \{0\}$ and $\lvert \zeta \rvert=1$. Hence $c_{jk}=0$ whenever $j-k\ne 1$ and we get that
\[
\varphi_d^\mathbb{C}(x)
= x \sum_{k=0}^\infty c_k (x \bar{x})^k = x \sum_{k=0}^\infty c_k \lvert x \rvert^{2k},
\]
where $c_k \coloneqq  c_{k+1,k}$. It remains to show that $c_k$ is exactly the $k$th Taylor coefficient of $\varphi_{2d}^\mathbb{R}$.  This will follow from the uniqueness of Taylor coefficients and $\varphi_d^\mathbb{C}(x)=\varphi_{2d}^\mathbb{R}(x)$ for all $x \in [-1,1]$, which we will establish next.

Let $u, v\in \mathbb{R}^n$. Then $\langle u,v\rangle= \langle v,u\rangle$.  
As in the discussion before \eqref{eq:dblint}, we may assume that $n=2$ and $u=(1,0)$, $v=(x,y)\in \mathbb{R}^2$, where $x=\langle u, v\rangle$, $y=\sqrt{1-x^2} \in \mathbb{R}$.  Then
\begin{align*}
\varphi_d^\mathbb{C}(x)&=\frac{1}{\pi^{2d}}\int_{\mathbb{C}^d}\!\int_{\mathbb{C}^d} \!
\frac{\langle z_1, x z_1 +y z_2\rangle_\mathbb{C}}{\lVert z_1 \rVert\lVert x z_1 +y z_2 \rVert} e^{-\lVert z_1 \rVert^2-\lVert z_2 \rVert^2}\, dz_1 dz_2\\
&=\frac{1}{\pi^{2d}}\int_{\mathbb{C}^d}\!\int_{\mathbb{C}^d} \!
\frac{\Re\langle z_1, x z_1 +y z_2\rangle_\mathbb{C}}{\lVert z_1 \rVert\lVert x z_1 +y z_2 \rVert} e^{-\lVert z_1 \rVert^2-\lVert z_2 \rVert^2}\, dz_1 dz_2\\
&=\frac{1}{(2\pi)^{2d}}\int_{\mathbb{R}^{2d}}\!\int_{\mathbb{R}^{2d}} \!
\frac{\langle w_1, x w_1 +yw_2\rangle_\mathbb{R}}{\|w_1\|\|x w_1 +y w_2\|} e^{-(\|w_1\|^2+\|w_2\|^2)/2}\, dw_1 dw_2 =\varphi_{2d}^\mathbb{R}(x).
\end{align*}
The last step follows from replacing $z_i = s_i+\mathrm{i}t_i \in\mathbb{C}^d$ by $w_i=(s_i,t_i)\in \mathbb{R}^{2d}$, $i=1,2$, and a change-of-variables $(w_1,w_2) \mapsto (w_1/\sqrt{2}, w_2/\sqrt{2})$. 

To prove \eqref{varphiexplform}, we may again assume $n=2$ and $u=(1,0)$, $v=(x,y)\in \mathbb{R}^2$, where $x=\langle u, v\rangle$ and $y=\sqrt{1-x^2} > 0$. Let $z_1,z_2\in\mathbb{R}^d$  and  $w = xz_1+yz_2\in \mathbb{R}^d$.  Then
\[
z_2=-\frac{x}{y}z_1 + \frac{1}{y} w, \qquad \lVert z_2 \rVert^2= \frac{x^2}{y^2}\lVert z_1 \rVert^2 + \frac{1}{y^2}\lVert w \rVert^2 -\frac{2x}{y^2}\langle z_1, w\rangle.
\]
With a change-of-variable $z_2\mapsto w$, $dz_2=y^{-d} dw$,
\begin{align*}
\varphi_d^\mathbb{R}(x)&=\frac{1}{(2\pi)^{d}}\!\int_{\mathbb{R}^d}\!\int_{\mathbb{R}^d} \!\Bigl\langle \frac{z_1}{\lVert z_1 \rVert},\frac{xz_1+yz_2}{\lVert xz_1+yz_2 \rVert}\Bigr\rangle e^{-(\lVert z_1 \rVert^2+\lVert z_2 \rVert^2)/2}\, dz_1dz_2\\
&=\frac{1}{(2\pi y)^{d}}\!\int_{\mathbb{R}^d}\!\int_{\mathbb{R}^d} \! \Bigl\langle \frac{z_1}{\lVert z_1 \rVert},\frac{w}{\lVert w \rVert}\Bigr\rangle e^{-(\lVert z_1 \rVert^2+\lVert w \rVert^2)/2y^2} e^{(x/y^2)\langle z_1, w\rangle}\, dz_1dw\\
&= \frac{y^{d}}{(2\pi)^{d}}\!\int_{\mathbb{R}^d}\!\int_{\mathbb{R}^d} \!\Bigl\langle \frac{z_1}{\lVert z_1 \rVert},\frac{w}{\lVert w \rVert}\Bigr\rangle e^{-(\lVert z_1 \rVert^2+\lVert w \rVert^2)/2} e^{x\langle z_1, w\rangle} \, dz_1dw,
\intertext{where the last expression follows from another change-of-variables $(z_1,z_2) \mapsto (z_1/y,z_2/y)$, and upon Taylor expanding $e^{x\langle z_1,w\rangle}$ becomes}
&=\frac{(1-x^2)^{d/2}}{(2\pi)^d}\sum_{k=0}^\infty \frac{ x^k}{k!}\!\int_{\mathbb{R}^d}\!\int_{\mathbb{R}^d} \!\!\frac{\langle z_1,w \rangle^{k+1}}{\lVert z_1 \rVert\lVert w \rVert} e^{-(\lVert z_1 \rVert^2+\lVert w \rVert^2)/2} \, dz_1dw.
\end{align*}
Let $I_k$ denote the double integral above.
For even $k$, as $\langle -z_1,w \rangle^{k+1}=-\langle z_1,w \rangle^{k+1}$, we get $I_k = 0$.  For odd $k$, we introduce polar coordinates on each copy of $\mathbb{R}^d$  as in Appendix~\ref{appendix}, with $\lVert z_1 \rVert=\rho_1$, $\lVert w \rVert=\rho_2$, and
\[
I_k =
\int_0^\infty\!\!\!\int_0^\infty\!\!\!\int_{\lVert z_1 \rVert=\rho_1}\!\int_{\lVert w \rVert=\rho_2} \!\!\!\langle\rho_1^{-1}z_1,\rho_2^{-1}w\rangle^{k+1} e^{-(\rho_1^2+\rho_2^2)/2}\rho_1^{d+k-1}\rho_2^{d+k-1} \,d\rho_1 d\rho_2d\sigma_1d\sigma_2.
\]
As the surface area of the $(d-1)$-sphere in $\mathbb{R}^d$ is $2\pi^{d/2}/\Gamma(d/2)$,
by \eqref{eq:moments} in Lemma~\ref{lem:expectations}, we have
\[
\frac{\Gamma(d/2)^2}{4\pi^d}\!\int_{\lVert z_1 \rVert=\rho_1}\!\int_{\lVert w \rVert=\rho_2} \langle\rho_1^{-1}z_1,\rho_2^{-1}w\rangle^{k+1}\,d\sigma_1d\sigma_2= \dfrac{1}{\sqrt{\pi}}\dfrac{\Gamma\bigl(d/2\bigr)\Gamma\bigl((k+2)/2\bigr)}{\Gamma\bigl((d+k+1)/2\bigr)};
\]
and a direct calculation gives
\[
\int_0^\infty\!\!\!\int_0^\infty\!\! \rho_1^{d+k-1}\rho_2^{d+k-1}e^{-\rho_1^2/2}e^{-\rho_2^2/2}\,d\rho_1 d\rho_2
=\biggl[\int_0^\infty \rho^{d+k-1}e^{-\rho^2/2}d\rho\biggr]^2
=2^{d+k-2}\Gamma((d+k)/2)^2.
\]
Taken together, we obtain the value of $I_k$ for odd $k$ and thus \eqref{varphiexplform}.
\end{proof}
The Taylor coefficients $b_{k,d}$ in Lemma~\ref{varphicCform}\eqref{varphicCform2} may be obtained from  Lemma~\ref{varphicCform}\eqref{varphiexplform} by expanding $(1-x^2)^{d/2}$ as a power series ($d$ odd) or a polynomial  ($d$ even).

The case $d =1$ of \eqref{defvarphic} will be particularly important for us  and requires special attention. For the rest of this article, we will write
\[
\varphi_\mathbb{R} \coloneqq \varphi_1^\mathbb{R}, \qquad \varphi_\mathbb{C} \coloneqq \varphi_1^\mathbb{C}, \qquad \varphi_\Bbbk \coloneqq \varphi_1^\Bbbk.
\]
In this case, the power series in Lemma~\ref{varphicCform}\eqref{varphicCform2} have the following forms:
\begin{align}
\varphi_\mathbb{R}(x) &= \frac{2}{\pi}\sum_{k=0}^\infty\frac{(2k-1)!!}{(2k)!!(2k+1)} x^{2k+1} = \dfrac{2}{\pi}\arcsin x,\label{eq:haa}\\
\varphi_\mathbb{C}(x) &= \frac{\pi}{4}\sum_{k=0}^\infty\left(\frac{(2k-1)!!}{(2k)!!}\right)^2 \frac{1}{k+1}x\lvert x\rvert^{2k} =  \dfrac{\pi}{4}x\;  {}_{2}F_{1}\Bigl(\frac{1}{2}, \frac{1}{2}; 2; \lvert x \rvert^2 \Bigr) .\label{defphiC}
\end{align}
As we noted earlier, $\varphi_\mathbb{C}$ was first introduced by Haagerup \cite{Haagerup}, who also showed that it has an integral expressions on the complex unit disk $\lvert x \rvert \le 1$ given by
\begin{equation}\label{eq:int}
\varphi_\mathbb{C}(x)= x\int_0^{\pi/2} \!\!\!
\frac{\cos^2t}{(1-\lvert x \rvert^2\sin^2t)^{1/2}}\,dt.
\end{equation}

The reason $\varphi_\Bbbk$ makes an appearance in the Grothendieck inequality is a result of the following identity \eqref{2-eq1} on the `Gaussian inner product of the sign of Hermitian inner products' \cite{Jameson, Haagerup}. The other two related identities \eqref{AN6} and \eqref{AN7} are due to Alon and Naor in \cite{Alon3} for $\Bbbk = \mathbb{R}$; we will fill-in the proof for $\Bbbk = \mathbb{C}$ here following their ideas.
\begin{lemma}\label{2-thm1}
For any $u,v \in \Bbbk^{n}$,
\begin{align}
\int_{\Bbbk^n} \!\! \langle u,z\rangle \langle z,v\rangle G_n^\Bbbk(z)\,dz &= \langle u, v \rangle. \label{AN6}
\intertext{If in addition $\lVert u \rVert=\lVert v \rVert=1$, then}
\int_{\Bbbk^n} \!\! \langle u,z\rangle \sign \langle z,v\rangle G_n^\Bbbk(z)\,dz &=
\begin{cases}
\sqrt{2/\pi} \, \langle u, v \rangle & \text{if}\;  \Bbbk = \mathbb{R},\\
\sqrt{\pi/4} \, \langle u, v \rangle & \text{if}\;  \Bbbk = \mathbb{C},
\end{cases} \label{AN7}
\intertext{and}
\int_{\Bbbk^n} \!\! \sign\langle u,z\rangle \sign\langle z,v\rangle G_n^\Bbbk(z)\,dz &=
\varphi_\Bbbk(\langle u, v \rangle). \label{2-eq1}
\end{align}
\end{lemma}

\begin{proof}
Setting $d=1$  in \eqref{defvarphic} gives us \eqref{2-eq1}.
The identities \eqref{AN6} and \eqref{AN7} for $\Bbbk = \mathbb{R}$ appeared in \cite[Equations~4.2 and 4.3]{Alon3}. For $\Bbbk = \mathbb{C}$, since
\[
\frac{1}{\pi^n}\int_{\mathbb{C}^n} \!\! z_i \bar{z}_j  e^{-\lVert z \rVert^2} \,dz =\delta_{ij},
\]
we obtain  \eqref{AN6}  via
\[
\frac{1}{\pi^n}\int_{\mathbb{C}^n} \!\! \langle u, z\rangle \langle z, v\rangle e^{-\lVert z \rVert^2} \,dz = \frac{1}{\pi^n}\int_{\mathbb{C}^n} \! \Bigl( \sum\nolimits_{i=1}^n \overline{u}_i z_i\Bigr)\Bigl(\sum\nolimits_{j=1}^n \bar{z}_j v_j \Bigr) e^{-\lVert z \rVert^2} \,dz
=\sum_{i=1}^n \overline{u}_i v_i=\langle u, v\rangle.
\]
Let $\lVert u \rVert=\lVert v \rVert=1$. We claim that
\begin{equation}\label{AN7C}
\frac{1}{\pi^n}\int_{\mathbb{C}^n} \!\!  \langle u, z\rangle \sign  \langle z, v\rangle e^{-\lVert z \rVert^2} \,dz =\frac{\sqrt{\pi}}{2} \langle u, v\rangle.
\end{equation}
Since both sides of \eqref{AN7C} are invariant if $u,v$ are replaced by $Qu, Qv$ for any unitary $Q\in\mathbb{C}^{n\times n}$, we may assume that $v=(1,0,\dots,0)$. Furthermore, since $u=(u_1,\dots,u_n)$ may be multiplied by $\zeta\in\mathbb{C}$, $\lvert \zeta \rvert =1$, so that
$\zeta u_1\ge 0$,  we may assume  that  $u_1\ge 0$. Hence
\[
\langle u,v\rangle =u_1,\quad \langle u,z\rangle= \sum_{i=1}^n \overline{u}_i z_i, \quad \langle z,v\rangle=\bar{z}_1=re^{\mathrm{i}\vartheta}, \quad \sign \langle z,v\rangle=e^{\mathrm{i}\vartheta}
\]
for some $r >0$, $\vartheta \in [0,2\pi)$.
Write $z=x+\mathrm{i}y$ where $x,y \in \mathbb{R}^n$ and set  $w=(x,y)\in\mathbb{R}^{2n}$.
Then the left side of \eqref{AN7C} becomes
\begin{equation}\label{AN7C2}
\frac{1}{\pi^n}\int_{\mathbb{R}^{2n}}\!
\Bigl(\sum\nolimits_{i=1}^n \overline{u}_i z_i\Bigr)\sign (\bar{z}_1) e^{-\lVert x \rVert^2 - \lVert y \rVert^2} dx_n\, dy_n\cdots dx_1\, dy_1.
\end{equation}
By Fubini, we may integrate first with respect to $dx_i dy_i$ to see that
\[
\int_{\mathbb{R}^{2n}}\!\!\! z_i \sign (\bar{z}_1)  e^{-\lVert x \rVert^2 - \lVert y \rVert^2} dx_n\, dy_n\cdots dx_1\, dy_1=0, \qquad i=2,\dots,n,
\]
and so the integral in \eqref{AN7C2} simplifies as
\begin{align*}
&\frac{1}{\pi^n} \int_{\mathbb{R}^{2n}}\!\!\! u_1 z_1\sign (\bar{z}_1)e^{-\lVert x \rVert^2 - \lVert y \rVert^2} dx_n\, dy_n\cdots dx_1\, dy_1\\
&\qquad =\frac{u_1}{\pi} \int_{\mathbb{R}^2}\lvert z_1 \rvert e^{-\lvert z_1 \rvert^2} dx_1\, dy_1
=\frac{u_1}{\pi}\int_0^{\infty}\!\!\!\int_0^{2\pi}\!\! r^2e^{-r^2}\, dr \, dt
= \frac{\sqrt{\pi}}{2} u_1,
\end{align*}
which gives the right side of \eqref{AN7C} as $u_1 =\overline{u}_1 = \langle u,v\rangle$.
\end{proof}

For a function $\varphi : S \to S$, we write $\varphi^{\circ k} = \varphi \circ \dots \circ \varphi$ for the $k$-fold composition of $\varphi$ with itself. We now establish an analog of \cite[Lemma~3.5]{Haagerup} for our use later.
\begin{lemma}\label{Halem3.5} 
\begin{enumerate}[\upshape (i)]
\item\label{it:PhiRe} The function $\varphi_\mathbb{R}$ is a homeomorphism of the closed  interval $[-1,1]$. Its fixed points are
\[
\{ x \in [-1,1] : \lvert x \rvert = 0 \text{ or }1\}.
\]
If $0 < \lvert x \rvert < 1$, then
\[
\lim_{k\to\infty} \varphi_\mathbb{R}^{\circ k}(x)=0.
\]

\item\label{it:phiC} The function $\varphi_\mathbb{C}$ is a homeomorphism of the closed unit disk $\mathbb{D} =\{ re^{\mathrm{i}\vartheta} \in \mathbb{C} : r\in[0,1],\; \vartheta \in [0,2\pi)\}$ that maps each line segment $[0, e^{\mathrm{i}\vartheta}] \coloneqq \{ re^{\mathrm{i}\vartheta} : r\in[0,1]\}$ to itself  for any  $\vartheta \in [0,2\pi)$. Its fixed points are
\[
\{z\in \mathbb{D} : \lvert z \rvert = 0\text{ or }1 \} = \mathbb{T} \cup \{0\}.
\]
If $0 < \lvert z \rvert < 1$, then
\[
\lim_{k\to\infty} \varphi_\mathbb{C}^{\circ k}(z)=0.
\]
\end{enumerate}
\end{lemma}
\begin{proof}
\begin{enumerate}[\upshape (i)]
\item Let $\varphi =\varphi_\mathbb{R}$. By \eqref{eq:int}, it is clearly a homemorphism of $[-1,1]$ with $\varphi(0)=0$, $\varphi(1)=1$, and $\varphi(-x)=-\varphi(x)$ for $x\in[1,-1]$.  
As $\varphi'(x)>0$ and $\varphi''(x)>0$ for $x\in(0,1)$, $\varphi$ is strictly increasing and strictly convex on $[0,1]$.  Thus $\varphi(x)<x$ for $x\in (0,1)$.  Fix $x\in (0,1)$. The sequence $y_k=\varphi^{\circ k}(x)$ is strictly decreasing and bounded below by $0$.  Let $y=\lim_{k\to\infty}y_k$.  Then $\varphi(y)=y$ and so $y=0$.  Now fix $x\in(-1,0)$. Since $\varphi(-x)=-\varphi(x)$, $y_k =\varphi^{\circ k}(x)$ is a strictly increasing sequence converging to $0$.

\item  Let $\varphi =\varphi_\mathbb{C}$ and let $h$ be the restriction of $\varphi$ to $[0,1]$. By \eqref{eq:int}, $h(0)=0$ and $h(1)=1$.  As $h'(x)>0$  for $x\in(0,1)$, $h$ is strictly increasing on $[0,1]$.  Thus $h$ is a homeomorphism of $[0,1]$. Since $\varphi(z)=\sign (z) h(\lvert z \rvert)$, it follows that $\varphi$ is a homeomorphism of $\mathbb{D}$ taking each $[0,e^{\mathrm{i}\theta}]$ to itself.

As $h''(x)>0$  for $x\in(0,1)$, $h$ is strictly convex on $[0,1]$. Thus $0<h(x)<x$ for $x\in (0,1)$ and so $0$ and $1$ are its unique fixed points on $[0,1]$.   Since $\varphi(z)=\sign (z) h(\lvert z \rvert)$, it follows that the fixed points of $\varphi$ are either the origin or on the unit circle.

As in case \eqref{it:PhiRe}, we have $\lim_{k\to \infty} h^{\circ k}(x)=0$ for any $x\in(0,1)$. Since  $\varphi(z)=\sign (z) h(\lvert z\rvert )$, we must also have $\lim_{k\to\infty} \varphi^{\circ k}(z)=0$ whenever $\lvert z\rvert \in (0,1)$.\qedhere
\end{enumerate}
\end{proof}

\section{Symmetric Grothendieck inequality with effective bounds}\label{sec:bounds}

In this section we will establish the two versions of symmetric Grothendieck inequalities mentioned earlier:
\begin{equation}\label{eq:SGI}
\lVert A \rVert_\gamma \le K_\gamma \lVert A \rVert_\theta, \qquad \lVert A \rVert_\Gamma \le K_\Gamma \lVert A \rVert_\Theta
\end{equation}
over both $\mathbb{R}$ and $\mathbb{C}$ with explicit bounds for the constants.

%
%

\begin{lemma}\label{lem:Phi}
Let  $\Bbbk = \mathbb{R}$ or $\mathbb{C}$. Let $\mathbb{D}=\{z\in \Bbbk: \lvert z \rvert \le 1\}$ and $\mathbb{S}_\eq^n(\mathbb{D}) =\{A \in \mathbb{S}_\eq^n :  a_{ij} \in \mathbb{D}\}$. The map
\begin{equation}\label{eq:Phi}
\Phi_\Bbbk: \mathbb{S}_\eq^n(\mathbb{D}) \to \mathbb{S}_\eq^n(\mathbb{D}), \quad (a_{ij}) \mapsto \bigl(\varphi_\Bbbk(a_{ij})\bigr),
\end{equation}
is a homeomorphism of $\mathbb{S}_\eq^n(\mathbb{D})$.  Its fixed points are
\[
\{A \in \mathbb{S}_\eq^n(\mathbb{D}) : \lvert a_{ij}\rvert = 0\text{ or }1 \text{ for all } i,j = 1,\dots,n\}.
\]
\end{lemma}
\begin{proof}
For $\Bbbk = \mathbb{R}$ or $\mathbb{C}$, $\varphi_\Bbbk(z)=\overline{\varphi_\Bbbk(\bar{z})}$ and $\varphi_\Bbbk$ is a homeomorphism of  $\mathbb{D}=\{z\in \Bbbk: \lvert z\rvert \le 1\}$, it follows that $\Phi_\Bbbk$ is a homeomorphism of $\mathbb{S}_\eq^n(\mathbb{D})$.   Since the fixed point of $\varphi_\Bbbk$  are the $z$'s with $\lvert z \rvert \in \{0,1\}$, it follows that the fixed points  of  $\Phi_\Bbbk$ are exactly the matrices with $\lvert a_{ij}\rvert \in \{0,1\}$ for all $i,j = 1,\dots,d$. 
\end{proof}

We denote the Schur product of two matrices $A, B \in \mathbb{S}^n$ by $A \circ B \in \mathbb{S}^n$. Recall that this is simply the coordinatewise product, i.e., the $(i,j)$th entry of $A \circ B$ is $a_{ij}b_{ij}$, $i,j=1,\dots,n$.  We also write $A^{\circ k}$ for the $k$-fold Schur product of $A$ with itself, i.e.,  the $(i,j)$th entry of $A^{\circ k}$ is $a_{ij}^k$.
\begin{lemma}\label{basgrin}
Let $A\in \mathbb{S}^n$ and $G \in \mathbb{G}^n$. For any $k\in\mathbb{N}$, set
\begin{equation}\label{eq:Mk}
M_k=A \circ \Phi_\Bbbk(G)^{\circ (k+1)} \circ \overline{\Phi_\Bbbk(G)}^{\circ k} \in \mathbb{S}^n
\end{equation}
with $\Phi_\Bbbk$ as defined in \eqref{eq:Phi}.
Then for all $k \in \mathbb{N}$,
\begin{equation}\label{eq:noninc}
\lVert M_k\rVert_\theta \le \lVert A\rVert_\theta.
\end{equation}
\end{lemma}
\begin{proof}
The proof will apply to both $\Bbbk = \mathbb{R}$ and $\mathbb{C}$ alike --- the only difference is that complex conjugation will have no effect when $\Bbbk=\mathbb{R}$. We will write $\Phi = \Phi_\Bbbk$ and $\varphi = \varphi_\Bbbk$ below for notational simplicity.

Let $x_1,\dots,x_n\in\Bbbk^n$ be unit vectors and $G = G(x_1,\dots,x_n) \in \mathbb{G}^n$ be their Gram matrix.
Consider the matrices  $A \circ \Phi(G)$, $A \circ\overline{ \Phi(G)} \in \mathbb{S}^n$,
whose $(i,j)$th entries are $a_{ij}\varphi(\langle x_i,x_j\rangle)$ and $a_{ij}\overline{\varphi(\langle x_i,x_j\rangle)}$ respectively.

Let $z\in\Bbbk^n$ with $\langle z,x_j\rangle\ne 0$ for all $j = 1,\dots,n$. By the definition of the $\theta$-seminorm,
\[
\biggl\lvert\sum_{i=1}^n\sum_{j=1}^n a_{ij}\sign\langle x_i,z\rangle \sign\langle z,x_j\rangle\biggr\rvert
\le \lVert A \rVert_\theta, \quad
\biggl\lvert\sum_{i=1}^n\sum_{j=1}^n a_{ij}\overline{\sign\langle x_i,z\rangle}\overline{\sign\langle z,x_j\rangle}\biggr\rvert \le \lVert A \rVert_\theta.
\]
By \eqref{2-eq1}, these imply
\begin{equation}\label{eq:2-thm1}
\biggl\lvert\sum_{i=1}^n\sum_{j=1}^n a_{ij}\varphi(\langle x_i,x_j\rangle)\biggr\rvert \le \lVert A \rVert_\theta, \quad
\biggl\lvert\sum_{i=1}^n\sum_{j=1}^n a_{ij}\overline{\varphi(\langle x_i,x_j\rangle)}\biggr\rvert \le \lVert A \rVert_\theta.
\end{equation}
Consider  a diagonal matrix $D= \diag(\delta_1,\dots,\delta_n)$ with $\delta_1,\dots,\delta_n \in \mathbb{T}$.  Clearly  we always have $\lVert D^*\! AD\rVert_\theta=\lVert A \rVert_\theta$. So replacing $A$ by $D^*\! AD$  in \eqref{eq:2-thm1} gives us
\[
\biggl\lvert\sum_{i=1}^n\sum_{j=1}^n a_{ij}\varphi(\langle x_i,x_j\rangle)\bar{\delta}_i \delta_j\biggr\rvert \le \lVert A \rVert_\theta, \quad
\biggl\lvert\sum_{i=1}^n\sum_{j=1}^n a_{ij}\overline{\varphi(\langle x_i,x_j\rangle)}\bar{\delta}_i \delta_j\biggr\rvert \le \lVert A \rVert_\theta.
\]
As $\delta_1,\dots,\delta_n \in \mathbb{T}$ are arbitrary, it follows that 
\[
\lVert A \circ \Phi(G) \rVert_\theta\le \lVert A \rVert_\theta, \quad \lVert A \circ\overline{ \Phi(G)} \rVert_\theta\le \lVert A \rVert_\theta,
\]
i.e., the operation of Schur multiplication by $ \Phi(G) $ or $\overline{ \Phi(G)}$ does not increase the $\theta$-seminorm. Hence \eqref{eq:noninc} follows.
\end{proof}

We now deduce that the symmetric Grothendieck inequalities \eqref{eq:GGI} and \eqref{eq:GGILE} hold with constants bounded by $\sinh(\pi/2)\approx 2.30130$ over $\mathbb{R}$ and $8/\pi-1\approx 1.54648$ over $\mathbb{C}$ respectively.
\begin{theorem}[Symmetric Grothendieck inequalities]\label{ubgenGC}
Let $A \in \mathbb{S}^n$ and $d \ge n$. Then there exist constants $K_\gamma, K_\Gamma >0$ independent of $d$ and $n$ such that
\begin{align}
\max_{\lVert x_i\rVert = 1}\; \biggl\lvert\sum_{i=1}^n \sum_{j=1}^n a_{ij}\langle x_i, x_j\rangle \biggr\rvert &\le K_\gamma \max_{ \lvert \delta_i \rvert = 1}\; \biggl\lvert\sum_{i=1}^n\sum_{j=1}^n a_{ij}  \bar{\delta}_i\delta_j\biggr\rvert \label{eq:SGIs}
\shortintertext{and}
\max_{\lVert x_i\rVert \le 1}\; \biggl\lvert\sum_{i=1}^n \sum_{j=1}^n a_{ij}\langle x_i, x_j\rangle \biggr\rvert 
&\le K_\Gamma \max_{ \lvert \delta_i \rvert \le 1}\; \biggl\lvert\sum_{i=1}^n\sum_{j=1}^n a_{ij}  \bar{\delta}_i\delta_j\biggr\rvert \label{eq:SGIn}
\end{align}
where $x_1,\dots,x_n \in \Bbbk^d$. Furthermore,
\begin{equation}\label{eq:SGC}
K_\Gamma^\mathbb{R} \le K_\gamma^\mathbb{R}\le \sinh \frac{\pi}{2}, \qquad
K_\Gamma^\mathbb{C} \le K_\gamma^\mathbb{C}\le \frac{8}{\pi}-1.
\end{equation}
\end{theorem}
\begin{proof}
Let $\Bbbk = \mathbb{R}$ or $\mathbb{C}$ and $\varphi_\mathbb{R}$ and $\varphi_\mathbb{C}$ be as in \eqref{eq:haa} and \eqref{defphiC}.   
We will first show that the inverse function of $\varphi_\Bbbk$ has a power series expansion of the form
\begin{equation}\label{eq:ck}
\varphi_\Bbbk^{-1}(z) = \sum_{k=0}^\infty c_{2k+1}z\lvert z \rvert^{2k}, \quad z\in \Bbbk,
\end{equation}
that is convergent when $\lvert z\rvert \le 1$ and that
\[
B_\Bbbk \coloneqq \sum_{k=0}^\infty \lvert c_{2k+1} \rvert =
\begin{cases}
\sinh \dfrac{\pi}{2} &\text{if } \Bbbk= \mathbb{R},\\
\dfrac{8}{\pi}-1 &\text{if } \Bbbk= \mathbb{C}.
\end{cases}
\]

\underline{\textsc{Case I}: $\Bbbk=\mathbb{R}$.}\; By \eqref{eq:int}, $\varphi_\mathbb{R}(x)=(2/\pi) \arcsin (x)$ and so $\varphi^{-1}_\mathbb{R}(x) =\sin (\pi x/2)$.  Therefore $B_\mathbb{R}=\sinh (\pi/2)$. Note that
\[
\varphi^{-1}_\mathbb{R}(x)
=\sum_{k=0}^\infty \frac{(-1)^k}{(2k+1)!} \Bigl(\frac{\pi}{2}\Bigr)^{2k+1}x \lvert x\rvert^{2k},
\]
is convergent when $\lvert x\rvert \le 1$.

\underline{\textsc{Case II}: $\Bbbk=\mathbb{C}$.}\; It follows from Lemma~\ref{Halem3.5}\eqref{it:phiC} that the power series of the inverse function of $\varphi_\mathbb{C}$ takes the form
\begin{equation}\label{eq:seriesC}
\varphi_\mathbb{C}^{-1}(z)=\sum_{k=0}^\infty b_{2k+1}z\lvert z \rvert^{2k}
\end{equation}
for some $b_{2k+1} \in \mathbb{R}$, $k =0,1,2,\dots.$
Unlike the real case, we do not have closed-form expressions for these coefficients. However, by \eqref{eq:int}, we do have that
\[
b_1=\lim_{z\to 0} \frac{z}{\varphi_\mathbb{C}(z)} = \frac{4}{\pi};
\]
and a result of Haagerup \cite{Haagerup} (see also \cite[Section~5]{FLZ18}) shows that
\[
b_{2k+1}\le 0\quad \text{for all}\; k\ge 1.
\]
Therefore
\begin{equation}\label{eq:cal}
1 =\varphi^{-1}_\mathbb{C}(1) =b_1+\sum_{k=1}^\infty  b_{2k+1} =b_1-\sum_{k=1}^\infty \lvert b_{2k+1} \rvert,
\end{equation}
and we obtain
\[
B_\mathbb{C}= \sum_{k=0}^\infty \lvert b_{2k+1} \rvert = b_1  +\sum_{k=1}^\infty \lvert b_{2k+1} \rvert = 2b_1-1=\frac{8}{\pi}-1.
\]
It also follows from \eqref{eq:cal} that the series in \eqref{eq:seriesC} converges when $\lvert z\rvert \le 1$.

We may now prove \eqref{eq:SGIs} starting from the tautology
\[
\langle x_i,x_j\rangle=\varphi_\Bbbk^{-1}\bigl(\varphi_\Bbbk(\langle x_i,x_j\rangle)\bigr).
\]
Let $M_k \in \mathbb{S}^n$ be as in \eqref{eq:Mk} and $\mathbbm{1}=(1,\dots,1) \in\mathbb{R}^n$.  Then by \eqref{eq:ck} and the observation that the $(i,j)$th entry of $M_k$ is $a_{ij} \langle x_i,x_j\rangle \lvert \langle x_i,x_j\rangle \rvert^{2k}$, we get
\[
\sum_{i=1}^n\sum_{j=1}^n a_{ij}\langle x_i, x_j\rangle=\sum_{i=1}^n\sum_{j=1}^n a_{ij}\varphi_\Bbbk^{-1}\bigl(\varphi_\Bbbk(\langle x_i,x_j\rangle)\bigr)= \sum_{k=0}^\infty  c_{2k+1} \mathbbm{1}^\tp  M_k\mathbbm{1}.
\]
Applying Lemma~\ref{basgrin}, we get
\[
\biggl\lvert\sum_{i=1}^n\sum_{j=1}^n a_{ij}\langle x_i, x_j\rangle\biggr\rvert \le \sum_{k=0}^\infty \lvert c_{2k+1} \rvert \lvert \mathbbm{1}^\tp  M_k\mathbbm{1}\rvert \le B_\Bbbk \lVert A \rVert_\theta.
\]
 
We now show that $K_\Gamma \le K_\gamma$, from which \eqref{eq:SGIn} follows. The same argument works over both $\mathbb{R}$ and $\mathbb{C}$.
Let $D = \diag(\delta_1,\dots,\delta_n)$ be a  diagonal matrix with $\delta_i\in[0,1]$, $i = 1,\dots,n$. Then by \eqref{eq:altGamma}, $\lVert DAD \rVert_\theta\le \lVert A \rVert_\Theta$. Since $\lVert DAD \rVert_\gamma\le K_\gamma \lVert DAD \rVert_\theta\le  K_\gamma \lVert A \rVert_\Theta$, by \eqref{eq:altGamma} again, we have that $\lVert A \rVert_\Gamma \le K_\gamma \lVert A \rVert_\Theta$.  Hence $K_\Gamma \le K_\gamma$.
\end{proof}
 
Should a distinction be necessary, we will refer to  \eqref{eq:SGIs} as the \emph{symmetric Grothendieck inequality for $\gamma$-seminorm} and \eqref{eq:SGIn} as the \emph{symmetric Grothendieck inequality for $\Gamma$-norm}. In case it is lost on the reader, the main point of Theorem~\ref{ubgenGC} is in establishing the bounds in \eqref{eq:SGC} for the symmetric Grothendieck constants. If we do not care about the size of the bounds, then we could just use \eqref{ineqvarnrm} alongside the original Grothendieck inequality \eqref{eq:GI2} to get that $K_\Gamma\le 2K_\G$, as
\[
\lVert A \rVert_\Gamma \le \lVert A \rVert_{\G} \le K_\G \lVert A \rVert_{\infty,1}= K_\G \lVert A \rVert_{\G,1} \le 2 K_G \lVert A \rVert_{\Gamma,1}= 2 K_G \lVert A \rVert_\Theta
\]
for any $A \in \mathbb{S}^n$. If a numerical value is not required, then Blei mentioned in \cite[p.~17]{BleiBook} that one may deduce the existence of an $\mathcal{O}(1)$-factor by adding absolute values to both sides of an inequality in \cite[Equation 2]{Alon2}. The preceding discussion only applies to the inequality \eqref{eq:SGIn}, the inequality \eqref{eq:SGIs} and the bounds on $K_\gamma$ do not follow from such general arguments as far as we can tell.

The following proposition quantifies the relations between the various Grothendieck constants $K_\gamma^\Bbbk$, $K_\Gamma^\Bbbk$, $K_\G^\Bbbk$, and between the respective real and complex versions.
\begin{proposition}\label{ineqbetegrodcons}
Let $K_\gamma^\Bbbk$, $K_\Gamma^\Bbbk$, and $K_\G^\Bbbk$ be respectively the smallest constants such that \eqref{eq:SGIs}, \eqref{eq:SGIn}, and \eqref{eq:GI} hold.  Then we have
\begin{equation}\label{ingrothconst}
K_\G^\Bbbk \le K_\Gamma^\Bbbk \le K_\gamma^\Bbbk,
\end{equation}
for $\Bbbk = \mathbb{R}$ or $\mathbb{C}$, and
\begin{equation}\label{realmajcomp}
K_\gamma^\mathbb{C}\le K_\gamma^\mathbb{R}, \qquad
K_\Gamma^\mathbb{C}\le K_\Gamma^\mathbb{R}, \qquad
K_\G^\mathbb{C}\le K_\G^\mathbb{R}.
\end{equation}
\end{proposition}
\begin{proof}
Let $W_{m,n} = \bigl\{\begin{bsmallmatrix} 0 & B \\ B^* & 0 \end{bsmallmatrix} \in \mathbb{S}^{m+n} : B \in \Bbbk^{m \times n} \bigr\}$. Then
\[
K_\G = \sup_{m,n \in \mathbb{N}} \biggl[
\max_{A \in W_{m,n}} \frac{\lVert A \rVert_\gamma}{\lVert A \rVert_\theta} \biggr]
= \sup_{m,n \in \mathbb{N}} \biggl[
\max_{A \in W_{m,n}} \frac{\lVert A \rVert_\Gamma}{\lVert A \rVert_\Theta} \biggr].
\]
Since $W_{m,n} \subseteq \mathbb{S}^{m+n}$, we get the first inequality in \eqref{ingrothconst}; the second inequality was already established as part of Theorem~\ref{ubgenGC}.

We next prove the first inequality in \eqref{realmajcomp}; the remaining two may be similarly proved. Let  $\delta = (\delta_1,\dots,\delta_n) \in \mathbb{C}^n$. We will write $\delta_i=\alpha_i+\mathrm{i}\beta_i$, $\alpha_i,\beta_i \in\mathbb{R}$ and
\[
\widehat{\delta} \coloneqq (\alpha_1,-\beta_1,\dots,\alpha_n,-\beta_n) \in\mathbb{R}^{2n}.
\]
Recall that $a=\alpha+\mathrm{i}\beta \in \mathbb{C}$ may be represented as
$\pi(a) = \begin{bsmallmatrix} \alpha & \beta \\ -\beta & \alpha \end{bsmallmatrix} \in \mathbb{R}^{2 \times 2}$. Let $A \in \mathbb{C}^{n \times n}$ be Hermitian. Then
\[
\widehat{A} \coloneqq \bigl( \pi(a_{ij}) \bigr)\in \mathbb{R}^{2n \times 2n}
\]
is symmetric, and furthermore,
\[
\delta^*\! A \delta=\widehat{\delta}^\tp \widehat{A} \widehat{\delta}.
\]

As the unit disk $\{(\alpha,\beta) \in\mathbb{R}^2 : \alpha^2 +\beta^2\le 1\}$ contains $\{(\alpha,\beta) \in\mathbb{R}^2 : \lvert \alpha \rvert = \lvert \beta \rvert =1/\sqrt{2} \}$, by choosing $\lvert \alpha_i \rvert = \lvert \beta_i \rvert  = 1/\sqrt{2}$, $i = 1,\dots,n$, we  deduce that
\begin{equation}\label{eq:3rd}
2\lVert A\rVert_\theta^\mathbb{C} \ge \lVert \widehat{A} \rVert_\theta^\mathbb{R}.
\end{equation}

Let $x_1,\dots,x_n\in\mathbb{C}^n$. We write $x_i = u_i + \mathrm{i}v_i$,  $u_i,v_i\in\mathbb{R}^n$,  and set
\[
y_{2i-1} \coloneqq \begin{bmatrix} u_i \\ v_i \end{bmatrix}, \quad y_{2i} \coloneqq \begin{bmatrix} -v_i \\ u_i \end{bmatrix} \in \mathbb{R}^{2n}
\]
for $i =1,\dots,n$. Then
\[
\sum_{i=1}^n\sum_{j=1}^n a_{ij}\langle x_i, x_j\rangle_\mathbb{C} 
=\frac{1}{2} \sum_{i=1}^{2n}\sum_{j=1}^{2n} \pi(a_{ij}) \langle y_i, y_j \rangle_\mathbb{R}.
\]
If $\lVert x_1\rVert = \dots = \lVert x_n \rVert =1$, then $\lVert y_1 \rVert = \dots = \lVert y_{2n} \rVert = 1$, and we obtain
\[
\lVert A\rVert_\gamma^\mathbb{C}
\le \frac{1}{2} \lVert \widehat{A} \rVert_\gamma^\mathbb{R} 
\le \frac{1}{2} K_\gamma^\mathbb{R} \lVert \widehat{A} \rVert_\theta^\mathbb{R}
\le K_\gamma^\mathbb{R} \lVert A \rVert_\theta^\mathbb{C},
\]
where the second and third inequalities are \eqref{eq:GGILE} and \eqref{eq:3rd} respectively. Therefore $K_\gamma^\mathbb{C}  \le K_\gamma^\mathbb{R}$.
\end{proof}

There are slight generalizations of the symmetric and original Grothendieck inequalities to Grothendieck $d$-norms that follow from Lemma~\ref{lem:d}. In the following, the symmetric Grothendieck inequalities \eqref{eq:SGI} and the original Grothendieck inequality \eqref{eq:GI2} may be obtained by setting $d = 1$ and taking $p \to \infty$. As usual, the constants $K_{\gamma,d,p}$, $K_{\Gamma,d,p} $, $K_{\G,d,p}$ below depend on the field but we will only indicate this dependence when comparing them over different fields. We will also write $K_{\gamma,d,\infty} \coloneqq \lim_{p \to \infty}  K_{\gamma,d,p}$ and likewise for $K_{\Gamma,d,\infty}$, $K_{\G,d,\infty}$ in the rest of this article.
\begin{corollary}\label{cor:dnorms}
Let  $d,p,m,n\in \mathbb{N}$. If $1 \le d \le p \le n$, then there exist finite constants $K_{\gamma,d,p}, K_{\Gamma,d,p} > 0$ independent of $n$  such that
\[
\lVert A \rVert_{\gamma,p}\le K_{\gamma,d,p} \lVert A \rVert_{\gamma,d},\qquad
\lVert A \rVert_{\Gamma,p}\le K_{\Gamma,d,p} \lVert A \rVert_{\Gamma,d}
\]
for all $A \in \mathbb{S}^n$. If $1 \le d \le p \le m+n$, then there exists a finite constant $K_{\G,d,p} >0$  independent of $m,n$, such that
\begin{equation}\label{eq:Gcd}
\lVert B \rVert_{\G,p}\le K_{\G,d,p} \lVert B \rVert_{\G,d}
\end{equation}
for all $B \in \Bbbk^{m \times n}$. If these constants are chosen to be smallest possible, i.e.,
\[
K_{\gamma,d,p} \coloneqq  \sup_{n \in \mathbb{N}} \biggl[
\max_{A \in \mathbb{S}^n} \frac{\lVert A \rVert_{\gamma,p}}{\lVert A \rVert_{\gamma,d}} \biggr], \quad
K_{\Gamma,d,p} \coloneqq  \sup_{n \in \mathbb{N}} \biggl[
\max_{A \in \mathbb{S}^n} \frac{\lVert A \rVert_{\Gamma,p}}{\lVert A \rVert_{\Gamma,d}} \biggr], \quad
K_{\G,d,p} \coloneqq  \sup_{m,n \in \mathbb{N}} \biggl[
\max_{B \in \Bbbk^{m\times n}} \frac{\lVert B \rVert_{\G,p}}{\lVert B \rVert_{\G,d}} \biggr],
\]
then
\begin{equation}\label{eq:RCineq}
\begin{gathered}
K_{\gamma,d,p}\le K_{\gamma,1,n} \le K_\gamma, \\
K_{\gamma,2d,2p}^\mathbb{R} \le K_{\gamma,d,p}^\mathbb{C}, \\
K_{\gamma,1,\infty} = K_\gamma,
\end{gathered}
\qquad
\begin{gathered}
K_{\Gamma,d,p} \le K_{\Gamma,1,n} \le K_\Gamma, \\
K_{\Gamma,2d,2p}^\mathbb{R} \le{} K_{\Gamma,d,p}^\mathbb{C}, \\
K_{\Gamma,1,\infty} = K_\Gamma,
\end{gathered}
\qquad
\begin{gathered}
K_{\G,d,p} \le K_{\G,1,m+n} \le K_\G; \\
K_{\G,2d,2p}^\mathbb{R} \le K_{\G,d,p}^\mathbb{C}; \\
K_{\G,1,\infty} = K_\G.
\end{gathered}
\end{equation}
Furthermore, for any $1 \le d \le p \le q \le \infty$, we have
\[
K_{\gamma,d,q}\le K_{\gamma,d,p}K_{\gamma,p,q}, \qquad
K_{\Gamma,d,q}\le K_{\Gamma,d,p}K_{\Gamma,p,q}, \qquad
K_{\G,d,q}\le K_{\G,d,p}K_{\G,p,q}.
\]
\end{corollary}
\begin{proof}
Only the inequalities in \eqref{eq:RCineq} require some justification; all other statements are obvious from definitions and  Lemma~\ref{lem:d}. Let $A\in\mathbb{S}^n(\mathbb{R})$. Then $\lVert A \rVert_{\gamma,p}^\mathbb{C}\le K_{\gamma,d,p}^\mathbb{C} \lVert A \rVert_{\gamma,d}^\mathbb{C}$. So by \eqref{gammaC}, $\lVert A \rVert_{\gamma,2p}^\mathbb{R}\le K_{\gamma,d,p}^\mathbb{C} \lVert A \rVert_{\gamma,2d}^\mathbb{R}$.  Hence $K_{\gamma,2d,2p}^\mathbb{R} \le K_{\gamma,d,p}^\mathbb{C}$. Likewise for the other two inequalities.
\end{proof}

Such variants of the original Grothendieck inequality, i.e., for the $(\G,d)$-norms, were first studied by Krivine \cite{Krivine2} and have appeared in many places \cite{Acin, Bri3,Bri1,Bri2, Bene, Hirsch}. 
In \cite{Acin,Bene,Hirsch,Krivine2}, $K_{\G,1,p}$ is denoted $K_\G(p)$ and called the order-$p$ Grothendieck constant; in \cite{Bri2}, the inequality \eqref{eq:Gcd} is called the generalized Grothendieck inequality and $K_{\G,d,p}$ denoted $K_\G(p \mapsto d)$. We will have more to say about these inequalities in Section~\ref{sec:mother}. In particular, Krivine \cite[p.~17]{Krivine2} showed that
\begin{equation}\label{eq:K2}
K_{\G,1,2}^\mathbb{R}=\sqrt{2},
\end{equation}
and this has some interesting consequences for us, first of which is a slight improvement of our bound for the real symmetric Grothendieck constants.
\begin{corollary}\label{Krimpr}
$K_\Gamma^\mathbb{R}\le K_\gamma^\mathbb{R}\le \sqrt{2}(8/\pi-1)\approx 2.18705$.
\end{corollary}
\begin{proof}
By Corollary~\ref{cor:dnorms} and \eqref{eq:K2}, $K_{\gamma}^\mathbb{R}=K_{\gamma,1,\infty}^\mathbb{R}\le K_{\gamma,1,2}^\mathbb{R}K_{2,\infty}^\mathbb{R}\le K_{\gamma,1,2}^\mathbb{R}K_{\gamma,1,\infty}^\mathbb{C}\le \sqrt{2}(8/\pi-1)$.
\end{proof}

A perhaps interesting aside that follows from \eqref{infty1C} and \eqref{eq:K2} is that for $B\in\mathbb{R}^{m\times n}$, we have
\[
\lVert B \rVert_{\infty,1}^\mathbb{C} = \lVert B \rVert_{\G,2}^\mathbb{R} \le K_{\G,1,2}^\mathbb{R}\lVert B \rVert_{\G,1}^\mathbb{R} =\sqrt{2}\lVert B \rVert_{\infty,1}^\mathbb{R},
\]
a sharp version of the bound $\lVert B \rVert_{\infty,1}^\mathbb{C}\le 2\lVert B \rVert_{\infty,1}^\mathbb{R}$ in \cite[Proposition~2.1]{HK05}.

\section{Symmetric Grothendieck inequality for cones}\label{sec:cones}

We will see that results such as the Nesterov $\pi/2$-Theorem \cite{Nes,Rietz} and Goemans--Williamson inequality \cite{GW95} are in fact just ``symmetric Grothendieck inequality for positive semidefinite matrices'' and ``symmetric Grothendieck inequality for graph Laplacians.''  To be more specific, let $\mathsf{C} = \{C_n \subseteq \mathbb{S}^n : n \in \mathbb{N}\}$ be a family of convex cones.   By a symmetric Grothendieck inequality for $\mathsf{C}$, we mean an inequality of the form
\begin{equation}\label{eq:CGI2}
\lVert A \rVert_\gamma \le K_\mathsf{C} \lVert A \rVert_\theta
\end{equation}
where
\begin{equation}\label{eq:CGC}
K_\mathsf{C} \coloneqq  \sup_{n \in \mathbb{N}} \biggl[
\max_{A \in C_n \subseteq \mathbb{S}^n} \frac{\lVert A \rVert_\gamma}{\lVert A \rVert_\theta} \biggr].
\end{equation}
Obviously, $K_\mathsf{C} \le K_\gamma$.
In particular, $K_\mathsf{C}$ is finite and does not depend on the dimension $n$; we will call it the \emph{conic Grothendieck constant} for $\mathsf{C}$.

We have already seen that if we set  $W_{m,n} = \bigl\{\begin{bsmallmatrix} 0 & B \\ B^* & 0 \end{bsmallmatrix} \in \mathbb{S}^{m+n} : B \in \Bbbk^{m \times n} \bigr\}$, a subspace and therefore
trivially a convex cone, then we obtain the original Grothendieck constant $K_\G$. Another immediate example is $C_n = \mathbb{S}^n(\mathbb{R}_\p\!)$ --- by Corollary~\ref{cor:nonnegRC}, we see that $K_\mathsf{C} = 1$.
We will discuss some nontrivial examples below. 

In Section~\ref{sec:RN}, we will see that for $C_n = \mathbb{S}^n_\p$, the conic Grothendieck constants may be determined exactly: $K_\mathsf{C}^\mathbb{R} = \pi/2$ and $K_\mathsf{C}^\mathbb{C} = 4/\pi$. Furthermore, the inequalities over $\mathbb{R}$ and $\mathbb{C}$ are the well-known Nesterov $\pi/2$-Theorem and  Ben-Tal--Nemirovski--Roos $4/\pi$-Theorem respectively.

In Section~\ref{sec:GW}, we will see that for $C_n =  \mathbb{L}^n$, the cone of weighted graph Laplacians on $n$-vertex graphs, we obtain the even better known inequality  of Goemans--Williamson. The conic Grothendieck constant in this case is  $K_\GW^\mathbb{R} = 1/\alpha_\GW^\mathbb{R}$ where $\alpha_\GW^\mathbb{R} \approx 0.878567$ is the Goemans--Williamson constant. We will prove a complex analogue.

An immediate advantage of obtaining these inequalities as symmetric Grothendieck inequalities is that we obtain them over both $\mathbb{R}$ and $\mathbb{C}$ simultaneously.

\subsection{Cones of positive semidefinite matrices}\label{sec:RN}

We will show that the Nesterov $\pi/2$-Theorem \cite[Theorem~3.3]{Nes}, which actually appeared earlier in \cite[Theorem 4]{Rietz}, as well as its complex analogue, the Ben-Tal--Nemirovski--Roos $4/\pi$-Theorem \cite[Equation~57]{Ben-Tal},  will follow easily from the discussions in the last two sections. The sharpness of the Nesterov $\pi/2$-Theorem \eqref{RNineq}, i.e., that $\pi/2$ is not just an upper bound but the exact value of the conic Grothendieck constant for $\mathbb{S}^n_\p(\mathbb{R})$,  was established by Alon and Naor \cite{Alon3}. Our approach will yield $4/\pi$ as the exact value of the conic Grothendieck constant for $\mathbb{S}^n_\p(\mathbb{C})$, showing that the Ben-Tal--Nemirovski--Roos $4/\pi$-Theorem  \eqref{CRNineq} is also sharp. While the key ideas for our sharpness proof are due to Alon and Naor \cite{Alon3}, we take the opportunity to slightly refine their proof, using exact expressions established in Lemma~\ref{lem:expectations} in place of asymptotic estimates and providing a more careful argument in Lemma~\ref{unformconvlem}, both to be found in the appendix.
\begin{theorem}\label{FRNineq}
Let $A\in \mathbb{S}_\p^n(\mathbb{R})$.  Then
\begin{equation}\label{RNineq}
\lVert A \rVert_\gamma^\mathbb{R} \le \frac{\pi}{2} \lVert A \rVert_\theta^\mathbb{R}.
\end{equation}
Let $A\in \mathbb{S}_\p^n(\mathbb{C})$.  Then
\begin{equation}\label{CRNineq}
\lVert A \rVert_\gamma^\mathbb{C} \le \frac{4}{\pi} \lVert A \rVert_\theta^\mathbb{C}.
\end{equation}
These inequalities are sharp.
In both inequalities, the seminorm $\lVert\,\cdot\,\rVert_\gamma$ may be replaced by the norms $\lVert\,\cdot\,\rVert_\Gamma$ or $\lVert\,\cdot\,\rVert_\G$ and the seminorm $\lVert\,\cdot\,\rVert_\theta$ may be replaced by the norms $\lVert\,\cdot\,\rVert_\Theta$ or $\lVert\,\cdot\,\rVert_{\infty,1}$.
\end{theorem}
\begin{proof}
Our proof works for $\Bbbk = \mathbb{R}$ and $\mathbb{C}$ alike. We will indicate field dependence in our notations only when it is necessary to do so.

In terms of Schur product, for any $A \in \mathbb{S}^n_\eq(\mathbb{D})$, the map $\Phi$ in \eqref{eq:Phi} may be expressed as
\[
\Phi(A) = \sum_{k=0}^\infty a_k  A \circ (A \circ \overline{A})^{\circ k}
\]
where, as in \eqref{eq:haa} and \eqref{defphiC},
\[
a_k^\mathbb{R} =\frac{2}{\pi}\frac{(2k-1)!!}{(2k)!!(2k+1)}, \qquad
a_k^\mathbb{C} = \frac{\pi}{4}\bigg(\frac{(2k-1)!!}{(2k)!!}\biggr)^2 \frac{1}{k+1}
\]
for $k =0,1,2,\dots.$ Let $G \in \mathbb{G}^n$. Then $G \in \mathbb{S}_\eq^n(\mathbb{D})$ and
\[
\Phi(G) \succeq a_0 G,
\]
as $a_k >0$  and $G \circ (G \circ \overline{G})^{\circ k}\succeq 0$ for all $k \in \mathbb{N}$ (Schur products of positive semidefinite matrices remain positive semidefinite). Hence for any $A \in \mathbb{S}^n_\p$ and $G \in \mathbb{G}^n$,
\[
\lVert A\rVert_\theta \ge \tr \bigl( A\Phi(G)\bigr) \ge a_0 \tr( A G),
\]
where the first inequality is by \eqref{eq:2-thm1}. The characterization of $\lVert\,\cdot\,\rVert_\gamma$ in Corollary~\ref{cor:altGinfty1} yields
\[
\lVert A \rVert_\theta \ge a_0 \max\{\tr\bigl(AG\bigr): G \in \mathbb{G}^n \} = a_0\lVert A \rVert_\gamma.
\]
It remains to observe that $a_0^\mathbb{R} =2/\pi$ and $a_0^\mathbb{C} = \pi/4$.
That the $\theta$- and $\gamma$-seminorms may be replaced by other Grothendieck norms is simply a consequence of \eqref{posdefcase1}. 

We now show that \eqref{RNineq} and \eqref{CRNineq} are sharp. 
Let $0< \varepsilon < 1/5n$ be arbitrary. Let $m \in \mathbb{N}$ and the unit vectors $x_1,\dots,x_m \in \Bbbk^n$ be as  in Lemma~\ref{unformconvlem}.  Set
\[
G \coloneqq G(x_1,\dots,x_m) \in \mathbb{G}^m, \qquad A \coloneqq \frac{1}{m^2}G \in \mathbb{S}_\eq^m \cap \mathbb{S}_\p^m.
\]
Setting $\alpha = 2$ in \eqref{eq:moments} of Lemma~\ref{lem:expectations} and in the lower bound in 
\eqref{basinmunuy} of Lemma~\ref{unformconvlem}, we obtain
\[
\frac{1}{n} -5\varepsilon \le\min_{\lVert v \rVert =1} \biggl[\frac{1}{m}\sum_{i=1}^m \lvert \langle x_i,v\rangle\rvert^2\biggr]\le \frac{1}{m^2}\sum_{j=1}^m\sum_{i=1}^m \lvert \langle x_i,x_j\rangle\rvert^2 = \tr(AG),
\]
and thus $\lVert A \rVert_\gamma\ge 1/n - 5\varepsilon$  by \eqref{eq:normcorrd}.

We next show that
\begin{equation}\label{eq:upbdtheta}
\lVert A \rVert_\theta \le \max_{\lVert v \rVert =1} \biggl[\frac{1}{m}\sum_{i=1}^m \lvert \langle x_i,v\rangle\rvert\biggr]^2 \le
\begin{cases}
\dfrac{1}{\pi}\dfrac{\Gamma(n/2)^2}{\Gamma\bigl((n+1)/2\bigr)^2} + 3\varepsilon & \Bbbk = \mathbb{R}, \\[4ex]
\dfrac{\pi}{4}\dfrac{\Gamma(n)^2}{\Gamma(n+1/2)^2} + 3\varepsilon & \Bbbk = \mathbb{C}.
\end{cases}
\end{equation}
The second inequality above comes from setting $\alpha = 1$ in \eqref{basinmunuy} of Lemma~\ref{unformconvlem} and replacing $n$ by $n-1$ in \eqref{eq:moments} of Lemma~\ref{lem:expectations}, noting that $\Gamma(3/2) = \sqrt{\pi}/2$. We will next show the first inequality.
For $t=(t_1,\dots,t_{m})\in \mathbb{T}^m$,
\[
t^*\! A t=\frac{1}{m^2}\sum_{i=1}^m \sum_{j=1}^m\langle x_i,x_j\rangle \bar{t}_i t_j=\Bigl\lVert \frac{1}{m}\sum\nolimits_{i=1}^m \bar{t}_i x_i\Bigr\rVert^2.
\]
By definition, $\lVert A \rVert_\theta$ is the maximum of $t^*\! A t$ over all $t\in \mathbb{T}^m$ and let this be attained at $s=(s_1,\dots,s_{m})\in\mathbb{T}^m$. 
Set $x \coloneqq \sum\nolimits_{i=1}^m \bar{s}_i x_i\in \mathbb{C}^n$ and $v \coloneqq x/\lVert x\rVert$.  Then for  $t\in \mathbb{T}^m$,
\[
\Bigl\lvert\Bigl\langle\sum\nolimits_{i=1}^m \bar{t}_i x_i, x\Bigr\rangle\Bigr\rvert =\lVert x \rVert \Bigl\lvert\Bigl\langle\sum\nolimits_{i=1}^m \bar{t}_i x_i, v\Bigr\rangle\Bigr\rvert \le \lVert x \rVert\sum_{i=1}^m \lvert \bar{t}_i \rvert \lvert\langle x_i, v\rangle\rvert =\lVert x \rVert\sum_{i=1}^m \lvert\langle x_i, v\rangle\rvert,
\]
with equality if and only if $t_i=\sign \langle x_i, v\rangle$, $i=1,\dots,m$. On the other hand, Cauchy--Schwartz and the maximality of $s$ yields
\[
\Bigl\lvert\Bigl\langle \sum\nolimits_{i=1}^m \bar{t}_i x_i, x \Bigr\rangle\Bigr\rvert\le \Bigl\lVert\sum\nolimits_{i=1}^m \bar{t}_i x_i\Bigr\rVert \Bigl\lVert\sum\nolimits_{i=1}^m \bar{s}_i x_i\Bigr\rVert\le \lVert x \rVert^2.
\]
It follows that we must have $\bar{s}_i=\sign\langle x_i, v\rangle$, $i =1,\dots,m$. Therefore
\begin{align*}
\lVert A \rVert_\theta  =  \Bigl\lVert \frac{1}{m}x\Bigr\rVert^2 &= \Bigl\lvert\Bigl\langle  \frac{1}{m}\sum\nolimits_{i=1}^m \sign\langle x_i, v\rangle x_i, \frac{1}{m}\lVert x \rVert v \Bigr\rangle\Bigr\rvert\\
&= \Bigl\lVert \frac{1}{m}x\Bigr\rVert \Bigl\lvert\frac{1}{m}\sum\nolimits_{i=1}^m \sign( \langle x_i,v\rangle) \langle x_i, v\rangle\Bigr\rvert\\  
&= \Bigl\lvert\frac{1}{m}\sum\nolimits_{i=1}^m \sign( \langle x_i,v\rangle) \langle x_i, v\rangle\Bigr\rvert^2= \Big(\frac{1}{m}\sum\nolimits_{i=1}^m \vert \langle x_i, v\rangle\vert \Big)^2
\end{align*}
and we have the first inequality in \eqref{eq:upbdtheta}. Hence we get
\[
\frac{\Vert A\Vert _{\gamma}}{\Vert A\Vert _\theta}\ge
\begin{cases}
\Bigl( \dfrac{1}{n} - 5\varepsilon\Bigr) \smash[t]{\biggl[ \dfrac{1}{\pi}\dfrac{\Gamma(n/2)^2}{\Gamma\bigl((n+1)/2\bigr)^2} + 3\varepsilon\biggr]^{-1}} & \Bbbk = \mathbb{R},\\[4ex]
\Bigl( \dfrac{1}{n} - 5\varepsilon\Bigr) \smash[b]{\biggl[\dfrac{\pi}{4}\dfrac{\Gamma(n)^2}{\Gamma(n+1/2)^2} + 3\varepsilon\biggr]^{-1}} & \Bbbk = \mathbb{C},
\end{cases}
\]
and as $\varepsilon >0$ is arbitrary,
\[
\frac{\Vert A\Vert _{\gamma}}{\Vert A\Vert _\theta}\ge 
\begin{cases}
\dfrac{\pi}{2}\biggl[ \dfrac{\Gamma(n/2+1/2)}{\Gamma(n/2)\sqrt{n/2}}\biggr]^2 \to \dfrac{\pi}{2} & \Bbbk = \mathbb{R},\\[4ex]
\dfrac{4}{\pi}  \biggl[\dfrac{\Gamma(n+1/2)}{\Gamma(n)\sqrt{n}}\biggr]^2 \to \dfrac{4}{\pi} & \Bbbk = \mathbb{C},
\end{cases}
\]
as $n \to \infty$, using the fact that $\lim_{t\to\infty} \Gamma( t+\alpha)/\bigl(\Gamma(t) t^{\alpha}\bigr)=1$ for any $\alpha \in \mathbb{C}$.
\end{proof}
Combined with \eqref{eq:allequal} of Corollary~\ref{cor:nonnegRC} and \eqref{thetaC} of Corollary~\ref{corARC2}, the  inequality \eqref{CRNineq} shows that for any $A\in \mathbb{S}^n_\p(\mathbb{R})$, we have
\[
\lVert A\rVert_\gamma^\mathbb{R} = \lVert A\rVert_\gamma^\mathbb{C} \le \frac{4}{\pi} \lVert A \rVert_\theta^\mathbb{C}  = \frac{4}{\pi}\lVert A\rVert_{\gamma,2}^\mathbb{R},
\]
giving the  $n=2$ case in \cite[Theorem~1]{Bri1}. Note however that  \eqref{CRNineq} cannot be deduced from \cite[Theorem~1]{Bri1}. We will have more to say about these inequalities in Proposition~\ref{lbKcd} where we extend Theorem~\ref{FRNineq} to an arbitrary pair of Grothendieck $d$- and $p$-norms.

\subsection{Cones of weighted Laplacians}\label{sec:GW}

We begin with a precaution: All matrices considered in this section will be real but the norms can be taken over either $\mathbb{R}$ or $\mathbb{C}$. See Section~\ref{sec:RCnorms} for a discussion of real and complex Grothendieck norms.

For any $A\in \mathbb{S}^n_\circ =\mathbb{S}^n_\circ(\mathbb{R})$, the space of $n\times n$ real symmetric matrices with zero diagonal, we let $L_A \in \mathbb{S}^n =\mathbb{S}^n(\mathbb{R})$ be defined by
\begin{equation}\label{eq:LA}
L_A \coloneqq \diag(A\mathbbm{1}) - A,
\end{equation}
where $\diag(x)\in \mathbb{S}^n$ denotes the diagonal matrix whose diagonal is $x \in \mathbb{R}^n$ and $\mathbbm{1} =(1,\dots,1) \in \mathbb{R}^n$ is the vector of all ones.

If in addition $A \in \mathbb{S}^n(\mathbb{R}_\p\!)$, i.e., $a_{ij}\ge 0$ for all $i,j = 1,\dots,n$, then $L_A$  is called a \emph{weighted Laplacian}. Note that this implies $L_A$ has all off-diagonal entries nonpositive. It also follows from definition that $L_A\mathbbm{1} =0$. In fact these last two conditions are enough to characterize the set of all weighted Laplacians:
\[
\mathbb{L}^n \coloneqq \{ L_A \in \mathbb{S}^n : A \in \mathbb{S}^n_\circ(\mathbb{R}_\p\!)\}
= \{ L \in \mathbb{S}^n :  L \mathbbm{1} = 0,\; \ell_{ij} \le 0 \; \text{for all}\; i \ne j \}.
\]
Clearly, $\mathbb{L}^n \subseteq \mathbb{S}^n_\p$ and  $\mathbb{L}^n$ is a convex cone.

We will now establish the symmetric Grothendieck inequality for the cone $\mathbb{L}^n$. The inequality in the case $\Bbbk = \mathbb{R}$ was first discovered by Goemans and Williamson \cite{GW95} and its sharpness established by Feige and Schechtman \cite{Feige1, Feige2}. The corresponding inequality for $\Bbbk = \mathbb{C}$ is new as far as we know. The constants in Theorem~\ref{GW95th} have approximate values
\[
\alpha_\GW^\mathbb{R}\approx 0.87856, \qquad \alpha_\GW^\mathbb{C}\approx 0.93494,
\]
and so $K_\GW^\mathbb{R} \approx 1.1382$, $K_\GW^\mathbb{C} \le 1.0696$.
\begin{theorem}\label{GW95th}
Let $\Bbbk = \mathbb{R}$ or $\mathbb{C}$ and $\varphi_\Bbbk$ be as defined in \eqref{eq:haa} or \eqref{defphiC} respectively. Let
\begin{equation}\label{eq:alpha}
\alpha_\GW^\Bbbk\coloneqq \inf_{0\le x \le 1}\frac{1+\varphi_\Bbbk(x)}{1+x}.
\end{equation}
Then for any  $L\in \mathbb{L}^n$,
\begin{equation}\label{GW}
\lVert L \rVert_\gamma^\Bbbk \le K_\GW^\Bbbk \lVert L \rVert_\theta^\Bbbk
\end{equation}
where the smallest possible constants
\begin{equation}\label{GWconst}
K_\GW^\mathbb{R} = \frac{1}{\alpha_\GW^\mathbb{R}}, \qquad K_\GW^\mathbb{C} \le \frac{1}{\alpha_\GW^\mathbb{C}}.
\end{equation}
In \eqref{GW}, the seminorm $\lVert\,\cdot\,\rVert_\gamma$ may be replaced by the norms $\lVert\,\cdot\,\rVert_\Gamma$ or $\lVert\,\cdot\,\rVert_\G$ and the seminorm $\lVert\,\cdot\,\rVert_\theta$ may be replaced by the norms $\lVert\,\cdot\,\rVert_\Theta$ or $\lVert\,\cdot\,\rVert_{\infty,1}$.
\end{theorem}
\begin{proof}
We will show that for both $\Bbbk = \mathbb{R}$ and $\mathbb{C}$,
\[
\alpha_\GW^\Bbbk = \inf_{ \lvert z \rvert<1} \frac{1-\Re\varphi_\Bbbk(z)}{1-\Re z} .
\]
The reader is reminded that excluding some or all boundary points makes no difference when taking infimum over a region.

\underline{\textsc{Case I}: $\Bbbk=\mathbb{R}$.}\; For $x\in [-1,1]$, $\varphi_\mathbb{R}(x) =(2/\pi) \arcsin x$.  If $x\in [0,1]$, then $f(x) \coloneqq 1-\varphi_\mathbb{R}(x)$ is a concave function with $f(0)=1$ and $f(1)=0$.  Therefore the graph of $f$ lies above the line $1-x$ for all $x\in[0,1]$, i.e., $f(x)\ge 1-x$.  Hence
\[
\inf_{0 \le x < 1} \frac{1-\varphi_\mathbb{R}(x)}{1-x} =1.
\]
If $x\in [-1,0]$, let $y=-x\in [0,1]$, then
\[
\frac{1-\varphi_\mathbb{R}(x)}{1-x}= \frac{1+\varphi_\mathbb{R}(y)}{1+y}.
\]
Thus
\[
\inf_{0 \le y \le 1} \frac{1+\varphi_\mathbb{R}(y)}{1+y}= \inf_{-1 \le x \le 0}\frac{1-\varphi_\mathbb{R}(x)}{1-x}= \inf_{-1 < x < 1}\frac{1-\varphi_\mathbb{R}(x)}{1-x},
\]
where the last equality follows since the infimum is not attained over $[0,1)$.

\underline{\textsc{Case II}: $\Bbbk=\mathbb{C}$.}\; Let $z=x +\mathrm{i}y$ where $x,y\in\mathbb{R}$, $x^2+y^2\le 1$. By \eqref{defphiC},
\begin{align*}
\Re\varphi_\mathbb{C}(z) &= \frac{\pi}{4}\sum_{k=0}^\infty\left(\frac{(2k-1)!!}{(2k)!!}\right)^2 \frac{1}{k+1}x(x^2+y^2)^{k}.
\intertext{If $x\ge 0$, then}
\Re\varphi_\mathbb{C}(z)&\le \frac{\pi}{4}\sum_{k=0}^\infty\left(\frac{(2k-1)!!}{(2k)!!}\right)^2 \frac{1}{k+1}x=x,
\end{align*}
where the last equality follows from $\varphi_\mathbb{C}(1)=1$, by Lemma~\ref{Halem3.5}\eqref{it:phiC}.  Therefore $1 -\Re \varphi_\mathbb{C}(z)\ge 1-\Re z$ when $\Re z\ge 0$, $\lvert z \rvert\le 1$. In fact,
\[
\inf_{\Re z\ge 0, \; \lvert z \rvert < 1} \frac{1 -\Re \varphi_\mathbb{C}(z)}{1-\Re z}= \frac{1 -\Re \varphi_\mathbb{C}(0)}{1-\Re 0} = 1,
\]
and so we may exclude the $\Re z \ge 0$ region when seeking the infimum over $\lvert z \rvert < 1$, i.e.,
\[
\inf_{\lvert z \rvert < 1} \frac{1 -\Re \varphi_\mathbb{C}(z)}{1-\Re z}=
\inf_{\Re z\le 0, \; \lvert z \rvert < 1} \frac{1 -\Re \varphi_\mathbb{C}(z)}{1-\Re z}.
\]
Now for $\Re z\le 0$, $\lvert z \rvert \le 1$, write $z=-x+\mathrm{i}y$ where $0 \le x \le 1$. Then
\[
\varphi_\mathbb{C}(x)  = \frac{\pi}{4}\sum_{k=0}^\infty\left(\frac{(2k-1)!!}{(2k)!!}\right)^2 \frac{1}{k+1}x^{2k+1} \le
\frac{\pi}{4}\sum_{k=0}^\infty\left(\frac{(2k-1)!!}{(2k)!!}\right)^2 \frac{1}{k+1}x(x^2+y^2)^{k} =  -\Re \varphi_\mathbb{C}(z).
\]
Since  $x =-\Re z$,
\[
 \inf_{0 \le x \le1} \frac{1+ \varphi_\mathbb{C}(x)}{1+x} \le \inf_{\Re z\le 0,\; \lvert z \rvert < 1} \frac{1-\Re\varphi_\mathbb{C} (z)}{1-\Re z}.
\]
The final step relies on a tautology: if $\Re z\le 0$, $\lvert z \rvert \le 1$, then $t = \Re z$ satisfies $\Re t \le 0$, $\lvert t \rvert <1$.  Hence it is trivially true that
\[
\inf_{\Re z\le 0,\; \lvert z \rvert < 1} \frac{1-\Re\varphi_\mathbb{C}(z)}{1-\Re z} \le
\inf_{\Re z\le 0,\; \lvert z \rvert<1}\frac{1-\varphi_\mathbb{C}(\Re z)}{1-\Re z} =
\inf_{0 \le x \le1} \frac{1+ \varphi_\mathbb{C}(x)}{1+x},
\]
where the last equality follows from $\varphi_\mathbb{C}(-x) = -\varphi_\mathbb{C}(x)$.

The remainder of this proof will work for $\Bbbk = \mathbb{R}$ and $\mathbb{C}$ alike; for notational simplicity, we write $\varphi = \varphi_\Bbbk$, $\alpha_\GW = \alpha_\GW^\Bbbk$, $G_n =G_n^\Bbbk$ in the following. Let $A \in \mathbb{S}^n_\circ(\mathbb{R}_\p\!)$. Then
\[
\rVert L_A \rVert_\gamma =\sum_{i=1}^n\sum_{j=1}^n a_{ij}(1-\langle x_i,x_j\rangle)
\]
for some unit vectors $x_1,\dots, x_n \in \Bbbk^n$. By \eqref{2-eq1},
\[
\int_{\Bbbk^n}\sum_{i=1}^n\sum_{j=1}^n a_{ij}\sign\langle x_i,z\rangle \sign\langle z,x_j\rangle G_n(z)\,dz
=\sum_{i=1}^n\sum_{j=1}^n a_{ij}\varphi (\langle x_i,x_j\rangle)
=\sum_{i=1}^n\sum_{j=1}^n a_{ij}\Re\varphi(\langle x_i,x_j\rangle),
\]
where the last equality holds since $A$ is real and symmetric. Therefore
\[
\int_{\Bbbk^n} \sum_{i=1}^n\sum_{j=1}^n a_{ij}\sign\langle x_i,z\rangle \sign\langle z,x_j\rangle G_n(z)\,dz=\sum_{i=1}^n\sum_{j=1}^n a_{ij}\Re\varphi(\langle x_i,x_j\rangle).
\]

Since $a_{ii}=0$, $i=1,\dots,n$, for any $\delta=(\delta_1,\dots,\delta_n) \in \Bbbk^n$,
\begin{equation}\label{LWxeq}
\delta^*  L_A\delta=\sum_{i=1}^n\sum_{j \ne i} a_{ij}\bigl(1-\Re(\bar{\delta}_i\delta_j) \bigr).
\end{equation}
For any $z\in\Bbbk^n$, let $\delta(z) \coloneqq (\sign \langle z,x_1\rangle ,\dots,\sign \langle z,x_n\rangle ) \in \Bbbk^n$. Since $a_{ij}\ge 0$, $i,j=1,\dots,n$,
\begin{align*}
\int_{\Bbbk^n} \!\!\delta(z)^*  L_A \delta(z) G_n(z)\,dz
&= \sum_{i=1}^n\sum_{j=1}^n a_{ij}\bigl(1-\Re\varphi( \langle x_i,x_j\rangle)\bigr)\\
&\ge \sum_{i=1}^n\sum_{j=1}^n a_{ij}\alpha_\GW(1-\Re\langle x_i,x_j\rangle)\\
&=\alpha_\GW \tr \bigl(L_A G(x_1,\dots,x_d)\bigr)=\alpha_\GW \rVert L_A \rVert_\gamma.
\end{align*}
Hence there exists $\delta=(\delta_1,\dots,\delta_n)\in\Bbbk^n$, $\lvert \delta_i \rvert=1$, such that $\delta^*  L_A \delta\ge \alpha_\GW \rVert L_A \rVert_\gamma$. The required inequality \eqref{GW} then follows from $\lVert L_A \rVert_\theta \ge \delta^*  L_A \delta$.
The last statement is a consequence of \eqref{posdefcase1}:
\[
\lVert L_A \rVert_\gamma=\lVert L_A \rVert_\Gamma=\lVert L_A \rVert_\G, \qquad
\lVert L_A \rVert_\theta = \lVert L_A \rVert_\Theta =  \lVert L_A \rVert_{\infty,1}
\]
as $L_A \in \mathbb{S}^n_\p$. This establishes that $K_{\GW}^\Bbbk \le 1/\alpha_{\GW}^\Bbbk$ in \eqref{GWconst}; that equality holds in the case when $\Bbbk = \mathbb{R}$ is a well-known result of Feige and Schechtman \cite{Feige1, Feige2}.
\end{proof}

To obtain the more familiar expression for $\alpha_\GW^\mathbb{R}$ in \cite{GW95}, note that
\[
\alpha_\GW^\mathbb{R} =\inf_{0 \le x \le 1} \frac{1+\varphi_\mathbb{R}(x)}{1+x} = \inf_{\pi/2 \le \theta\le \pi}\frac{2}{\pi} \frac{\theta}{1-\cos\theta} = \inf_{0< \theta\le \pi}\frac{2}{\pi} \frac{\theta}{1-\cos\theta}.
\]
Naturally, we suspect that equality also holds in \eqref{GWconst} when  $\Bbbk = \mathbb{C}$, i.e.,
\[
K_\GW^\mathbb{C} \overset{?}{=} \frac{1}{\alpha_\GW^\mathbb{C}}.
\]
While most of the construction in \cite{Feige1, Feige2} carries over to $\mathbb{C}$, the ultimate difficulty is in obtaining the value of $\lVert L \rVert_\theta^\mathbb{C}$ for the weighted Laplacian $L$ of the constructed graph ---  unlike $\lVert L \rVert_\theta^\mathbb{R}$, which is essentially maxcut,  $\lVert L \rVert_\theta^\mathbb{C}$ has no combinatorial interpretation.

Similar to our discussions at the end of Section~\ref{sec:RN}, since $\mathbb{L}^n \subseteq \mathbb{S}^n_\p(\mathbb{R})$, when combined with \eqref{eq:allequal} of Corollary~\ref{cor:nonnegRC} and \eqref{thetaC} of Corollary~\ref{corARC2}, the complex Goemans--Williamson inequality \eqref{GW} shows that for any $L\in \mathbb{L}^n$,
\[
\lVert L \rVert_\gamma^\mathbb{R} = \lVert L \rVert_\gamma^\mathbb{C} \le K_\GW^\mathbb{C} \lVert L \rVert_\theta^\mathbb{C} = K_\GW^\mathbb{C} \lVert L \rVert_{\gamma,2}^\mathbb{R}.
\]
We will have more to say about these inequalities and constants in Proposition~\ref{prop:GWd} where Theorem~\ref{GW95th} is extended to an arbitrary pair of Grothendieck $d$- and $p$-norms.

\subsection{Cones of diagonally dominant matrices}\label{sec:diagdom} 

As in the last section, all matrices considered in this section will be real but the norms can be taken over either $\mathbb{R}$ or $\mathbb{C}$. Again, we write $\mathbb{S}^n =\mathbb{S}^n(\mathbb{R})$ throughout the rest of this section.

Let $\mathbb{S}_{\DD}^n \coloneqq \{A \in \mathbb{S}^n : a_{ii}\ge \sum_{j\ne i}\lvert a_{ij} \rvert\}$ be the cone of symmetric diagonally dominant matrices. Clearly,
\[
\mathbb{L}^n\subseteq \mathbb{S}_{\DD}^n\subseteq \mathbb{S}^n_\p.
\]
The relation between the first two cones may be more precisely characterized.  Let $\mathbb{S}_{\DD}^n(\mathbb{R}_\p\!) \subseteq \mathbb{S}_{\DD}^n$ be the subcone of diagonally dominant matrices with nonnegative entries.
\begin{lemma}\label{SDDdecom}
Every $A\in \mathbb{S}_{\DD}^n$ has a unique decomposition $A =H + L$ with $H \in \mathbb{S}_{\DD}^n(\mathbb{R}_\p\!)$, $L \in\mathbb{L}^n$, and $h_{ij}\ell_{ij}=0$ whenever $i\ne j$, i.e.,
$\mathbb{S}_{\DD}^n= \mathbb{S}_{\DD}^n(\mathbb{R}_\p\!) + \mathbb{L}^n$.
\end{lemma}
\begin{proof}
Let $B \in \mathbb{S}^n_\circ(\mathbb{R}_\p\!)$, i.e., nonnegative symmetric with zero diagonal, be defined by
\[
b_{ij}\coloneqq
\begin{cases}
-a_{ij}  &\text{if } a_{ij}<0 \text{ and } i \ne j,\\
0 &\text{if } a_{ij}\ge 0 \text{ or } i = j.
\end{cases}
\]
Then $L_B = \diag(B\mathbbm{1}) - B \in \mathbb{L}^n $ by \eqref{eq:LA}.  Let $H \coloneqq A - L_B$.  Then $h_{ij}\ge  0$ and $h_{ij}\ell_{ij}=0$ for $i\ne j$. Therefore $H \in  \mathbb{S}^n(\mathbb{R}_\p\!)$. Since $a_{ii}\ge\sum_{j\ne i}  h_{ij} + \ell_{ij}$,  we also have $H \in \mathbb{S}_{\DD}^n$.  Uniqueness follows since $B$ is uniquely determined by $A$ and thus so are $L_B$ and $H$.
\end{proof}

We now prove a Grothendieck inequality for diagonally dominant matrices that may be regarded as an inequality `in between' the Nesterov $\pi/2$-Theorem and the Goemans--Willamson inequality over $\mathbb{R}$, and `in between' the Ben-Tal--Nemirovski--Roos $4/\pi$-Theorem and the complex Goemans--Willamson inequality  over $\mathbb{C}$.
\begin{proposition}\label{interprop}
Let $\Bbbk = \mathbb{R}$ or $\mathbb{C}$. Let $\alpha_\GW^\Bbbk$ be as in \eqref{eq:alpha} and $a_0^\mathbb{R} =2/\pi$, $a_0^\mathbb{C} = \pi/4$.
Then for any $A\in \mathbb{S}_{\DD}^n$,
\begin{equation}\label{eq:SDD}
\lVert A \rVert_\gamma^\Bbbk \le \biggl(1+\frac{1-a_0^\Bbbk}{\alpha_\GW^\Bbbk}\biggr)\lVert A \rVert_\theta^\Bbbk.
\end{equation}
In \eqref{eq:SDD}, the seminorm $\lVert\,\cdot\,\rVert_\gamma$ may be replaced by the norms $\lVert\,\cdot\,\rVert_\Gamma$ or $\lVert\,\cdot\,\rVert_\G$ and the seminorm $\lVert\,\cdot\,\rVert_\theta$ may be replaced by the norms $\lVert\,\cdot\,\rVert_\Theta$ or $\lVert\,\cdot\,\rVert_{\infty,1}$.
\end{proposition}
\begin{proof}
The same proof will work for both  $\Bbbk = \mathbb{R}$ and $\mathbb{C}$ and we drop $\Bbbk$ in our notations below. We may assume that $A \ne 0$.
Let $G\in \mathbb{G}^n$ be such that $\lVert A \rVert_\gamma=\tr(AG)$, i.e., attaining the maximum in Corollary~\ref{cor:altGinfty1}. Let $A= H + L$ be the unique decomposition given by Lemma~\ref{SDDdecom}. Set $t=\tr(LG)/\tr(AG)\in[0,1]$. We claim that
\begin{equation}\label{ineqCgamthet}
\lVert A \rVert_{\gamma}\le \frac{1}{\max\{ a_0+(\alpha_\GW-a_0)t, 1-t\}}\lVert A \rVert_\theta.
\end{equation}
Let $\varphi$ be as in \eqref{eq:haa} or \eqref{defphiC}. Then
\begin{align*}
\lVert A \rVert_\theta&\ge \tr \bigl( A\varphi(G) \bigr) =\tr \bigl( H\varphi(G) \bigr)+\tr\bigl( L\varphi(G) \bigr)\ge a_0\tr(PG) + \alpha_\GW \tr(LG) \\
&=a_0 \tr(AG)+
(\alpha_\GW-a_0)\tr(LG)= [a_0 + (\alpha_\GW-a_0)t]\tr(AG)= [a_0 +(\alpha_\GW-a_0)t]\lVert A \rVert_\gamma.
\end{align*}
Since $H\in \mathbb{S}^n(\mathbb{R}_\p\!)$, by Lemma~\ref{lem:nonneg}\eqref{it:nonneg}, we have $\lVert H \rVert_\gamma=\mathbbm{1}^\tp H\mathbbm{1}$. Since $L\in \mathbb{L}^n$, we have $\mathbbm{1}^\tp L\mathbbm{1}=0$. Thus $ \lVert H \rVert_\gamma =\mathbbm{1}^\tp H\mathbbm{1}=\mathbbm{1}^\tp\! A\mathbbm{1}\le \lVert A \rVert_\theta$ and so
\[
(1-t)\lVert A \rVert_\gamma=(1-t)\tr(AG)=\tr(PG)\le \lVert H \rVert_\gamma\le \lVert A \rVert_\theta,
\]
giving us \eqref{ineqCgamthet}. 
The required inequality \eqref{eq:SDD} follows from \eqref{ineqCgamthet} by minimizing $\max\{a_0+(\alpha_\GW-a_0)t, 1-t\}$ over $t \in [0,1]$, observing that the minimum $\alpha_\GW/(1-a_0+\alpha_\GW)$ is achieved when $a_0+(\alpha_\GW-a_0)t=1-t$.
\end{proof}

\subsection{Mother of all Grothendieck inequalities}\label{sec:mother}

It is straightforward to combine Corollary~\ref{cor:dnorms} and \eqref{eq:CGI2} to obtain an (almost) all-encompassing inequality. Let $d,p,n\in\mathbb{N}$ with $1 \le d \le p \le n$ and $\mathsf{C} = \{C_n \subseteq \mathbb{S}^n : n \in \mathbb{N}\}$ be a family of cones. Then
\begin{equation}\label{eq:MGI}
\lVert A \rVert_{\gamma,p} \le K_{\mathsf{C},d,p} \lVert A \rVert_{\gamma,d}
\end{equation}
where
\begin{equation}\label{eq:MGC}
K_{\mathsf{C},d,p} \coloneqq  \sup_{n \in \mathbb{N}} \biggl[
\max_{A \in C_n \subseteq \mathbb{S}^n} \frac{\lVert A \rVert_{\gamma,p}}{\lVert A \rVert_{\gamma,d}} \biggr] \le K_\mathsf{C} \le K_\gamma.
\end{equation}
Again the result holds over both $\Bbbk=\mathbb{R}$ and $\mathbb{C}$. We have the following special cases:
\begin{enumerate}[\upshape (i)]
\item\label{it:GI} Grothendieck inequality \cite{Grothendieck, Lindenstrauss}:
\[
C_n = \bigl\{\begin{bsmallmatrix} 0 & B \\ B^* & 0 \end{bsmallmatrix} \in \mathbb{S}^n : B \in \Bbbk^{m \times (n - m)} \bigr\}, \qquad d = 1, \qquad p =n \to \infty;
\]

\item order-$p$ Grothendieck inequality \cite{Acin, Bene, Hirsch, Krivine2}:
\[
C_n = \bigl\{\begin{bsmallmatrix} 0 & B \\ B^* & 0 \end{bsmallmatrix} \in \mathbb{S}^n : B \in \Bbbk^{m \times (n - m)} \bigr\}, \qquad d=1, \qquad p \le m;
\]

\item generalized Grothendieck inequality \cite{Bri3}:
\[
C_n = \bigl\{\begin{bsmallmatrix} 0 & B \\ B^* & 0 \end{bsmallmatrix} \in \mathbb{S}^n : B \in \Bbbk^{m \times (n - m)} \bigr\}, \qquad d< p \le m;
\]

\item\label{it:SGI} symmetric Grothendieck inequality for $\gamma$-seminorm \eqref{eq:SGIs}:
\[
C_n = \mathbb{S}^n, \qquad d =1, \qquad p = n \to \infty;
\]

\item Nesterov $\pi/2$-Theorem \cite{Nes, Rietz} and Ben-Tal--Nemirovski--Roos $4/\pi$-Theorem \cite{Ben-Tal}:
\[
C_n = \mathbb{S}^n_\p, \qquad d = 1, \qquad p =n \to \infty;
\]

\item Goemans--Williamson inequality \cite{GW95}:
\[
C_n = \mathbb{L}^n, \qquad d = 1, \qquad p =n \to \infty;
\]

\item\label{it:rGI} rank-constrained positive semidefinite Grothendieck inequality \cite{Bri1, Bri2}:
\[
C_n = \mathbb{S}^n_\p, \qquad d =1, \qquad p \le n.
\]
\end{enumerate}
Missing from this list is the symmetric Grothendieck inequality for $\Gamma$-norm \eqref{eq:SGIn}, which cannot be obtained by restricting \eqref{eq:MGI} to any subspace or cone of $\mathbb{S}^n$. We may of course also define a $\Gamma$-norm version of \eqref{eq:MGI} and \eqref{eq:MGC} but since $\lVert \, \cdot \, \rVert_\Gamma = \lVert \, \cdot \, \rVert_\gamma$ on both $\mathbb{S}^n_\circ$ and $\mathbb{S}^n_\p$, they give the same inequalities  \eqref{it:GI}--\eqref{it:rGI} except for \eqref{it:SGI}, which becomes \eqref{eq:SGIn} instead.

We have the following extension of Corollary~\ref{cor:dnorms} to any family of cones.
\begin{lemma}\label{lem:CGC}
Let $d,p\in\mathbb{N}$ with $1 \le d \le p \le \infty$ and $\mathsf{C} = \{C_n \subseteq \mathbb{S}^n : n \in \mathbb{N}\}$ be a family of cones. Then
\[
K_{\mathsf{C},d,p} \le  K_{\mathsf{C},1,n} \le K_\mathsf{C}, \qquad
K_{\mathsf{C},1,\infty} = K_\mathsf{C}, \qquad
K_{\mathsf{C},d,p}^\mathbb{C} \ge K_{\mathsf{C},2d,2p}^\mathbb{R}.
\]
\end{lemma}
\begin{proof}
The first inequality and the second limit follow from their definitions in \eqref{eq:CGC} and \eqref{eq:MGC}. The last inequality follows from \eqref{gammaC}, i.e., $\lVert A\rVert_{\gamma,d}^\mathbb{C}=\lVert A\rVert_{\gamma,2d}^\mathbb{R}$ for any $A\in\mathbb{S}^n(\mathbb{R})$, and the obvious inclusion $\mathbb{S}^n(\mathbb{R}) \subseteq \mathbb{S}^n(\mathbb{C})$.
\end{proof}

One may also obtain bounds for the constants $K_{\mathsf{C}, d, p}$ for specific families of cones $\mathsf{C}$ such as the positive semidefinite cones and the cones of weighted Laplacians.

For the family of positive semidefinite cones  $\mathsf{C} = \{\mathbb{S}_\p^n : n \in \mathbb{N}\}$, the following lower bound for $K_{\mathsf{C}, d, p}^\mathbb{R}$ was established by Bri\"et, Buhrman, and Toner in \cite[Theorem~1]{Bri3} and its limiting expression for $K_{\mathsf{C},d,\infty}^\mathbb{R}$ by Bri\"et, de Oliveira Filho, and Vallentin in \cite[Theorem~2]{Bri1}. We will state the analogues for $\mathbb{C}$. Our lower bound  for $K_{\mathsf{C}, d, p}^\mathbb{C}$ is a straightforward consequence of Lemma~\ref{lem:CGC} and the lower bound for $K_{\mathsf{C}, d, p}^\mathbb{R}$ in \cite[Theorem~1]{Bri3}; but we will provide an alternative simple proof for the limiting expression $K_{\mathsf{C},d,\infty}^\Bbbk$ that applies to both $\Bbbk = \mathbb{R}$ and $\mathbb{C}$ alike and that is nearly identical to our proof of  Nesterov $\pi/2$-Theorem and Ben-Tal--Nemirovski--Roos $4/\pi$-Theorem in Theorem~\ref{FRNineq}.
\begin{proposition}\label{lbKcd}
Let $d,p\in\mathbb{N}$ with $1 \le d \le p \le \infty$ and $\mathsf{C} = \{\mathbb{S}_\p^n : n \in \mathbb{N}\}$.  Then
\begin{equation}\label{eq:lowerbd}
K_{\mathsf{C},d,p}^\Bbbk\ge 
\begin{cases}
\dfrac{d}{p}\biggl[\dfrac{\Gamma\bigl((p+1)/2\bigr)\Gamma(d/2)}{\Gamma(p/2)\Gamma\bigl((d+1)/2\bigr)}\biggr]^2 &\Bbbk = \mathbb{R},\\[4ex]
\dfrac{d}{p}\biggl[\dfrac{\Gamma\bigl((2p+1)/2\bigr)\Gamma(d)}{\Gamma(p)\Gamma\bigl((2d+1)/2\bigr)}\biggr]^2  &\Bbbk = \mathbb{C}.
\end{cases}
\end{equation}
Furthermore, for any fixed $d\in \mathbb{N}$,
\begin{equation}\label{eq:limit}
K_{\mathsf{C},d,\infty}^\Bbbk = 
\begin{cases}
\dfrac{d\Gamma(d/2)^2}{2\Gamma\bigl((d+1)/2\bigr)^2} &\Bbbk = \mathbb{R},\\[4ex]
\dfrac{d\Gamma(d)^2}{\Gamma\bigl((2d+1)/2\bigr)^2}  &\Bbbk = \mathbb{C}.
\end{cases}
\end{equation}
\end{proposition}
\begin{proof}
The lower bound in \eqref{eq:lowerbd} for $\mathbb{R}$ was established in \cite[Theorem~1]{Bri3}.  The lower bound for $\mathbb{C}$ thus follows from  Lemma~\ref{lem:CGC}  since $K_{\mathsf{C},d,p}^\mathbb{C} \ge K_{\mathsf{C},2d,2p}^\mathbb{R}$. The limiting expression for $\mathbb{R}$  in \eqref{eq:limit} was established in \cite[Theorem~2]{Bri1} but  we will provide an alternative simple proof that works for both $\mathbb{R}$ and $\mathbb{C}$.

Taking limits as $p \to \infty$ in \eqref{eq:lowerbd}, we see that
\[
K_{\mathsf{C},d,\infty}^\mathbb{R} \ge
\dfrac{d\Gamma(d/2)^2}{2\Gamma\bigl((d+1)/2\bigr)^2},\qquad
K_{\mathsf{C},d,\infty}^\mathbb{C} \ge
\dfrac{d\Gamma(d)^2}{\Gamma\bigl((2d+1)/2\bigr)^2};
\]
and it remains to establish that `$\ge$' above may be replaced by `$\le$'. This will in turn follow if we can show that for any $A \in \mathbb{S}^n_\p(\mathbb{R})$,
\begin{equation}\label{eq:genN}
\lVert A\rVert_\gamma^\mathbb{R} \le \dfrac{d\Gamma(d/2)^2}{2\Gamma((d+1)/2)^2}  \lVert A\rVert_{\gamma,d}^\mathbb{R};
\end{equation}
and for any $A \in \mathbb{S}^n_\p(\mathbb{C})$,
\begin{equation}\label{eq:genBT}
\lVert A\rVert_\gamma^\mathbb{C} \le\dfrac{d\Gamma(d)^2}{\Gamma((2d+1)/2)^2}\lVert A\rVert_{\gamma,d}^\mathbb{C}.
\end{equation}
Let $\Bbbk =\mathbb{R}$ or $\mathbb{C}$. As in the proof of Theorem~\ref{FRNineq}, for any $A \in \mathbb{S}^n_\p(\Bbbk)$ and $G \in \mathbb{G}^n(\Bbbk)$, we have
\[
\Phi_d^\Bbbk(G) \coloneqq \bigl(\varphi_d^\Bbbk (g_{ij})\bigr) \succeq b_{1,d}^\Bbbk G,
\]
where $b_{1,d}^\Bbbk$ is the first Taylor coefficient of $\varphi_d^\Bbbk$ as in Lemma~\ref{varphicCform}. 
Using the characterization of $\lVert\,\cdot\,\rVert_\gamma$ in Corollary~\ref{cor:altGinfty1}, let $G \in \mathbb{G}^n(\Bbbk)$ be such that  $\tr(AG) =\lVert A \rVert_\gamma^\Bbbk$. Then 
\[
\lVert A\rVert_{\gamma, d}^\Bbbk \ge  \tr \bigl( A \Phi_d^\Bbbk(G)  \bigr)
\ge b_{1,d}^\Bbbk \tr(A G)= b_{1,d}^\Bbbk\lVert A\rVert_\gamma^\Bbbk.
\]
By Lemma~\ref{varphicCform}\eqref{varphiexplform},
\[
b_{1,d}^\mathbb{R}=\frac{2}{\sqrt{\pi}}
\frac{\Gamma((d+1)/2)^2\Gamma(3/2)}{\Gamma(d/2)\Gamma((d+2)/2)}
=\frac{2}{d}\frac{\Gamma((d+1)/2)^2}{\Gamma(d/2)^2};
\]
and so by Lemma~\ref{varphicCform}\eqref{varphicCform2},
\[
b_{1,d}^\mathbb{C} = b_{1,2d}^\mathbb{R} = \frac{1}{d}\frac{\Gamma((d+1/2)^2)}{\Gamma(d)^2}.
\]
These give \eqref{eq:genN} and \eqref{eq:genBT} respectively.
\end{proof}
For $d =1$, we get $K_{\mathsf{C},1,\infty}^\mathbb{R} = \pi/2$ and $K_{\mathsf{C},1,\infty}^\mathbb{C} = 4/\pi$, so \eqref{eq:genN} generalizes the Nesterov $\pi/2$-Theorem and \eqref{eq:genBT} generalizes the Ben-Tal--Nemirovski--Roos $4/\pi$-Theorem.

In \cite[Theorem~1]{Bri3}, the lower bounds for $K_{\mathsf{C},d,p}^\mathbb{R}$ in \eqref{eq:lowerbd} are stated for matrices in $\mathbb{R}^{n \times n}$ as opposed to $\mathbb{S}^n_\p(\mathbb{R})$ but note that their proof actually assumes the latter. Since $\Bbbk^{n \times n} \supseteq \mathbb{S}^n_\p$, any lower bound for the latter is automatically one for the former.

The following discussion for the family of cones of weighted Laplacians $\mathsf{C} = \{\mathbb{L}^n : n \in \mathbb{N}\}$ is essentially that of Bri\"et, de Oliveira Filho, and Vallentin in \cite[Section~6]{Bri1},  where the constant $\alpha_d^\mathbb{R}$ is denoted $v(d)$. Our two minor contributions here are to extend it to $\mathbb{C}$ and to relate the constants over $\mathbb{R}$ and $\mathbb{C}$. For $d \in \mathbb{N}$, we will let $K_{\GW,d}^\Bbbk > 0$ be sharpest constant so that
\[
\lVert L \rVert_\gamma^\Bbbk \le K_{\GW,d}^\Bbbk \lVert L \rVert_{\gamma,d}^\Bbbk
\]
for all $L \in \mathbb{L}^n$ and all $n \in \mathbb{N}$. Clearly, $K_{\GW,1}^\Bbbk = K_\GW^\Bbbk$, where the latter is as defined in Theorem~\ref{GW95th}.
\begin{proposition}\label{prop:GWd}
Let $\mathsf{C} = \{\mathbb{L}^n : n \in \mathbb{N}\}$. For each $d \in \mathbb{N}$, let $\varphi_d^\Bbbk$ be as in Lemma~\ref{varphicCform}, and
\begin{equation}\label{eq:alphad}
\alpha_d^\Bbbk\coloneqq \inf_{0\le x \le 1}\frac{1+\varphi_d^\Bbbk(x)}{1+x}.
\end{equation}
Then
\[
K_{\mathsf{C},d,\infty}^\Bbbk = K_{\GW,d}^\Bbbk  \le \frac{1}{\alpha_d^\Bbbk}
\qquad\text{and}
\qquad
K_{\GW,d}^\mathbb{C} = K_{\GW,2d}^\mathbb{R}.
\]
\end{proposition}
\begin{proof}
Since $\mathbb{L}^n \subseteq \mathbb{S}^n(\mathbb{R})$,  by Proposition~\ref{prop:RCnorms}, we have $\lVert L \rVert_{\gamma,d}^\mathbb{C}=\lVert L \rVert_{\gamma,2d}^\mathbb{R}$  and thus $K_{\GW,d}^\mathbb{C} = K_{\GW,2d}^\mathbb{R}$. By definition, we have $K_{\mathsf{C},d,\infty} = K_{\GW,d}$, and it remains to establish $1/\alpha_d^\Bbbk$ as an upper bound. By Lemma~\ref{varphicCform}\eqref{varphicCform2},  $\varphi_d^\mathbb{C}(x) = \varphi_{2d}^\mathbb{R}(x)$ for $x \in [-1,1]$, it suffices to show that $K_{\GW,d}^\mathbb{R} \le 1/\alpha_d^\mathbb{R}$ and the corresponding result for $\mathbb{C}$ will follow.  To avoid clutter, we  drop the superscript $\mathbb{R}$ in the rest of this proof.

The same proof that we gave for the $d=1$ case in Theorem~\ref{GW95th} applies here with minor modifications and we will just outline the main steps. Note that $f \coloneqq 1-\varphi_d$ is concave with $f(0)=1$ and $f(1)=0$ as in the $d=1$ case.
Let $A \in \mathbb{S}^n_\circ(\mathbb{R}_\p\!)$ and $L = L_A \in \mathbb{L}^n$. As in the proof of   Theorem~\ref{GW95th},
\[
\rVert L \rVert_\gamma =\sum_{i=1}^n\sum_{j=1}^n a_{ij}(1-\langle x_i,x_j\rangle)
\]
for some unit vectors $x_1,\dots, x_n \in \mathbb{R}^n$; and
\begin{align*}
\lVert L \rVert_{\gamma,d}&\ge  \int_{\mathbb{R}^{d \times n}} \! \sum_{i=1}^n\sum_{j=1}^n \ell_{ij}\bigl\langle \sign(Zx_i), \sign(Zx_j)\bigr\rangle G_{d,n}(Z)\,dZ\\
&=\sum_{i=1}^n\sum_{j=1}^n a_{ij}\bigl(1-\varphi_d(\langle x_i,x_j\rangle)\bigr)\ge \alpha_d\sum_{i=1}^n\sum_{j=1}^n a_{ij}(1-\langle x_i,x_j\rangle)=
\alpha_d\rVert L_A \rVert_\gamma. \qedhere
\end{align*}
\end{proof}
It also follows from Lemma~\ref{varphicCform}\eqref{varphicCform2} and \eqref{eq:alphad} that
\[
\alpha_d^\mathbb{C} = \alpha_{2d}^\mathbb{R}.
\]
In particular the complex Grothendieck constant in \eqref{eq:alpha}, $\alpha_\GW^\mathbb{C} = \alpha_1^\mathbb{C} = \alpha_2^\mathbb{R}$. In principle, one may use the explicit expression for $\varphi^\mathbb{R}_d(x)$ in  Lemma~\ref{varphicCform}\eqref{varphiexplform} to obtain the numerical value of $\alpha_d^\mathbb{R}$. This was in fact done for $d=2$ and $3$ in \cite[Section~6]{Bri1} using a different expression.

\section{Applications}\label{sec:apps}

We will discuss some consequences of Theorem~\ref{ubgenGC} to computer science and optimization. We will deduce some polynomial-time approximation bounds that we think are new.

\subsection{Maxcut}

Let $G = (V,E)$ be an undirected graph on the vertex set $V = \{1,\dots,n\}$. A \emph{weighted adjacency matrix} of $G$  is a matrix $A\in\mathbb{S}^n_\circ$ where $a_{ij}=0$ whenever $\{i,j\} \notin E$. In which case, $L_A$ as defined in \eqref{eq:LA} is the corresponding weighted graph Laplacian.

A \emph{cut} of $G$ is a partition of the vertex set $V = S\cup S^c$ into disjoint nonempty sets.  Given a weighted adjacency matrix $A \in\mathbb{S}^n_\circ$,  the weight of the cut is
\[
\cut(S, S^c) \coloneqq \sum\nolimits_{i\in S,\; j\in S^c} a_{ij} =\frac{1}{4} e_S^\tp L_Ae_S
\]
where $e_S = (e_1,\dots,e_n) \in\{-1,1\}^n$ is defined by $e_i=1$ if and only if $i\in S$. 
It follows from \eqref{LWxeq} that
\begin{equation}\label{eq:maxcut}
\max_{ S\subsetneq V} \lvert \cut(S, S^c) \vert =\frac{1}{4} \lVert L_A \rVert_\theta^\mathbb{R}.
\end{equation}
If the weights are nonnegative, i.e., $A \in \mathbb{S}^n_\circ(\mathbb{R}_\p\!)$, then $L_A \in \mathbb{S}^n_\p$ and we may drop the absolute value in the left-side of \eqref{eq:maxcut}, which is called the \emph{maxcut} of $G$ weighted by $A$ and is known to be NP-hard. The Goemans--Williamson inequality, i.e., combining \eqref{GW} and \eqref{GWconst} for $\Bbbk = \mathbb{R}$,
\begin{equation}\label{eq:GWineq}
\alpha_\GW^\mathbb{R} \lVert L_A \rVert_\gamma^\mathbb{R}\le \lVert L_A \rVert_\theta^\mathbb{R},
\end{equation}
yields a polynomial-time approximation to within a factor of $\alpha_\GW^\mathbb{R}\approx 0.87856$ as both $\lVert L_A \rVert_\gamma^\mathbb{R}$ and the unit vectors $x_1,\dots,x_n \in \mathbb{R}^n$ that attain its value as in \eqref{eq:seminorms2} may be computed to arbitrary accuracy in polynomial-time.   Incidentally, by Corollary~\ref{cor:RCnorms}, we have
\[
\alpha_\GW^\mathbb{C} \lVert L_A \rVert_\gamma^\mathbb{R}\le \lVert L_A \rVert_\theta^\mathbb{C},
\]
and so $\lVert L_A \rVert_\theta^\mathbb{C}$ can be approximated to within a factor of $\alpha_\GW^\mathbb{C}\approx 0.93494$ in polynomial-time. The catch is that, unlike its real counterpart, $\lVert L_A \rVert_\theta^\mathbb{C}$ does not give us the maxcut.

Goemans and Williamson \cite{GW95} also showed  that for
a randomly chosen unit vector $z\in \mathbb{R}^n$ and $S_z \coloneqq \{i\in V : \langle z, x_i\rangle \ge 0\}$, the expected value of $\cut(S_z, S_z^c)$ will be at least
$\alpha_\GW^\mathbb{R} \lVert L_A \rVert_\gamma^\mathbb{R} \le \lVert L_A \rVert_\theta^\mathbb{R}$. 

Consider a bipartite graph $G = (V,E)$, i.e., $V = V_1\cup V_2$ and $E\subseteq V_1\times V_2$.  Let $ V_1  =\{1,\dots,m\}$ and $ V_2 =\{m+1,\dots,m+n\}$. Then its weighted adjacency matrix takes the form
\[
A= \begin{bmatrix} 0 & B\\B^\tp & 0\end{bmatrix} \in \mathbb{S}^{m +n}_\circ
\]
for some $B \in\mathbb{R}^{m\times n}$. Note that $B$ is not required to be a nonnegative matrix. The \emph{cut norm} of $B$ is defined  \cite{Alon3} as
\[
 \lVert B \rVert_{\cut} \coloneqq \max_{I_1\subseteq V_1,\; I_2\subseteq V_2} \Bigl\lvert \sum\nolimits_{i\in I_1,\; j\in I_2} b_{ij} \Bigr\rvert.
\]
\begin{proposition}\label{prop:cutnorm}
Let $A \in \mathbb{S}^{m+n}_\circ$ and $B \in \mathbb{R}^{m \times n}$ be as above. Then
\begin{equation}\label{eq:thetacut}
\frac{1}{8} \lVert L_A \rVert_\theta \le \lVert B \rVert_{\cut}\le \frac{3}{8} \lVert L_A \rVert_\theta.
\end{equation}
\end{proposition}
\begin{proof} 
For the partition $S=I_1\cup (V_2\setminus I_2)$ and $S^c =(V_1 \setminus I_1)\cup  I_2$,
\[
\cut(S, S^c) = \sum\nolimits_{i\in I_1,\; j\in I_2}b_{ij}+\sum\nolimits_{i\in V_1\setminus I_1, j\in V_2\setminus I_2} b_{ij}.
\]
The equality \eqref{eq:maxcut} yields
\begin{align*}
\lVert L_A \rVert_\theta &= 4\max_{I_1\subseteq V_1,\; I_2\subseteq V_2} \Bigl\lvert \sum\nolimits_{i\in I_1,\; j\in I_2}b_{ij}+\sum\nolimits_{i\in V_1\setminus I_1, \; j\in V_2\setminus I_2} b_{ij} \Bigr\rvert \\
&\le
 4\max_{I_1\subseteq V_1,\; I_2\subseteq V_2} \Bigl( \Bigl\lvert\sum\nolimits_{i\in I_1,\; j\in I_2}b_{ij}\Bigr\rvert + \Bigl\lvert\sum\nolimits_{i\in V_1\setminus I_1, j\in V_2\setminus I_2} b_{ij}\Bigl\lvert \Bigr)\le  8\lVert B\rVert_{\cut},
\end{align*}
i.e., the left inequality in \eqref{eq:thetacut}.

Setting $I_1$ to be arbitrary and $I_2=V_1$ or setting $I_1=V_2$ and $I_2$ to be arbitrary gives
\[
\max\Bigl(\Bigl\lvert\sum\nolimits_{i\in I_1,\; j\in V_1} b_{ij}\Bigr\rvert, \Bigl\lvert\sum\nolimits_{i\in V_2,\; j\in I_2}b_{ij}\Bigr\rvert\Bigr)\le \frac{1}{4}\lVert L_A \rVert_\theta.
\]
Hence
\[
\Bigl\lvert\sum\nolimits_{i\in I_1,\; j\in V_1\setminus I_2} b_{ij}
-\sum\nolimits_{i\in V_2\setminus I_1,\; j\in I_2}b_{ij}\Bigr\rvert
=\Bigl\lvert\sum\nolimits_{i\in I_1,\; j\in V_1} b_{ij}-\sum\nolimits_{i\in V_2,\; j\in I_2}b_{ij}\Bigl\lvert
\le \frac{1}{2} \lVert L_A \rVert_\theta.
\]
Setting $I_1$ to be arbitrary and $I_2=V_2\setminus I_2$, we get
\[
\Bigl\lvert\sum\nolimits_{i\in I_1,\; j\in I_2}b_{ij}\Bigr\rvert \le \frac{3}{8} \lVert L_A \rVert_\theta,
\]
which yields the right inequality in \eqref{eq:thetacut}.
\end{proof}

Now note that the Goemans--Williamson inequality \eqref{eq:GWineq} does not apply to the cut norm as  $B \in \mathbb{R}^{m \times n}$ is not required to be nonnegative; so $A\in \mathbb{S}^{m +n}_\circ$ may have negative entries and $L_A\in \mathbb{S}^{m +n}$ is no longer guaranteed to be positive semidefinite. This is where the symmetric Grothendieck inequality \eqref{eq:SGIs} can be useful, since it does not require positive semidefiniteness. When applied to Proposition~\ref{prop:cutnorm}, we obtain the following polynomial-time approximation for the cut norm.
\begin{corollary}
Let $A \in \mathbb{S}^{m+n}_\circ$ and $B \in \mathbb{R}^{m \times n}$ be as above. Then
\[
\frac{1}{8K_\gamma^\mathbb{R}} \lVert L_A \rVert_\gamma \le  \lVert B \rVert_{\cut} \le \frac{3}{8} \lVert L_A \rVert_\gamma.
\]
\end{corollary}

\subsection{Nonconvex quadratic programming}\label{subsec:Spread}

Another consequence of the symmetric Grothendieck inequality is a new convex relaxation bound that we will describe below after providing some context. Let $A \in \mathbb{S}^n(\mathbb{R})$ and consider the following ubiquitous quadratic programs in combinatorial optimization \cite{Alon1,Alon2,Alon3,Charikar,GW95,Khot,Khot2},
\begin{equation}\label{Maxprobcube}
q_\leq(A) \coloneqq \max_{x \in [-1,1]^n} x^\tp\! Ax ,\qquad
q_\eq(A) \coloneqq \max_{x \in \{-1,1\}^n} x^\tp\! Ax,
\end{equation}
i.e., maximization of a quadratic function over the unit cube and the vertices of the unit cube respectively. Both problems are known to be NP-hard \cite{deKlerk,Nes} and it is customary to consider the following standard  convex relaxations:
\begin{equation}\label{SDPmaxrel}
\begin{aligned}
r_\leq(A)  &\coloneqq \max\{\tr(AX) : x_{ii} \le 1, \; i = 1,\dots,n, \; X\succeq 0\},\\
r_\eq(A)&\coloneqq \max\{\tr(AX) : x_{ii} = 1, \; i = 1,\dots,n, \; X\succeq 0\}.
\end{aligned}
\end{equation}
The problems in  \eqref{SDPmaxrel} are semidefinite programs whose solutions can be found by interior-point algorithms to any arbitrary precision $\varepsilon>0$ in polynomial time \cite{NN94}.
Note that the identity matrix is within the sets of feasible matrices in the maximization problems \eqref{SDPmaxrel}.

If we let $A_0 \in \mathbb{S}^n_\circ(\mathbb{R})$ be the matrix obtained from $A\in \mathbb{S}^n(\mathbb{R})$ by replacing its diagonal entries with zeros, then 
\[
q_\eq(A)=q_\eq(A_0)+\tr A=q_\leq(A_0)+\tr A, \qquad r_\eq(A)=r_\eq(A_0)+\tr A=r_\leq(A_0)+\tr A.
\]
Nemirovski, Roos, and Terlaky \cite{NRT99}  proved a fundamental inequality that relates $r_\leq(A)$ and $q_\leq(A)$, namely,
\begin{equation}\label{NRTineqv0}
r_\leq(A) \le 2\log 2n \cdot q_\leq(A).
\end{equation}
When combined with \eqref{eq:nonnegdiag} that we will establish later, we get
\begin{equation}\label{NRTineqv1}
r_\eq(A)\le 2\log 2n \cdot q_\eq(A) \quad \textrm{if}\; \tr A\ge 0.
\end{equation}
The nonnegative trace condition cannot be omitted, e.g., the inequality does not hold for $A=-I$.
Megretski \cite{Meg01} improved the factor $2\log 2n$ in \eqref{NRTineqv0} and \eqref{NRTineqv1} to $2\log n$ for $n>60$. Without relying on  \cite{NRT99}, Charikar and Wirth \cite{Charikar}  obtained an analog of \eqref{NRTineqv1} for $A \in \mathbb{S}^n_\circ(\mathbb{R})$  and  coined the term ``Grothendieck-type inequality'' for \eqref{NRTineqv1}.
Given  $A  \in \mathbb{S}^n_\circ(\mathbb{R})$, Alon et al.\ \cite{Alon2} defined  a simple undirected graph $G$ by the zero pattern of  $A$, and thereby proved a variant of \eqref{NRTineqv1} where the factor $2\log 2n$ is replaced by $\Omega\bigl(\log\omega(G)\bigr)$,  with $\omega(G)$  the clique number of $G$.

All three estimates of the gap between $r_\eq(A)$ and $q_\eq(A)$ share the common feature  that the upper bounds of their ratio either  grow unbounded with the dimension of $A$ (e.g., $2 \log 2n$ or $2 \log n$) or at  least depends on $A$ (e.g., $\log\omega(G)$). In fact, these gaps are the smallest possible: Megretski  \cite[p.~2]{Charikar} showed that the best possible factor in \eqref{NRTineqv1} is $\Omega(\log n)$ while Alon et al.\ \cite{Alon2} showed that it is $\Omega\bigl(\log\omega(G)\bigr)$. With these in view, we find the following theorem somewhat surprising --- unlike the bounds of Alon et al., Charikar and Wirth, Megretski, the bounds below are in terms of universal constants independent of $n$ and $A$.
\begin{theorem}\label{thm:SGIqr}
For all $n \in \mathbb{N}$, any $A \in \mathbb{S}^n(\Bbbk)$, and both $\Bbbk = \mathbb{R}$ and $\mathbb{C}$,
\begin{equation}\label{eq:SGIqr}
\begin{gathered}
\max\,\{q_\eq(A), q_\eq(-A)\}\le \max\,\{r_\eq(A), r_\eq(-A)\}\le K^\mathbb{\Bbbk}_\gamma\max\,\{q_\eq(A), q_\eq(-A)\},\\
\max\,\{q_\leq(A), q_\leq(-A)\} \le \max\,\{r_\leq(A), r_\leq(-A)\} \le K^\mathbb{\Bbbk}_\Gamma\max\,\{q_\leq(A), q_\leq(-A)\}.
\end{gathered}
\end{equation}
These inequalities are sharp, i.e., $K^\mathbb{\Bbbk}_\gamma$ and $K^\mathbb{\Bbbk}_\Gamma$ are the smallest possible constants for \eqref{eq:SGIqr}.
\end{theorem}
While the objective function in \eqref{SDPmaxrel} is always real-valued and thus \eqref{SDPmaxrel} is well-defined over both $\mathbb{R}$ and $\mathbb{C}$, we will have to rewrite \eqref{Maxprobcube} as
\[
q_\leq(A) \coloneqq \max_{x \in \mathbb{D}^n} x^*\! Ax ,\qquad
q_\eq(A) \coloneqq \max_{x \in \mathbb{T}^n} x^*\! Ax,
\]
so that they apply to both  $\mathbb{R}$ and $\mathbb{C}$. Theorem~\ref{thm:SGIqr} and all discussions below will hold for both $\mathbb{R}$ and $\mathbb{C}$. By the next proposition, the inequalities in  \eqref{eq:SGIqr} are just restatements of the symmetric Grothendieck inequalities \eqref{eq:SGIs} and \eqref{eq:SGIn}.
\begin{proposition}\label{lastprop}
Let $A \in \mathbb{S}^n$. Then $q_\eq(A)\le q_\leq(A)$, $r_\eq(A)\le r_\leq(A) $, and
\begin{equation}\label{eq:qrrelations}
\begin{aligned}
\lVert A \rVert_\Theta&=\max\,\{q_\leq(A), q_\leq(-A)\},\qquad &\lVert A \rVert_\Gamma&=\max\,\{r_\leq(A), r_\leq(-A)\}, \\
\lVert A \rVert_\theta&=\max\,\{q_\eq(A), q_\eq(-A)\},\qquad &\lVert A \rVert_\gamma&=\max\,\{r_\eq(A), r_\eq(-A)\}.
\end{aligned}
\end{equation}
If in addition the diagonal entries $a_{11},\dots,a_{nn} \ge 0$, then
\begin{equation}\label{eq:nonnegdiag}
q_\eq(A)= q_\leq(A), \qquad r_\eq(A) = r_\leq(A).
\end{equation}
\end{proposition}
\begin{proof}
The equalities in \eqref{eq:qrrelations} follow from \eqref{eq:normcorrd} and
\begin{align*}
q_\leq(A) &= \max\nolimits_{X  \in \bbGamma^n_1} \tr( AX), &
r_\leq(A) &= \max\nolimits_{X  \in \bbGamma^n} \tr( AX), \\
q_\eq(A) &= \max\nolimits_{X  \in \mathbb{G}^n_1} \tr( AX), &
r_\eq(A) &= \max\nolimits_{X  \in \mathbb{G}^n} \tr( AX).
\end{align*}

Let $a_{ii} \ge 0$, $i =1,\dots,n$. We will prove that $q_\eq(A) = q_\leq(A)$; similar arguments will apply to $r_\eq(A) = r_\leq(A)$; and the remaining two equalities in \eqref{eq:nonnegdiag} then follow from \eqref{eq:qrrelations}.  Let  $\max_{x \in \mathbb{D}^n} x^*\! Ax$ be attained at $x=(x_1,\dots,x_n) \in \mathbb{D}^n$, with a maximal number of coordinates $x_i\in \mathbb{T}$.  We claim that all $x_1,\dots,x_n \in \mathbb{T}$, i.e., $x \in \mathbb{T}^n$. Suppose to the contrary that $\lvert x_i \vert < 1$ for some $i$.

\underline{\textsc{Case I}: $\Bbbk=\mathbb{R}$.}\; Fixing all other coordinates except $x_i$, the quadratic function $x^\tp\! A x = a_{ii} x_i^2 + \cdots$ is convex in the variable $x_i$ as $a_{ii} \ge 0$. Thus its maximum is attained with $x_i\in\mathbb{T} =\{-1,1\}$, a contradiction.

\underline{\textsc{Case II}: $\Bbbk=\mathbb{C}$.}\; Fixing all other coordinates except $x_i$ and writing $x_i=u_i+\mathrm{i}v_i$ with $u_i^2 + v_i^2 < 1$, the quadratic function $x^*\! A x = a_{ii}(u_i^2 + v_i^2) + \cdots$ is convex in $(u_i,v_i)$ as $a_{ii} \ge 0$. Thus its maximum is attained with $x_i\in \mathbb{T} =  \{e^{\mathrm{i}\vartheta} \in\mathbb{C} : \vartheta\in [0, 2\pi)\}$, a contradiction.
\end{proof}

We will next deduce an analogue of Theorem~\ref{thm:SGIqr} with sum in place of max. For any $A \in \mathbb{S}^n$, we define the \emph{stretch} of $A$ as
\[
\str(A) \coloneqq \Bigl( \max\nolimits_{x \in \mathbb{T}^n} x^*\! Ax \Bigr) - \Bigl( \min\nolimits_{x \in \mathbb{T}^n} x^*\! Ax \Bigr) = q_\eq(A)+q_\eq(-A),
\]
and the \emph{spread} of $A$ as
\[
\spr(A) \coloneqq \Bigl( \max\nolimits_{X  \in \mathbb{G}^n} \tr( AX)\Bigr) - \Bigl( \min\nolimits_{X  \in \mathbb{G}^n} \tr( AX)\Bigr) = r_\eq(A)+r_\eq(-A).
\]
The notion of stretch appeared in  \cite{Charikar} but was not given a name. We next show that the stretch, despite being the difference of two NP-hard quantities, can be approximated up to a universal constant factor $K_\gamma$ by the spread, which is polynomial-time computable to arbitrary precision.
\begin{theorem}
Let  $A \in \mathbb{S}^n$. Then
\[
\str(A) \le \spr(A) \le K_\gamma \str(A).
\]
\end{theorem}
\begin{proof}
The first inequality is obvious. We will prove the second one.
Let $\alpha\in\mathbb{R}$. Clearly,
\begin{align*}
q_\eq(A+\alpha I)&=q_\eq(A)+n\alpha, &q_\eq(-A-\alpha I)&=q_\eq(-A)-n\alpha,
&\spr(A+\alpha I)&=\spr(A),\\
r_\eq(A+\alpha I)&=r_\eq(A)+n\alpha, &r_\eq(-A-\alpha I)&=r_\eq(-A)-n\alpha,
&\str(A+\alpha I)&=\str(A).
\end{align*}
Now set $\alpha= \bigl(q_\eq(-A)-q_\eq(A)\bigr)/2n$.  Then $q_\eq(A+\alpha I)=q_\eq(-A-\alpha I) =\lVert A+\alpha I \rVert_\theta$ by \eqref{eq:qrrelations}.  By \eqref{eq:qrrelations} again and the symmetric Grothendieck inequality \eqref{eq:SGIs},
\begin{align*}
\spr(A) =\spr(A+\alpha I) &\le  r_\eq(A+\alpha I) + r_\eq(-A - \alpha I) \le 2\lVert A+\alpha I\rVert_\gamma\\
&\le 2K_\gamma \lVert A+\alpha I\rVert_\theta = K_\gamma \str(A+\alpha I)= K_\gamma \str(A),
\end{align*}
as required.
\end{proof}

\section*{Acknowledgment}
The work of the first author is partially supported by the Simons Collaboration Grant for Mathematicians.
The work of the second author is partially supported by NSF IIS 1546413 and the Eckhardt Faculty Fund.

\bibliographystyle{abbrv}

\appendix

\section{Absolute value of inner product of random unit vectors}\label{appendix}

When establishing that the sharpness of the Nesterov $\pi/2$-Theorem and its complex analogue in Theorem~\ref{FRNineq}, we needed the first and second moments of the absolute value of inner product of random vectors over $\mathbb{R}$ and $\mathbb{C}$. To avoid distracting our readers, we have deferred these materials to this appendix. In the following,  $\Gamma$ and $\Beta$ will denote the Gamma and Beta functions.

We will write $\Sph^n(\Bbbk) \coloneqq \{x \in \Bbbk^{n+1} : \lVert x \rVert = 1\}$ for the $n$-sphere over $\Bbbk = \mathbb{R}$ or $\mathbb{C}$. We begin our discussion with $\mathbb{R}$.
In terms of spherical coordinates on $\mathbb{R}^{n+1}$ given by
$(r,\vartheta_1,\dots,\vartheta_n)$,  the surface area element of the real $n$-sphere $\Sph^n(\mathbb{R})$ is
\[
d\sigma_n^\mathbb{R} \coloneqq \sin^{n-1}(\vartheta_1)\sin^{n-2}(\vartheta_2)\cdots \sin(\vartheta_{n-1})\, d\vartheta_1 \, d\vartheta_2\cdots d\vartheta_n
\]
and the volume element of $\mathbb{R}^{n+1}$ is $r^n\, dr\,d\sigma_n^\mathbb{R}$. Recall that the surface area of $\Sph^n(\mathbb{R})$ is $2\pi^{(n+1)/2}/\Gamma((n+1)/2$. The \emph{spherical measure} on $\Sph^n(\mathbb{R})$ is defined as
\[
d\omega_n^\mathbb{R} \coloneqq \frac{\Gamma((n+1)/2)}{2\pi^{(n+1)/2}}d\sigma_n^\mathbb{R},
\]
which is clearly a probability measure on $\Sph^n(\mathbb{R})$. The corresponding probability distribution, called the \emph{uniform distribution} on $\Sph^n(\mathbb{R})$, is invariant under the action of orthogonal matrices.

When speaking of the above notions over $\mathbb{C}$, we simply identify $\mathbb{C}^n$ with $\mathbb{R}^{2n}$ and thus
\[
\Sph^n(\mathbb{C}) = \Sph^{2n+1}(\mathbb{R}), \qquad d\sigma_n^\mathbb{C} = d\sigma_{2n+1}^\mathbb{R},  \qquad d\omega_n^\mathbb{C} = d\omega_{2n+1}^\mathbb{R}.
\]

In the following, $\mathbb{D}$ will denote either the real unit disk $[-1,1]$ or the complex unit disk $\bigl\{\rho e^{\mathrm{i}\phi} \in \mathbb{C} : \rho \in [0,1],\; \phi \in [0,2\pi)\bigr\}$ as appropriate. The  cases $\alpha =2$ and a less precise  version of $\alpha =1$ of Lemma~\ref{lem:expectations} have appeared as \cite[Equations~4.5 and 4.6]{Alon3} for $\mathbb{R}$ and \cite[Lemma~3]{So} for $\mathbb{C}$. The version below holds for any $\alpha \in (0,\infty)$, not necessarily an integer.
\begin{lemma}\label{lem:expectations}
Let $U_n,V_n$ be independent random vectors uniformly distributed on the $n$-sphere $\Sph^n(\Bbbk) = \{x \in\Bbbk^{n+1} : \lVert x \rVert =1\}$.  Let $T_n: \Sph^n(\Bbbk) \times \Sph^n(\Bbbk)  \to \mathbb{D}$ be the random variable defined by the inner product $T_n = \langle U_n, V_n\rangle $.  Then the probability density function of $T_n$ is given by
\begin{equation}\label{eq:pdf}
f_{T_n}(t)=\begin{cases}
\dfrac{1}{\sqrt{\pi}}\dfrac{\Gamma\bigl((n+1)/2\bigr)}{\Gamma(n/2)}(1-t^2)^{(n-2)/2} & \Bbbk = \mathbb{R},\\[4ex]
\dfrac{n}{\pi} (1-\lvert t \rvert^2)^{n-1} &\Bbbk = \mathbb{C},
\end{cases}
\end{equation}
for all $t \in \mathbb{D}$. 
For any $\alpha > 0$,
\begin{equation}\label{eq:moments}
\mathbb{E}(\lvert \langle U_n,V_n\rangle\rvert^\alpha)=
\begin{cases}
\dfrac{1}{\sqrt{\pi}}\dfrac{\Gamma\bigl((n+1)/2\bigr)\Gamma\bigl((\alpha+1)/2\bigr)}{\Gamma\bigl((n+\alpha+1)/2\bigr)} &\Bbbk = \mathbb{R},\\[4ex]
\dfrac{\Gamma(\alpha/2+1)\Gamma(n+1)}{\Gamma(n+\alpha/2+1)} &\Bbbk = \mathbb{C}.
\end{cases}
\end{equation}
\end{lemma}
\begin{proof}
For $\Bbbk = \mathbb{R}$,  \eqref{eq:pdf} may be found in \cite{Cho}.
It remains to prove \eqref{eq:pdf} for $\Bbbk = \mathbb{C}$. We identify $\mathbb{C}^{n+1} = \mathbb{R}^{2n+2}$ and introduce the following coordinates for $x = (x_1,\dots,x_{2n+2}) \in \mathbb{R}^{2n+2}$:
\begin{align*}
x_1&=\sqrt{r^2-\rho^2}\cos (\vartheta_1), &x_{2n-1}&=\sqrt{r^2-\rho^2}\sin (\vartheta_1)\cdots \sin (\vartheta_{2n-2})\cos (\vartheta_{2n-1}),\\
x_2&=\sqrt{r^2-\rho^2}\sin (\vartheta_1)\cos (\vartheta_2), &x_{2n}&=\sqrt{r^2-\rho^2}\sin (\vartheta_1)\cdots \sin (\vartheta_{2n-2})\sin(\vartheta_{2n-1}),\\
x_3&=\sqrt{r^2-\rho^2}\sin (\vartheta_1)\sin (\vartheta_2)\cos (\vartheta_3), &x_{2n+1}&=\rho\cos \phi,\\
&\qquad\qquad\qquad\vdots & x_{2n+2}&=\rho\sin \phi,
\end{align*}
where $r=\lVert x\rVert$, $\rho \in [0,r]$, $\vartheta_1,\dots,\vartheta_{2n-1} \in [0,\pi]$, and $\vartheta_{2n-1},\phi\in [0, 2\pi)$. This coordinate system combines the standard spherical coordinates on $\mathbb{R}^{2n}$ and the Hopf coordinates on $\Sph^3 \subseteq \mathbb{C}^2$.

The Jacobian $J=\partial (x_1,\dots,x_{2n+2})/\partial (r,\vartheta_1,\dots,\vartheta_{2n-1}, \rho, \phi)$ has determinant
\[
\det J=r(r^2-\rho^2)^{n-1}\rho \sin^{2n-2}(\vartheta_1)\sin^{2n-3}(\vartheta_2)\cdots \sin(\vartheta_{2n-2})\, d\vartheta_1\, d\vartheta_2\cdots d\vartheta_{2n-1};
\]
and since $r=1$ on $\Sph^n(\mathbb{C})=\Sph^{2n+1}(\mathbb{R})$, we get
\[
d\sigma_n^\mathbb{C}=(1-\rho^2)^{n-1}\rho\, d\rho\sin^{2n-2}(\vartheta_1)\sin^{2n-3}(\vartheta_2)\cdots \sin(\vartheta_{2n-2})\, d\vartheta_1 \, d\vartheta_2\cdots d\vartheta_{2n-1}\, d\phi.
\]
Integrating out $\vartheta_1,\dots,\vartheta_{2n-1}$ and dividing the result by the surface area of $\Sph^{2n+1}(\mathbb{R})$ yields $(n/\pi) (1-\rho^2)^{n-1}\rho\,d\rho\,d\phi$ as the surface element of $\mathbb{D}=\bigl\{\rho e^{\mathrm{i}\phi} \in \mathbb{C} : \rho \in [0,1],\; \phi \in [0,2\pi)\bigr\}$, and thus $(n/\pi) (1-\rho^2)^{n-1}$ is the expression for $f_{T_n}$ over $\mathbb{C}$ in polar coordinates.

Over $\mathbb{R}$, \eqref{eq:moments} follows from
\begin{align*}
\mathbb{E}(\lvert\langle U_n, V_n \rangle\rvert^\alpha) & =\frac{1}{\sqrt{\pi}}\frac{\Gamma\bigl((n+1)/2\bigr)}{\Gamma(n/2)}\int_{-1}^1 \lvert t \rvert^\alpha(1-t^2)^{(n-2)/2}\,dt\\
&=\frac{1}{\sqrt{\pi}}\frac{\Gamma\bigl((n+1)/2\bigr)}{\Gamma(n/2)}\int_0^1 s^{(\alpha-1)/2}(1-s)^{(n-2)/2}\,ds\\
&=\frac{1}{\sqrt{\pi}}\frac{\Gamma\bigl((n+1)/2\bigr)}{\Gamma(n/2)}\Beta\Bigl(\frac{\alpha+1}{2}, \frac{n}{2}\Bigr)
=\frac{1}{\sqrt{\pi}}\frac{\Gamma\bigl((n+1)/2\bigr)\Gamma\bigl((\alpha+1)/2\bigr)}{\Gamma\bigl((n+\alpha+1)/2\bigr)};
\end{align*}
and over $\mathbb{C}$, it follows from
\begin{align*}
\mathbb{E}(\lvert \langle U_n,V_n\rangle\rvert^\alpha) &=\frac{n}{\pi}\int_{0}^1\int_0^{2\pi}r^\alpha(1-r^2)^{n-1} \, rdrd\phi=n\int_0^1 s^{\alpha/2}(1-s)^{n-1}\, ds\\
&=n\Beta\Bigl(\frac{\alpha}{2}+1, n\Bigr)=n\frac{\Gamma(\alpha/2+1)\Gamma(n)}{\Gamma(n+\alpha/2+1)}. \qedhere
\end{align*}
\end{proof}
In their proof that \eqref{RNineq} is sharp, Alon and Naor \cite{Alon3} used the fact that over $\mathbb{R}$, the random variable  $\sqrt{n} \cdot T_n$ converges in distribution to the standard real Gaussian  $Z_n^\mathbb{R}$. While we did not need this in our proof of the sharpness of \eqref{RNineq} and \eqref{CRNineq}, we would like to point out here that this asymptotic normality follows immediately from \eqref{eq:pdf} for both $\mathbb{R}$ and $\mathbb{C}$. Over $\mathbb{R}$, direct calculation shows
\[
\mathbb{P}(\sqrt{n} \cdot T_n\le r)=
\frac{1}{\sqrt{2\pi}}\frac{\Gamma(n/2+1/2)}{\Gamma(n/2)\sqrt{n/2}}\int_{-\infty}^r \biggl(1-\frac{s^2}{n}\biggr)^{\!\! n/2-1} ds \to
\frac{1}{\sqrt{2\pi}}\int_{-\infty}^r e^{-s^2/2}\, ds =\mathbb{P}( Z_n^\mathbb{R} \le r) 
\]
as $n \to \infty$. Over $\mathbb{C}$,  observe that $r \sqrt{n} \cdot T_n$ has the same distribution as $\sqrt{n} \cdot T_n$ for any $\lvert r \rvert=1$, so it  suffices to show that $\lvert \sqrt{n} \cdot T_n \rvert$ converges in distribution to $\lvert Z_n^\mathbb{C} \rvert$:
\begin{align*}
\mathbb{P}(\lvert \sqrt{n} \cdot T_n\rvert\le r)&=\mathbb{P}\Bigl(\lvert T_n \rvert \le \frac{r}{\sqrt{n}}\Bigr)
=\frac{n}{\pi}\displaystyle\int_0^{\frac{r}{\sqrt{n}}} \int_0^{2\pi}(1-\rho^2)^{n-1}\rho\, d\rho\,d\theta\\
&=1-\biggl(1-\frac{r^2}{n}\biggr)^{\!\! n}\to  1-e^{-r^2}=\mathbb{P}(\lvert Z_n^\mathbb{C} \rvert\le r)
\end{align*}
as $n \to \infty$, for any $r\le \sqrt{n}$. Our proof of sharpness in Theorem~\ref{FRNineq} will however rely on the following.
\begin{lemma}\label{unformconvlem}
Let $U_n, V_n$ be independent random vectors uniformly distributed on the $n$-sphere $\Sph^n(\Bbbk)= \{x\in\Bbbk^{n+1} : \lVert x \rVert =1\}$. 
Let $\alpha \ge 1$ and $0 < \varepsilon < 1$. Then there exist $m \in \mathbb{N}$ and $x_1,\dots,x_m\in \Sph^n(\Bbbk)$ such that 
\begin{equation}\label{basinmunuy}
\begin{aligned}
\mathbb{E}( \vert \langle U_n, V_n \rangle\vert^\alpha) -(2\alpha+1)\varepsilon &\le\min_{v\in \Sph^n(\Bbbk)} \biggl[\frac{1}{m}\sum_{i=1}^m \lvert \langle x_i,v\rangle\rvert^\alpha\biggr]\\
&\le \max_{v\in \Sph^n(\Bbbk)} \biggl[\frac{1}{m}\sum_{i=1}^m \lvert \langle x_i,v\rangle\rvert^\alpha\biggr]
\le \mathbb{E}( \vert \langle U_n, V_n \rangle\vert^\alpha) +(2\alpha+1)\varepsilon.
\end{aligned}
\end{equation}
\end{lemma}
\begin{proof}
To avoid clutter, we write $\Sph^n = \Sph^n(\Bbbk)$ and $d\omega_n = d\omega_n^\Bbbk$ as the same proof works for $\Bbbk = \mathbb{R}$ and $\mathbb{C}$ alike.
Let $\Sph^n$  be covered by closed $\varepsilon$-balls  $B_1,\dots, B_p\subseteq \Bbbk^{n+1}$, where none of the balls are redundant, i.e., $\Sph^n \subseteq  B_1 \cup \dots \cup B_p$ but  $\Sph^n \nsubseteq (B_1 \cup \dots \cup B_p)\setminus B_k$ for any $k$.  Let
\[
F_k \coloneqq \Sph^n \cap B_k \setminus \Bigl(\bigcup\nolimits_{j =1}^{k-1} B_j \Bigr), \qquad k =1,\dots,p.
\]
Then $F_1,\dots,F_p$ are disjoint Borel sets that partition $\Sph^n$.  We may assume that
\[
0 < \omega_n(F_1) \le \omega_n(F_2) \le\dots\le \omega_n(F_p)
\]
by choosing $B_1,\dots,B_p$ appropriately. Let $t_k \coloneqq \sum_{j=1}^k \omega_n(F_j)$, $k=1,\dots,p$, and  $t_0 \coloneqq 0$. Note that $t_p =\sum_{j=1}^p \omega_n(F_j)= \omega_n(\Sph^n) =1$ and so $0=t_0 <t_1<\dots <t_p=1$ is a partition of the interval $[0,1]$ with $\omega_n(F_k)=t_k-t_{k-1}$, $k= 1,\dots,p$.

Let $m \in \mathbb{N}$ be such that $1/m<\varepsilon \omega_n(F_1)$.   Let $m_k$ be the number of points  $j/m \in (t_{k-1}, t_k]$, $j =1,\dots,m$.  Then $(m_1 + \dots + m_k)/m\le t_k$, $k= 1,\dots,p$. Hence $\lvert \omega_n(F_k)-m_k/m\rvert \le 1/m\le \varepsilon \omega_n(F_1)\le \varepsilon \omega_n(F_k)$. Let $q_k \coloneqq m_1+\dots+ m_k$, $k =1,\dots,p$, and $q_0 \coloneqq 0$. Now choose $m$ distinct points $x_1, \dots, x_m \in \Sph^n$ by choosing any $m_k$ distinct points $x_{q_{k-1}+1},\dots, x_{q_k} \in F_k$, $k =1,\dots,p$. 

For a fixed $k\in \mathbb{N}$ and $u\in \Sph^n$, the mean value theorem yields the existence of $v_k\in \overline{F}_k$ such that $\int_{v\in F_k} \lvert \langle u, v \rangle\rvert^\alpha\,d\omega_n= \omega_n(F_k)\lvert \langle u, v_k \rangle\rvert^\alpha$.  As $x_{q_{k-1}+1},\dots, x_{q_k}\in F_k$, we have $\lVert v_k- x_j\rVert \le 2\varepsilon$ for all $j=q_{k-1}+1,\dots,q_k$.  Since $\lvert x^\alpha-y^\alpha \rvert \le \alpha\lvert x-y \rvert$ for any $x,y\in [0,1]$ and $\alpha \ge 1$, we see that
\begin{align*}
\bigl\lvert \lvert \langle u, v \rangle\rvert^\alpha - \lvert \langle u, w \rangle\rvert^\alpha\bigr\rvert &\le \alpha\bigl\lvert \lvert \langle u, v \rangle\rvert - \lvert \langle u, w \rangle\rvert\bigr\rvert  
\le \alpha \lvert \langle u, v \rangle - \langle u, w \rangle\rvert \le \alpha\lVert v-w\rVert
\end{align*}
for any $v,w\in \Sph^n$ and thus  $\bigl\lvert \lvert \langle u, v_k \rangle\rvert^\alpha - \lvert \langle u, x_j \rangle\rvert^\alpha \bigr\rvert \le 2\alpha\varepsilon$ for $j=q_{k-1}+1,\dots,q_k$.
Now since
\[
 \int_{v\in F_k}\lvert \langle u, v \rangle\rvert^\alpha \,d\omega_n= \omega_n(F_k) \cdot \lvert \langle u, v_k \rangle\rvert^\alpha =  \frac{m_k}{m}  \lvert \langle u, v_k \rangle\rvert^\alpha + \Bigl(\omega_n(F_k)- \frac{m_k}{m} \Bigr)\lvert \langle u, v_k \rangle\rvert^\alpha,
\]
we have
\begin{align*}
\biggl\lvert \int_{v\in F_k}\lvert \langle u, v \rangle\rvert^\alpha \,d\omega_n&-\frac{1}{m}\sum_{j=q_{k-1}+1}^{q_k} \lvert \langle u, v_j \rangle\rvert^\alpha \biggr\rvert\\
&\le \Bigl\lvert \omega_n(F_k)-\frac{m_k}{m}\Bigr\rvert \cdot \lvert \langle u, v_j \rangle\rvert^\alpha
+\frac{1}{m}\sum_{j=q_{k-1}+1}^{q_k} \bigl\lvert \lvert \langle u, v_k \rangle\rvert^\alpha-\lvert \langle u, x_j \rangle\rvert^\alpha \bigr\rvert\\
&\le \varepsilon \omega_n(F_k)+2\alpha\varepsilon \frac{m_k}{m},
\end{align*}
and so
\begin{align*}
\Biggl\lvert \int_{v\in \Sph^n} \lvert \langle u, v \rangle\rvert^\alpha \,d\omega_n -\frac{1}{m}\sum_{j=1}^m \lvert \langle u, x_j \rangle\rvert^\alpha \Biggr\rvert
&=\Biggl\lvert \sum_{k=1}^p\biggl(\int_{v\in F_k}\lvert \langle u, v \rangle\rvert^\alpha \,d\omega_n - \frac{1}{m}\sum_{j=q_{k-1}+1}^{q_k} \lvert \langle u, x_j \rangle\rvert^\alpha\biggr)\Biggr\rvert \\
&\le\sum_{k=1}^p \Bigl( \varepsilon \omega_n(F_k)+2\alpha\varepsilon \frac{m_k}{m}\Bigr)=(2\alpha+1)\varepsilon.
\end{align*}
Note that the orthogonal/unitary invariance of the spherical measure $\omega_n$ implies that the conditional expectation $\mathbb{E}(|\langle U_n, V_n\rangle|^\alpha \mid U_n=u)$ does not depend on the value $u$ and so
\[
\mathbb{E}( \lvert \langle U_n, V_n \rangle\rvert^\alpha) = \mathbb{E}(|\langle U_n,V_n \rangle |^\alpha  \mid U_n=u) = \int_{v\in \Sph^n} \lvert \langle u, v \rangle\rvert^\alpha \,d\omega_n,
\]
giving us \eqref{basinmunuy}.
\end{proof}
For $m$ sufficiently big, we expect $m$ uniformly distributed points $x_1,\dots, x_m \in \Sph^n$ and a partition $\Sph^n = F_1\cup \dots \cup F_p$ to approximately satisfy $m_k/m\approx \omega_n(F_k)$, where $m_k$ is the number of $x_i$'s falling on $F_k$.  Our construction above relies on a more precise version of this observation.

\end{document}